\documentclass[11pt]{article}

\usepackage{epsfig}
\usepackage{amssymb, latexsym}
\usepackage{amscd}
\usepackage[all]{xy}
\usepackage{epsf}
\usepackage[mathscr]{eucal}
\usepackage{tikz}
\usepackage{tikz-cd}[/tikz/column sep=1.1em]
\usetikzlibrary{positioning, fit, calc}

\DeclareFontFamily{U}{solomos}{}
\DeclareErrorFont{U}{solomos}{m}{n}{10}
\DeclareFontShape{U}{solomos}{m}{n}{
  <-> s*[1.1]  gsolomos8r
}{}

\usepackage{amsmath,amsfonts,amscd,amssymb,amsthm}

\usepackage[titletoc]{appendix}
\usepackage[sc]{titlesec}

\long\def\comment#1\endcomment{}

\theoremstyle{plain}

\newtheorem{theorem}{\sc Theorem}[section]
\newtheorem{lemma}[theorem]{\sc Lemma}

\newtheorem{prop}[theorem]{\sc Proposition}
\newtheorem{coroll}[theorem]{\sc Corollary}
\newtheorem{klemma}[theorem]{\sc Key-lemma}

\theoremstyle{plain}

\newtheorem{defn}[theorem]{\sc Definition}

\theoremstyle{exercise}
\newtheorem{remark}[theorem]{\sc Remark}

\newtheorem{example}[theorem]{\sc Example}

\tikzset{Rightarrow/.style={double equal sign distance,>={Implies},->},
triple/.style={-,preaction={draw,Rightarrow}},
quadruple/.style={preaction={draw,Rightarrow,shorten >=0pt},shorten >=1pt,-,double,double
distance=0.2pt}}
\comment
\begin{tikzcd}
    a \arrow[r,triple] &    b  \arrow[r,quadruple] 
    & c\\
\end{tikzcd}
\endcomment

\comment
The length of the arrows is determined by the column sep of the cd matrix. If you use \begin{tikzcd}[/tikz/column sep=1.1em]     a \arrow[r,triple] &    b  \arrow[r,quadruple]      & c\\ \end{tikzcd} the arrows will become shorter.
\endcomment

\comment
\begin{tikzcd}
  A \arrow[r, bend left, "f", ""{name=U,inner sep=1pt,below}]
  \arrow[r, bend right, "g"{below}, ""{name=D,inner sep=1pt}]
  & B
  \arrow[Rightarrow, from=U, to=D, "\alpha"]
\end{tikzcd}
\endcomment

\makeatletter \@addtoreset{equation}{section} \makeatother

\def\eqref#1{\thetag{\ref{#1}}}

\let\latexref=\ref
\def\ref#1{{\normalfont{\latexref{#1}}}}

\newcommand{\ldot}{{\:\raisebox{2,3pt}{\text{\circle*{1.5}}}}}
\newcommand{\udot}{{\:\raisebox{3pt}{\text{\circle*{1.5}}}}}
\newcommand{\mb}{{\bullet}}

\def\dlim_#1{{\displaystyle\lim_{#1}}^\hdot}

\newcommand{\C}{{\mathcal C}}

\newcommand{\id}{\operatorname{\rm id}}

\newcommand{\eqto}{\mathrel{\stackrel{\sim}{\to}}}

\newcommand{\Ob}{\mathrm{Ob}}

\newcommand{\Ab}{\mathsf{Ab}}

\newcommand{\Mod}{\mathrm{Mod}}

\newcommand{\Hom}{\mathrm{Hom}}

\newcommand{\Hoch}{\mathrm{Hoch}}

\newcommand{\op}{\mathrm{op}}

\newcommand{\dg}{\mathrm{dg}}

\newcommand{\holim}{\mathrm{holim}}
\newcommand{\hocolim}{\mathrm{hocolim}}
\newcommand{\colim}{\mathrm{colim}}
\newcommand{\Id}{\mathrm{Id}}

\newcommand{\Sets}{\mathscr{S}ets}

\newcommand{\Cat}{{\mathscr{C}at}}
\newcommand{\Top}{{\mathscr{T}op}}
\renewcommand{\top}{\mathrm{top}}

\newcommand{\Fun}{{\mathrm{Fun}}}
\newcommand{\SSets}{{\mathscr{SS}ets}}

\newcommand{\diag}{\mathrm{diag}}

\renewcommand{\k}{\Bbbk}

\newcommand{\Assoc}{\mathrm{Assoc}}

\renewcommand{\Ab}{{\mathscr{A}b}}

\newcommand{\pprime}{{\prime\prime}}

\newcommand{\Ord}{\mathrm{Ord}}

\newcommand{\Tot}{\mathrm{Tot}}

\newcommand{\Sym}{\mathrm{Sym}}
\newcommand{\Disk}{\mathbf{Disk}}
\newcommand{\GDisk}{\mathbf{GDisk}}

\renewcommand{\min}{\mathrm{min}}
\renewcommand{\max}{\mathrm{max}}

\newcommand{\Op}{\mathrm{Op}}

\newcommand{\Tree}{\mathbf{Tree}}
\newcommand{\Out}{\mathrm{Out}}
\newcommand{\Symm}{\mathrm{Symm}}
\newcommand{\Des}{\mathrm{Des}}

\newcommand{\out}{\mathrm{out}}
\newcommand{\Int}{\mathrm{Int}}
\renewcommand{\kappa}{\varkappa}

\newcommand{\boxx}{\square}
\newcommand{\ns}{\mathrm{ns}}
\newcommand{\height}{\mathrm{ht}}
\newcommand{\codim}{\mathrm{codim}}
\newcommand{\bmu}{\boldsymbol{\mu}}
\newcommand{\bsigma}{\boldsymbol{\sigma}}
\newcommand{\simeqto}{{\overset{\scriptscriptstyle\sim}{\to}}}
\newcommand{\Sing}{\mathrm{Sing}}
\newcommand{\glob}{\mathrm{glob}}
\newcommand{\Bicat}{\mathrm{Bicat}}

\newcommand{\sevafigc}[4]{\begin{figure}[h]\centerline{
 \epsfig{file=#1,width=#2,angle=#3}}
\bigskip\caption{#4}\end{figure}}

\def\wtilde#1{\widetilde{#1}\vphantom{#1}}

\title{\sc{Generalised Joyal disks, $\Theta_d$-colored $(d+1)$-operads, and higher Deligne conjecture}}

\author{\sc{Boris Shoikhet}}
\date{}

\begin{document}\maketitle
{\footnotesize
\begin{center}{\parbox{4,5in}{{\sc Abstract.}
The classical Deligne conjecture (which by now has found several different proofs) says that the Hochschild cohomological complex of a small dg category over a field $\k$ has a structure of a $C^\udot(E_2,\k)$-algebra. A similar statement for any small dg weak $d$-category is called the generalised Deligne conjecture, it says the cohomological Hochschild complex of a dg weak $d$-category (appropriately defined) admits an action of the operad $C^\udot(E_{d+1},\k)$. It has been proven in the context of $\infty$-categories in [L], though within this solution the $(d+1)$-algebra structure is visible only in $\infty$-categorical world, and can be transferred to the set-theoretical world only on the level of homology groups rather than on the level of complexes. This paper is the first in a series of papers, in which our goal is to construct a $(d+1)$-algebra providing a solution to the generalised Deligne conjecture, on the level of complexes.

This paper contains the combinatorial core of this $(d+1)$-algebra structure, for any $d\ge 1$. Namely, we construct a colored $(d+1)$-operad in sets (in the sense of M.Batanin), denoted by $\mathbf{seq}_d$. Its category of colors (= the category of unary operations) is the category $\Theta_d$ of A.Joyal [J], dual to the category of Joyal $d$-disks [J], [Be2,3]. For $d=1$, our $\Delta$-colored operad $\mathbf{seq}_1$ coincides with the operad $\mathbf{seq}$ of D.Tamarkin [T3]. 

We introduce a concept of a generalised Joyal $d$-disk, extending the category of Joyal $d$-disks so that the category of generalised $d$-disks admits an analogue of the so-called funny product of ordinary categories. Based on this concept, we define a higher analogue of the lattice path operad [BB], which is a symmetric $\Theta_d$-colored operad. ``Inside'' its $(d+1)$-desymmetrisation, we depict our $\Theta_d$-colored $(d+1)$-operad $\mathbf{seq}_d$. 

We prove that the construction indeed gives rise to a $(d+1)$-operad, and that this $\Theta_d$-colored $(d+1)$-operad $\mathbf{seq}_d$ is contractible in the dg and in the topological condensations, for any $d\ge 1$. The contractibility of a $(d+1)$-operad $\mathcal{O}$ is a key property, which, due to the Batanin symmetrisation theorem [Ba1,2], endows any its  ($d$-terminal) algebra with an $E_{d+1}$-algebra structure. 

A classical example of such algebra for $d=1$ is the Hochschild complex of a small dg category [T3].
As a new example, for $d=2$,
 the operad $\mathbf{seq}_2$ acts naturally on the 2-cocellular complex associated [PS] with a dg bicategory $C$, whose dg totalisation $\Hoch_2^\udot(C)$ is a 2-dimensional analogue of the Hochschild cohomological complex of a dg category. The contractibility of the 3-operad $\mathbf{seq}_2$ in the dg condensation gives an explicit solution to the Deligne conjecture in dimension 2. We consider this and other applications of our results in the next papers.

}}
\end{center}
}

\section*{\sc Introduction}
\subsection{\sc }
This paper is author's attempt to find explicitly higher structures on deformation complexes of a (weak) $n$-category (though no deformation complexes emerge here). This problem is called ``higher Deligne conjecture''. The name originates from the statement firstly conjectured by Deligne that for a dg algebra $A$ over $\k$ (or, more generally, for a small dg category $A$) the cohomological Hochschild complex $\Hoch^\udot(A)$ has a structure of $C_\ldot(E_2,\k)$-algebra. Nowadays this statement has many proofs, see [MS1-3], [KS], [T2], [T3] among others. It became a topic of active research after Tamarkin found [T1] a new proof of the famous Kontsevich formality theorem [K], in which the argument depended on the validity of Deligne conjecture. There is a proof of general Deligne conjecture in the $\infty$-categorical setting [L], but, to the best of our knowledge, it (currently) can not be applied to the scheme of Tamarkin's proof of formality phenomena, even for the classical case of Hochschild cochains. 

The proof given by Tamarkin in [T3] (which was inspired by the previous proof [MS1-3]) is the most closed to this paper. One of the ideas was that, as the Hochschild complex is a dg totalization of a cosimplicial complex, whose degree $n$ component is formed by the length $n$ cochains, it is conceptually right to look for a {\it colored} operad with the category of unary operations $\Delta$, acting on this cosimplicial complex. One can ``condense'' the colors by the standard functor $\Delta\to C^\udot(\k)$ to get a single-color dg operad; similarly one considers the total complex by use of the same standard functor and get the cohomological Hochschild complex of $A$. It is a general fact, shown in greater generality in [BB] by an interpretation via the Day-Street convolution, that the condensed operad acts on the totalized complex.

Another important point is that [T3] deals with 2-operads rather than with symmetric operads, introducing a $\Delta$-colored 2-operad in $\Sets$ denoted by $\mathbf{seq}$. 2-operads are the $n=2$ case of Batanin $n$-operads [Ba1,2,3]. An advantage of $n$-operads over symmetric operads is that, by their definition and their very nature, $n$-operads act on globular objects, such as the underlying $n$-globular object of a given (possibly weak) $n$-category. In this way, an action of the terminal $n$-operad action on an $n$-globular object $G$ is the same that a strict $n$-category structure on $G$. The terminal $n$-operad in (a monoidal model) category $\mathcal{M}$ has the monoidal unit in each its arity component, and when each of these components weakly equivalent to the monoidal unit, an action of such operads on $\mathcal{M}$-enriched\footnote{By $\mathcal{M}$-enriched $n$-globular object (resp., $n$-category) we mean $\mathcal{M}$-enrichment only for the cells (resp., for the morphisms) of the top degree $n$.} $n$-globular object $G$ defines a weak $\mathcal{M}$-enriched $n$-category structure on $G$. The arity of $n$-operad is no longer a natural number, but an $n$-level tree. The $n$-level trees form a category, and an $n$-operadic composition is associated with a morphism on this category. (Likewise, the composition in a non-symmetric (1-)operad is associated with a morphism $[m]\to [n]$ of 1-ordinals, which are the same that $1$-level trees). A link between $n$-operads and $E_n$-algebras is given by a remarkable symmetrization theorem of Batanin [Ba2,3]. It says that the derived symmetrisation of a contractible (pruned, $(n-1)$-terminal) $n$-operad is homotopically the operad $E_n$. Here the contractibility of an $n$-operad means that there is a weak equivalence of operads from the operad to the final operad. For example, in [T3] one considers all small dg categories over a given field, they form a 2-globular dg-enriched object $\Cat_\dg^{\glob}$, whose 0-cells are small dg categories over $\k$, 1-cells are dg functors, and 2-cells are {\it coherent natural transformation} which are, for two dg functors $F,G\colon C\to D$, the (cosimplicial version of the) Hochschild cochains $\Hoch^\udot(C, _FD_G)$ with coefficients in bimodule $_FD_G(-,=)=D(F-, G=)$. The colored 2-operad $\mathbf{seq}$ acts on $\Cat_\dg^\glob$.
When one restricts to a single dg category $A$, and only its identity endofunctor, one gets a 1-terminal globular subobject $\overline{A}$ in the 2-globular object $\Cat_\dg^{\glob}$. The same 2-operad $\mathbf{seq}$ acts on this 1-terminal subobject as well. The object $\overline{A}$ is non-trivial only at level 2, where it is the (cosimplicial) Hochschild complex $\Hoch^\udot(A)$. In such case, the symmetrization of a (contractible) 1-terminal operad $\mathbf{seq}$ acts on the $\overline{A}$, as well as its symmetrisation.  That is, by Batanin symmetrisation theorem, a symmetric operad having homotopy type $C_\ldot(E_2; \k)$ acts on the Hochschild complex $\Hoch^\udot(A)$.

The results of [BM1,2] show ``universality'' of the simplicial Tamarkin operad, they show that its different condensations give rise to more general duoidal Deligne conjecture. One considers the question ``What do $\mathcal{V}$-enriched categories form?'', where $\mathcal{V}$ is a symmetric monoidal or, more generally, a duoidal category, and a suitable system of standard simplices in $\mathcal{V}$. Then the $\mathcal{V}$-condensation of $\mathbf{seq}$ acts on the $\mathcal{V}$-enriched 2-globular set. The question of contractibility of the obtained $\mathcal{V}$-2-operad is more subtle, and should be studied separately in each case of interest.

\subsection{\sc }
One can define analogues of ``derived natural transformations'' (given by Hochschild cochains as above) in more general context of strict dg $d$-categories (here one can consider any enrichment instead of complexes). The case $d=2$ was discussed in [PS], where the ``derived modifications'' between (classical)  natural transformations $\eta,\theta\colon F\Rightarrow G\colon C\to D$ were constructed. So these are derived 3-arrows, while in the lower dimensions $k=1,2$ one considers the classical $k$-morphisms. One important observation was that the complex of such derived modifications was a dg totalization of a functor $\Theta_2\to C^\udot(\k)$ where $\Theta_2$ is the category introduced by Joyal [J], it is dual to the category of Joyal 2-disks in $\Sets$. 
This fact is  a closed cousin of the construction of 2-nerve of a strict 2-category [J, Be2], which is a functor $\Theta_2^\op\to\Sets$. 
The cochains of derived modifications are given by a sort of ``2-dimensional Hochschild cochains'' (whose input arguments are 2-chains of 2-morphisms and the output is also a 2-morphism; the classical modifications are then degree 0 cohomology of the derived ones). One expects that similar ``derived higher modifications'' in the framework of dg strict $d$-categories are given as the totalization of functors $\Theta_d\to C^\udot(\k)$. 
Then the arguments discussed above indicate that an appropriate answer to the question ``What do dg $d$-categories form''
should be given by an action of a  contractible $\Theta_d$-colored $(d+1)$-operad $\mathbf{seq}_d$, in dg condensation. Moreover,  results [BM1,2] (or rather a possible generalisation thereof) indicate that the case of weak $n$-categories can also be treated by the same operad $\mathbf{seq}_d$, via suitable ``system of $d$-cells'' by which the condensation of the operad is defined; the same system of cells is used in totalization of ``higher modifications''. Shortly, these $\Theta_d$-colored $(d+1)$-operads $\mathbf{seq}_d$ (which are pruned and $d$-terminal $(d+1)$-operads) presumably have the same universality (in the sense of loc.cit.) for higher categorical questions as the Tamarkin 2-operad $\mathbf{seq}$ has for problems admitting the duoidal interpretation.

In this paper, we provide a construction of the $\Theta_d$-colored $(d+1)$-operad $\mathbf{seq}_d$, for any $d\ge 1$. We also prove their contractibility in the topological and in the dg condensations.

The Batanin symmetrisation theorem is applied to these operads, giving $E_{d+1}$-algebra acting on the (dg enriched) $d$-categorical `` derived modifications''. It gives an explicit form of higher Deligne conjecture. In particular, one gets shifted by $-d$ $L_\infty$ algebra structure, providing higher analogs of the Gerstenhaber bracket. 

We would like to point out a link between our paper and [BD].
In [BD], a $E_3$-algebra structure on the Davydov-Yetter complex of a monoidal (dg) category is constructed, by a tricky use of $\Delta$-colored complexity 3 suboperad of the lattice path operad $\mathcal{L}$. However, the Davydov-Yetter complex is a truncated complex of the PS-complex. The DY-complex governs deformations of a monoidal category keeping the underlying 1-category fixed, whence the PS-complex governs the deformations of the entire structure.  Roughly, our complex [PS] is the total complex of a bicomplex, and the DY-complex is the kernel of the vertical differential of this bicomplex at degree 0 row, which results in naturality condition for the DY-cochains (which lacks for the PS-complex). Consequently, the DY-complex (or rather the underlying cosimplicial monoid) enjoys the property of ``2-commutativity'' in terminology of [BD], but the PS-complex literally lacks it, but the idea was that it is still ``homotopically 2-commutative''. We do not know how do define a homotopy $n$-commutative cosimplicial monoid in general, but the idea (for $n=2$ case) was that the ``two-dimensional'' brace (see the first row in Figure \ref{figureintro} below) provides a homotopy up to which the 2-commutativity holds. This idea was one of the starting points for this project.

\subsection{\sc  }
For $d=2$, Figure \ref{figureintro} visualises how the quadratic part of the $L_\infty$ structure of degree -2 looks like. This shifted $L_\infty$ structure is a part of $E_3$-algebra on 2-dimensional Hochschild cochains [PS], obtained from the dg condensation of the operad $\mathbf{seq}_2$. Here in the Figure \ref{figureintro} the light grey area represents a cochain $D_1$ and the dark-grey area represents a cochain $D_2$ (the cochains are understood as elements of the complex $\Hoch_2^\udot(C)$ for a small dg bicategory $C$, introduced in [PS], see Section 6). (The area shown in white does not mean any third cochain, it is used to schematically display the operations; in this pictorial presentation of operations we assume 2-shuffles operations, see Appendix B). At the left-hand side we display two elements $L_1$ and $L_2$ (depending on $D_1$ and $D_2$), which we may also consider as elements in the 3-operad $\mathbf{seq}_2$, corresponded to the 3-graph $T^3_0$ with two leaves 1 and 2 such that $1<_02$. 
We define an operation $D_1\{\{D_2\}\}$ by taking the sum $L_1+L_2$ of the two operations in the left-hand side. Then the picture shows that $[d, L_1+L_2]=R_{11}-R_{22}$, in notations of Figure \ref{figureintro}.
After application of the skew-symmetrization by $D_1$ and $D_2$, $R_{11}-R_{22}$ vanishes, and it gives a closed skew-symmetric operation of degree -2 on $\Hoch_2^\udot(C)$. It is the quadratic part of the corresponding $L_\infty$ structure of degree -2 obtained from the action of $\mathbf{seq}_2$ on $\Hoch_2^\udot(C)$ proven in Section 6 (see Theorem 3 below in the Introduction). 
We call the operation $D_1\{\{D_2\}\}$ {\it the 2-dimensional brace}, as its ``principal term'' $L_1$ is given by insertion of a 2-dimensional cochain inside another. However, the (skew-symmetrised) term $L_1$ alone does not descents to a closed operation on cohomology, and the term $L_2$ can be consider as a ``correction'' to produce a closed operation.

\sevafigc{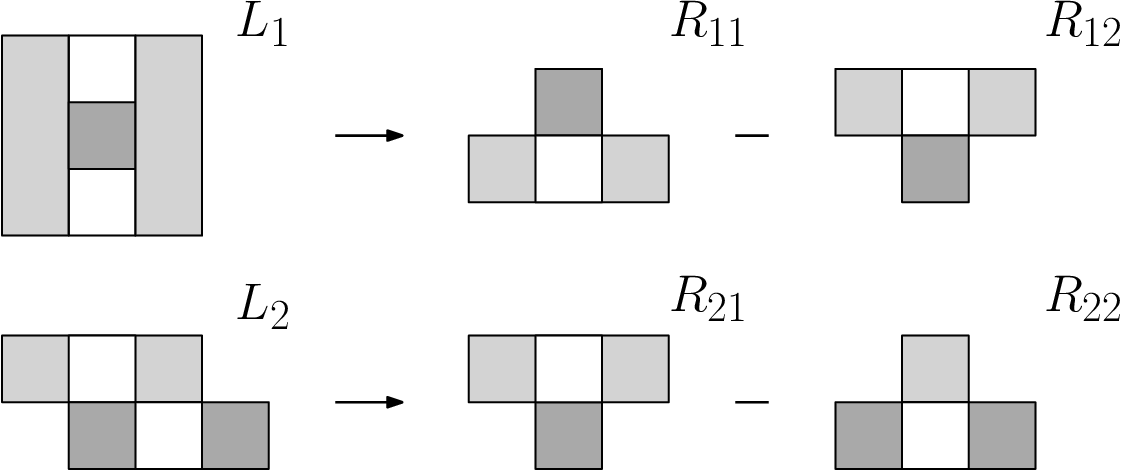}{100mm}{0}{The expression $L_1+ L_2$  (with arguments $D_1$ and $D_2$)  is a 2-dimensional analogue of the Getzler-Jones brace operation $D_1\{D_2\}$ [GJ],  denote it by  $D_1\{\{D_2\}\}$. The skew-symmetrization of $D_1\{\{D_2\}\}$ by $D_1$ and $D_2$ gives a closed element. Indeed, in $L_1+L_2$ the terms  $R_{12}$ and $R_{21}$ cancel each other, and the remaining terms $R_{11}$ and $R_{22}$ cancel after the skew-symmetrization. Here the arrows from the l.h.s. to the r.h.s. represent the differential. \label{figureintro}}

What is shown in Figure \ref{figureintro} is a two-dimensional analogue of the well-known brace formula of Getzler-Jones for Hochschild cochains:
$$
[d, D_1\{D_2\}]=D_1\cup D_2\mp D_2\cup D_1
$$
and after the skew-symmetrization by $D_1$ and $D_2$ it gives a closed operation, which is the Gerstenhaber bracket.

The entire $L_\infty$ structure on $\Hoch^\udot_2(C)$ contains higher non-trivial Taylor components (it is not just reduced to a Lie bracket). These components can be also expressed explicitly in a similar explicit way. That is, the Batanin symmetrisation theorem [Ba1,2] applied to the action of the (dg condensed) 3-operad $||\mathbf{seq}_2||_\dg$ on the 2-Hochschild complex $\Hoch^\udot_2(C)$ of a small dg bicategory endows the latter with a combinatorially highly non-trivial homotopy $E_3$-structure, and the quadratic part of the corresponding $L_\infty$ structure, shown in Figure \ref{figureintro}, is just the simplest instance of it.

\subsection{\sc  }
Here we outline the methods and the results of this paper. 

A new concept introduced in the paper is the one of {\it a generalised Joyal $d$-disk}. The motivation came from an attempt to generalise the lattice path operad $\mathcal{L}$ of Batanin-Berger [BB], which is a colored symmetric operad  in $\Sets$ with the category of unary operations $\Delta$. In the case of $\Delta$ the approach of lattice paths is equivalent to the approach of shuffles adopted in [T3], though even for this case the action of the unary operations from $\Delta$ become more natural and direct in the lattice path description. For dimensions $d\ge 2$, that is, for analogous operads with the category of unary operations $\Theta_d$, $d\ge 2$, the two approaches (via ``higher shuffles'' and via higher lattice paths) are {\it not} equivalent, and the former one does not give rise fully to a desired $\Theta_d$-functoriality. 

The definition of the classical lattice path operad (which we recall in in Section \ref{sectionlp1}) uses the funny product of ordinary categories, and an interpretation of the (Joyal dual) ordinals as categories (generated by linear posets) with fixed end-objects. It was quite clear that for higher $d$ one has to consider Joyal $d$-disks in $\Sets$ (the dual category to $\Theta_d$ [J]), but we need a ``funny-like'' product on the category of $d$-disks (which it lacks) to mimic the definition of the lattice path operad. To fix it, we introduce generalised $d$-disks, in which, roughly, the linear intervals in fibres are replaced by categories with two distinguished objects, called the source and the target objects, which we call {\it generalised intervals}. Then we define the higher lattice path operad $\mathcal{L}^d$ (Section \ref{sectiongenlp1}), it is a $\Theta_d$-colored symmetric operad in $\Sets$. Like for the case $d=1$, it contains a family of suboperads depending on the complexity. However, for $d>1$ one considers ``multi-complexity'', specifying the complexity for each of $d$ levels. We have ``blocks'' $\mathcal{L}^d_{(\bmu,\bsigma)}$ which are subfunctors of $\mathcal{L}^d(k)\colon (\Theta_d^\op)^{\times k}\times\Theta_d\to\Sets$, depending on $(\bmu,\bsigma)\in \mathcal{K}(k)^{\times d}$ (for a given arity $k$), the $d$-cartesian power of the arity $k$ poset $\mathcal{K}(k)$. The poset $\mathcal{K}(k)$ was introduced in [Be1], see also [BeuM] (we call it the Berger poset), it is the arity $k$ component of his complete graph operad loc.cit.
Thus $(\bmu,\bsigma)=((\mu_1,\sigma_1),\dots, (\mu_d,\sigma_d))$. The subfunctor $\mathcal{L}^d_{(\bmu,\bsigma)}$ consists of all generalised lattice paths whose $(\sigma, \mu)$-parameters at level $i$ less or equal to $(\mu_i,\sigma_i)$. These are elementary blocks by which all our suboperads of interest in $\mathcal{L}^d$ are built, in the sense they are colimit by some posets of these elementary blocks $\mathcal{L}^d_{(\bmu,\bsigma)}$. 
Our first main result is:
\vspace{1mm}

\noindent {\sc Theorem 1.} {\it For each $(\bmu,\bsigma)\in \mathcal{K}(k)^{\times d}$, the topological and the dg condensations of the functor $\mathcal{L}^d_{(\bmu,\bsigma)}\colon (\Theta_d^\op)^{\times d}\times \Theta_d\to \Sets$ is contractible.}

\vspace{1mm}

In Theorem 1, by abuse of terminology, we call ``condensation'' application of the realization by the $\Theta_d^\op$-arguments followed by the totalization by the $\Theta_d$-argument. The dg realization/totalization is understood in the category of $R$-modules, for an arbitrary ring $R$, and ``contractibility'' in this dg case means that the corresponding complex of $R$-modules is quasi-isomorphic to $R[0]$. 

Theorem 1 is stated as Theorems  \ref{propthetatop} and \ref{propthetadg}, the proofs are in Section 4.
Similar results for $d=1$ (the case of category $\Delta$) are proven in [MS3] for the topological case, and in [BB], [BBM] for the dg case. Our first idea was to generalise this approach for higher $d$. However,  the proof of $d=1$ case is based on [MS3, Prop. 12.7, Prop. 13.4] which roughly say that the topological realization  $|\mathcal{L}_{(\mu,\sigma)}[n]|$ for fixed cosimplicial argument $[n]$ is homeomorphic to 
$|\mathcal{L}_{(\mu,\sigma)}[0]|\times \Delta^n$, in the way that the two cosimplicial topological spaces are homeomorphic. This way of arguing reduces the computation of the totalization to computation of the totalization of standard cosimplicial topological space $[n]\to\Delta^n$. Following [BBM], we refer to this reduction as {\it whiskering}. As we said, it fails for $d>1$. Our method is an essential refinement of the methods of [MS3], even when $d=1$. To make the proof more accessible, we start with re-proving Theorem 1 for the case $d=1$ by our method, in Section 3. At the beginning of Section 3, the reader will find a more detailed outline of the strategy of the proof, which we apply in Section 4 for the case of general $d$. 

Recall that for $d=1$ the lattice path symmetric operad contains the complexity $\le n$ suboperads $\mathcal{L}_{\le n}$ [BB]. ``Inside'' the desymmetrisation of $\mathcal{L}_{\le n}$ there is a $\Delta$-colored $n$-operad $\mathrm{Tam}_n$, whose arity component associated with an  $n$-level tree $T$ is defined as a single block $\mathcal{L}_{\mu,\id}$, formed by all lattice paths $\omega$ whose $(i,j)$-projection has parameters $(\mu,\sigma)(\omega_{ij})\le (k-1,\id)$, where $\le$ is understood in the sense of Berger posets $\mathcal{K}(|T|)$, and $i<_{n-k+1}j$ in $T$ (we recall definitions of all concepts mentioned here in Sections 1 and 2). That is, $\mu_{ij}=k-1$ if $i<_{n-k+1}j$ in $T$, $\sigma_{ij}=\id$ if $i<j$. It follows that $\mathrm{Tam}_n(T)$ is a single block $\mathcal{L}_{(\mu,\id)}$. For $n=2$ it is the Tamarkin 2-operad $\mathrm{Tam}_2=\mathbf{seq}$. 

The idea of construction of the $(d+1)$-operad $\mathbf{seq}_d$ is to find it ``inside'' the desymmetrisation of $\mathcal{L}^d$ such that the arity components $\mathbf{seq}_d(T)$ are subfunctors $(\Theta_d^\op)^{\times |T|}\times\Theta_d\to\Sets$, where $T$ is a (pruned) $(d+1)$-level tree. The difference with $d=1$ case is that, for $d>1$, some components $\mathbf{seq}_d(T)$ are formed {\it by more than 1 blocks} $\mathcal{L}^d_{(\bmu,\bsigma)}$. Namely, $\mathbf{seq}_d(T)$ may be formed by several such blocks, being their union, and thus $\mathbf{seq}_d(T)$ is a colimit over a poset whose objects are corresponded to some such blocks. 
This phenomenon emerges at first place for $\mathbf{seq}_2(T^3_0)$, where $T^3_0$ is a pruned 3-level graph with vertices 1 and 2 such that $1<_02$.  

The two problems arise: (1) what is a requirement on the collections of such $(\bmu(T),\bsigma(T))$ which constitute $\mathbf{seq}_d(T)$ for a $(d+1)$-level tree $T$ so that all together it gives rise to a $(d+1)$-operad, and (2) how to prove the contractibility of this $(d+1)$-operad. (Recall that each single block is contractible by Theorem 1, so we need to arrange it such a way that a poset the colimit of which gives $\mathbf{seq}_d(T)$ is contractible, for any $T$). 

The first question is rather easy to answer, see Proposition \ref{proptwoleaves}. Namely, in the category of pruned $\ell$-level trees there are morphisms which are isomorphisms on the sets of leaves. They are called {\it quasi-bijections}.  When one also orders the leaves, the quasi-bijections with a given number of leaves form a poset, called {\it the Milgram poset} [BFSV], [Be1]. When the number of leaves is $k$, this poset $\mathcal{M}^\ell_k$ realises to the arity $k$ component $E_\ell(k)$ of the topological little  $\ell$-disk operad $E_\ell$ loc.cit. 
In particular, the Milgram poset $\mathcal{M}^\ell_2$, for the pruned $\ell$-level trees with 2 leaves, realises to a sphere $S^{\ell-1}$. 
Although these posets do not form an operad,  the arity components of the Berger complete graph operad in posets can be regarded as the Milgram posets ``operadic completed'' with ``limit'' elements, so that all together one gets an operad in posets. Indeed, as we know e.g. from the example of the Stasheff polyhedra operad $A_\infty$, the closeness under the operadic composition results in necessity of the limit strata, when points in the configuration space approach each other; here we witness a similar phenomenon. In this analogy, the Milgram posets (for all arity, as a sub-poset of the Berger complete graph operad) is analogous to the open top dimension stratum (such strata do not constitute and operad, but are homotopically equivalent to the components of the Stasheff operad).

It turns out, that, to get a $(d+1)$-operad, our conditions on $(\bmu,\bsigma)$ should be given only for 2-leaf $(d+1)$-trees
in a way compatible with quasi-bijections of two-leaf trees (for the precise statement, the reader is referred to Proposition \ref{proptwoleaves}). Then one canonically extends this data for $(d+1)$-level trees with two leaves to arbitrary $(d+1)$-level trees, so that it forms a (pruned, $d$-terminal) $(d+1)$-operad. At this moment, the reader may guess that there are many trivial examples of such assignment for 2-leaf trees. That is correct, but we sweep most of them out by the following requirement, motivated by what we really want from our operads: for the ``deepest'' $(d+1)$-level tree $T^{d+1}_0$ with two leaves  (so that $1<_02$), the ``leading component'' has to be 
$(\bmu,\bsigma)=(121)|(121)|(121)|\dots|(121)$ (with notations explained in Section \ref{sectiond=2}). It means that we want to have a ``$d$-dimensional brace operation'' for the ``deepest'' tree with two leaves $T^{d+1}_0$. It forces some constraints on other trees $T^{d+1}_a$ with two leaves, for which $1<_a2$, and then the compatibility condition for quasi-bijections leads, for general $d$, to a rather complicated combinatorics. We discuss in detail how it works for $d=2$ in Section \ref{sectiond=2}, and for $d=3$ in Section \ref{sectiond=3}. 

The case of general $d$ is discussed in Section \ref{sectiongen}. The Key-Lemmas \ref{lemmaconj1} and \ref{lemmaconj2} provide a combinatorial core of the constructions, to which the property of $\mathbf{seq}_d$ to be a $(d+1)$-operad, and its contractibility, are reduced. \footnote{These Key-Lemmas were stated as Conjectures 1 and 2 in v1, where they were proved for $d=1,2,3$ only. This v2 contains proofs for general $d$.}

\vspace{1mm}

\noindent {\sc Theorem 2.} {\it 
The constructions of Section 5 give rise to a $\Theta_d$-colored pruned $d$-terminal $(d+1)$-operad 
$\mathbf{seq}_d$ in $\Sets$, contractible in the topological and in the dg condensation (for the dg case, over any ring $R$). }

\vspace{1mm}

.

For $d=2$, we introduced [PS] a $\Theta_2$-colored 3-globular set $\Bicat^{\Theta}_3(\k)$, whose 0-cells are small dg bicategories over $\k$, 1-cells are dg strong bifunctors, 2-cells are dg strong natural transformations, and the 3-cells are thought of as ``derived modifications'' (we recall the construction in Section 6). The dg-totalizations $A^\udot(C,D)(F,G)(\eta,\theta)$ of the 2-cocellular complexes of 3-cells $A^{\Theta}(C,D)(F,G)(\eta,\theta)$, associated with a pair of small dg bicategories $C,D$ over $\k$, two strong dg functors $F,G\colon C\to D$, and two strong dg natural transformations, are interpreted in [PS] as complexes of ``derived modifications''  $\eta\Rrightarrow\theta$ (recall that classically a modification is a 3-arrow in the 3-category of strict 2-categories, see e.g. [Le]). 

We call the complex $A^\udot(C,C)(\Id,\Id)(\id,\id)$ the 2-Hochschild complex of a small dg bicategory $C$, and denote it by $\Hoch^\udot_2(C)$. Consider the 3-globular dg object $\Bicat_3^\udot(C)$ obtained as the dg $\Theta_2$-totalization of a $\Theta_2$-colored 3-globular object $\Bicat^{\Theta}_3(\k)$. We claim that the $\Theta_2$-colored 3-operad $\mathbf{seq}_2$ acts on $\Bicat^\Theta_3(\k)$. Therefore, this action induces an action of the dg condensation $||\mathbf{seq}_2||_\dg$ of the 3-operad $\mathbf{seq}_2$ on the dg totalization $\Bicat^\udot(\k)$. We provide this action in our next paper. 

Each dg bicategory $C$ defines a 2-terminal 3-globular subcomplex of $\Bicat_3(\k)$, whose only non-trivial part is given by $\Hoch^\udot_2(C)$. Therefore, the symmetrisation of the dg 3-operad $||\mathbf{seq}_2||_\dg$ acts on the complex $\Hoch^\udot_2(C)$. By Batanin symmetrisation theorem [Ba2,3], see also Appendix, it gives rise to an action of the chain $E_3$ operad on the same complex. 

The complex $\Hoch^\udot_2(C)$ has an almost straightforward generalisation for higher $d$, at least for the case of a strict dg $d$-category. This complex appears as the result of $\Theta_d$-totalization of a functor $X_C\colon \Theta_d\to C^\udot(\k)$ naturally associated, likewise in [PS] for $d=2$, with a strict dg $d$-category $C$. Denote this complex by $Hoch^\udot_d(C)$. We claim that the $\Theta_d$-colored $(d+1)$-operad $\mathbf{seq}_d$ acts naturally on $X_C$. Therefore, the dg condensation $||\mathbf{seq}_d||_\dg$ acts on $\Hoch^\udot_d(C)$. Then our Theorem 2 and the Batanin symmetrization theorem provide $\Hoch^\udot_d(C)$ with a $C^\udot(E_{d+1},\k)$-algebra structure. 

This questions are closely related with the following Claim, which prove will appear in the sequels papers. Let $\mathcal{O}$ be a $(d-1)$-terminal $d$-operad in a symmetric monoidal category $\mathcal{V}$, equipped with a map of $d$-operads $1\to \mathcal{O}$, where $1$ is the ``terminal'' $d$-operad, equal to the monoidal unit in any arity. Then the association $T\mapsto \mathcal{O}(T)$ gives rise to a functor $\Theta_d\to\mathcal{V}$, and the $(d+1)$-operad $\mathbf{seq}_d$ acts on this $d$-cocellular object in $\mathcal{V}$. 

For $d=1$ it is a 2-operadic version of the known claim that the ``operad of natural operations'' is isomorphic to the lattice path operad of complexity 2. Note that for $d>1$, the map from $\mathbf{seq}_d$ to the ``operad of natural operations'' for $d$-operads fails to be an isomorphism.

\subsection{\sc Organisation of the paper}
In Section 1 we recall $d$-level trees, $d$-disks, and define the category of generalised Joyal $d$-disks, as well as a monoidal category stucture on it, given by a generalisation of the funny product of ordinary categories.

In Section 2 we define the generalised lattice path operad, a symmetric $\Theta_d$-colored operad. We recall the Berger complete graph operad posets $\mathcal{K}(k)$ (where $k$ is the arity), and associate with any $(\bmu,\bsigma)\in \mathcal{K}(k)^{\times d}$ a subfunctor $\mathcal{L}^d_{(\bmu,\bsigma)}$. We state Theorems \ref{propthetatop} and \ref{propthetadg}, saying that each $\mathcal{L}_{(\bmu,\bsigma)}$ is contractible in the topological and the dg condensations (see Theorem 1 above).

The reader who agrees to take Theorems \ref{propthetatop} and \ref{propthetadg} for granted, and who is mainly interested in the construction and combinatorics of the operads $\mathbf{seq}_d$, can skip Sections 3 and 4, devoted to their proofs, at least during the first reading, and switch directly to Section 5. 

The proofs of Theorems \ref{propthetatop} and \ref{propthetadg} are rather technical, we make use of Quillen's model categories, Reedy model structures for Reedy category indexed diagrams in a monoidal model category, homotopy properties of the realization and the totalization as developed in the Appendix to [BM]. In Section 3, we re-prove the contractibility results for $d=1$ by our method. The same but technically more involved way of arguing proves Theorems \ref{propthetatop} and \ref{propthetadg} for general $d$ is employed in Section 4. 
Section 4 is the technical core of the paper.

In Section 5 we provide a construction of the operads $\mathbf{seq}_d$, for any $d\ge 1$, and prove their contractibility in the topological and in the dg condensations. Crucial combinatorial properties are stated in Key-Lemmas \ref{lemmaconj1} and \ref{lemmaconj2}, then we employ homotopy theoretical methods, such as the Quillen Theorem A and Thomason homotopy colimit theorem, to transfer these combinatorial properties to the contractibility of $\mathbf{seq}_d$. We prove here Theorem 2.

In Appendix we recall basic definitions and facts on Batanin $n$-operads.

\subsection{\sc Notations}
We denote by $\Top$ any of the nice categories of topological spaces, e.g. the category of compactly generated topological spaces. Then the category $\Top$ is closed symmetric monoidal category. Moreover, $\Top$ is a monoidal Quillen model category in the sense of [Ho, Ch.4]. 
We adopt the algebraic convention on differentials in complexes, by which all differentials have degree +1.

\subsection*{}
\subsubsection*{\sc Acknowledgements}
The author is thankful to Michael Batanin for many discussions, for his interest and encouragement, along several years.

The author is grateful to PDMI RAS and to EIMI (Saint-Petersburg, Russia) for excellent working conditions and wonderful atmosphere, which were instrumental in bringing this project from vague ideas to fruition. 

The work was partially supported by the Support grant
for International Mathematical Centres Creation and Development, issued by the Ministry of Higher Education and Science
of Russian Federation and PDMI RAS, agreement 
075-15-2025-344 on April 29, 2025.

\section{\sc Generalised Joyal disks and their funny product}\label{section1}
\subsection{\sc Joyal disks and the categories $\Theta_d$}\label{section11}
Denote by $\overline{\mathbb{G}}_d$ the following reflexive (co)globular diagram
$$
\xymatrix{
0\ar@<1ex>[r]^{t_0}\ar@<-1ex>[r]_{s_0}&1\ar@<1ex>[r]^{t_1}\ar@<-1ex>[r]_{s_1}\ar[l]&2\ar[l]&\dots&\ar@<1ex>[r]^{t_{d-2}}\ar@<-1ex>[r]_{s_{d-2}}&d-1\ar[l]\ar@<1ex>[r]^{t_{d-1}}\ar@<-1ex>[r]_{s_{d-1}}&d\ar[l]
}
$$
where the middle arrows in backward direction are $i_k\colon k+1\to k$, and the relations are
$$
s_{k+1}s_k=t_{k+1}s_k, s_{k+1}t_k=t_{k+1}t_k, i_ks_k=i_kt_k=\id_k
$$
An interval in $\Sets$ is a totally ordered finite set, denote by $\min$ and $\max$ its minimal and maximal elements, correspondingly. An interval is degenerated if $\min=\max$, in this case the interval is a single point.

\begin{defn}{\rm
A Joyal $d$-disk in a category $\Sets$ is a functor $D\colon \overline{\mathbb{G}}_d\to\Sets$ such that, for each $k\le d-1$, $x\in X_k=D(k)$, the fibre $D(i_k)^{-1}(x)$ is an interval with minimum $D(s_k)(x)$ and maximum $D(t_k)$, $X_0$ is a singleton, and for $k>0$ the set of singular points in $X_k$ is $\partial X_{k-1}$. A point $x\in X_k$ is called singular if $D(s_k)(x)=D(t_k)(x)$, or, equivalently, $D(i_k)^{-1}(x)$ is a single point. It is clear from the relations above that any point in $\partial X_{k-1}:=D(s_{k-1})(X_{k-1})\cup D(t_{k-1})(X_{k-1})\subset X_k$ is singular. A point which is not singular is called internal. 

A map $f: D\to D^\prime$ of disks in $\Sets$ is a natural transformations of functors $\overline{\mathbb{G}}_d\to\Sets$, that is, it is given by its components $f_k\colon D(k)\to D^\prime(k)$ which commute with the structure maps. 
}
\end{defn}
It follows in particular that $X_1$ is an interval. We similarly define a disk in the category $\Top$.\\

The category of $d$-disks in $\mathcal{E}$ is denoted by $\Disk_d(\mathcal{E})$, where $\mathcal{E}$ is any category where one can define 1-intervals, in particular $\mathcal{E}$ can be $\Sets$ or $\Top$.

\begin{defn}\label{defntree}
{\rm A $d$-level tree $T$ is a sequence of 1-ordinals (some of which may be empty ordinals) and their maps
\begin{equation}\label{eqntree}
t_d\xrightarrow{i_{d-1}}t_{d-1}\xrightarrow{i_{d-1}}t_{d-2}\to\dots\xrightarrow{i_0}t_0
\end{equation}
where $t_0$ is $[0]$ and the maps $\{i_k\}$ are not necessarily surjective. (When all ordinals are non-empty and the maps $\{i_k\}$ are surjective the $n$-level tree is called {\it pruned}). For two $d$-level trees $T$ and $S$, a map $\phi\colon T\to S$ is defined in a bit tricky way, not just a component-wise maps of ordinals commuting with the maps $i_s$. Namely, a map $\phi\colon T\to S$ is defined as a collection of maps $\{\phi_k\colon t_k\to s_k\}_{1\le k\le n}$ of the {\it underlying sets}, commuting with the maps $i_s$, and such that for any $a\in t_k$ the restriction of $\phi_{k+1}$ to $i_k^{-1}(a)$ preserves the order. The category of $d$-level trees is denoted by $\Tree_d$.
}
\end{defn}
The reader is referred to [Ba2], Sect.4 and [Ba3], Sect.2 for more detail on this definition. The category of $d$-trees in fundamental for the definition of a $d$-operad by Batanin. We recall $d$-operads in Appendix A. The reader is referred to [Ba1-3] for a thorough treatment. 

\vspace{1mm}

Let $T$ be a $d$-level tree. Define a disk $\overline{T}$ whose components $\overline{t}_k$ are the sets obtained from $T_k$ by adding two extra (distinct) points, called the minimum and the maximum, to any set $i_{k}^{-1}(x), x\in T_{k-1}$. It gives rise to a $d$-disk in $\Sets$ structure on $\overline{T}$. The set of internal points in $\overline{t}_k$ is $t_k$. Any $d$-disk in $\Sets$ is of the form $\overline{T}$ for a $d$-level tree in $\Sets$. The added minima and maxima are called external points. In Figure \ref{figure2} all added external points are shown by little circles. Note that any map of $d$-disks respects the minima and the maxima, but in general internal points can map to external ones. 

\sevafigc{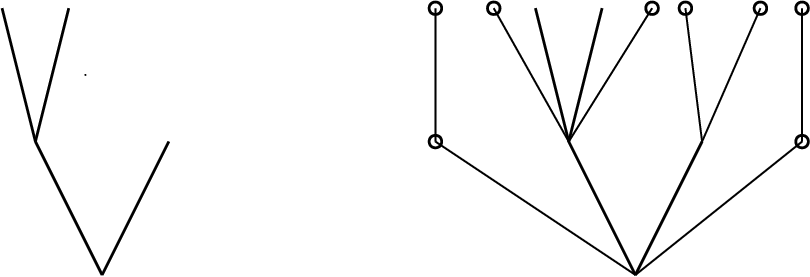}{70mm}{0}{A 2-level graph $T$ (left) and the disk $\overline{T}$ (right)\label{figure2}}

Note that the assignment $T\mapsto \overline{T}$ embeds $\Tree_d$ as a subcategory of $\Disk_d(\Sets)$ having the same objects and those morphisms which map internal points to internal (that is, only external points are mapped to external points). 

Denote $B^k=\{x\in \mathbb{R}^k| \sum x_i^2\le 1\}$ the standard topological $k$-dimensional ball, let $i_k\colon B^{k+1}\to B^k$ be the projection along the $(k+1)$-st axis, and let $s_k, t_k\colon B^k\to B^{k+1}$ be the maps $s_k\colon x\mapsto 
(x, -\sqrt{1-\sum x_i^2})$ and $t_k\colon x\mapsto (x\colon \sqrt{1-\sum x_i^2})$. We get a $d$-disk in $\Top$, called the standard topological $d$-disk. The set of internal points in $B_k$ is $\mathrm{Int} B_k$.

\vspace{1cm}

We are going to define the category $\Theta_d$ which is in fact dual to the category $\Disk_d(\Sets)$. This category was defined in [J], but the definition via the wreath product provided below is due to C.Berger [Be3]. 

For any ordinary category $\mathscr{A}$, one defines the wreath-product category $\Delta\wr \mathscr{A}$, as follows. 
Its objects are tuples $([n]; A_1,\dots,A_n)$, $[n]\in \Delta, A_i\in \mathscr{A}$. A morphism $$
\Phi\colon ([n];A_1,\dots,A_n)\to ([m];B_1,\dots,B_m)$$
is a tuple $(\phi;\{\phi_i^j\})$, where 
\begin{equation}\label{thetamorph}
\phi\colon [n]\to [m]\in \Delta,\ \  
\phi_i^j\colon A_i\to B_j\in \mathscr{A}, 1\le i\le n, \phi(i-1)+1\le j\le \phi(i) 
\end{equation}
The composition is defined naturally.

By definition 
\begin{equation}\label{thetadef}
\Theta_1=\Delta \text{  and  }\Theta_\ell=\Delta\wr\Theta_{\ell-1}
\end{equation}
In particular, $\Theta_2=\Delta\wr\Delta$. 

Recall that a $d$-globular set is a functor $\mathbb{G}_d^\op\to \Sets$, where $\mathbb{G}_d\subset \overline{\mathbb{G}}_d$ is the non-reflexive subcategory, that is, it is the subcategory having the same objects, and the morphisms are only $s_k,t_k$ but not $i_k$. A $d$-globular set is what we get from a strict $d$-category forgetting the compositions and the unit morphisms. A 1-globular set is just an oriented graph. 

The objects of $\Theta_d$ can be directly identified with $d$-level trees, as is explained in Remark \ref{remthetaordinals}. On the other hand, there is another way to look at this correspondence, as follows.  One associates a $d$-globular set $T^*$ (of a rather special type, which we call a globular diagram) to a $d$-level tree $T$. In Figure \ref{figure3} we show an example for $d=2$. 

\sevafigc{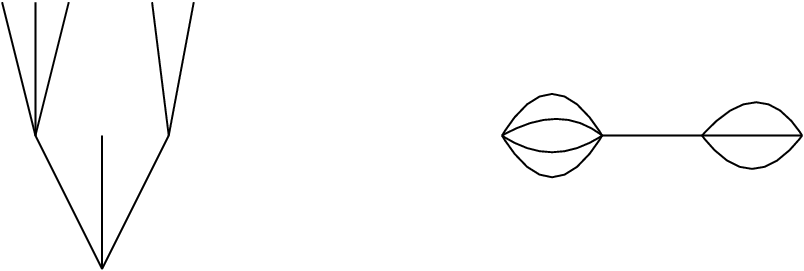}{80mm}{0}{A 2-level tree $T$ (left) and the corresponding 2-globular diagram $T^*$ (right)\label{figure3}}

In general the construction of a $d$-globular set $T^*$ for a $d$-level tree $T$ is due to Batanin  [Ba1], Sect. 4, see also [Be2], Sect. 1. The reader can easily get how it acts on objects from Figure \ref{figure3}. However, an important non-trivial point within this construction is Proposition \ref{bergerprop} below.  With a $d$-globular diagram $D$ one associates the free strict $d$-category generated by this diagram, denote it by $\omega_d(D)$ (see [Ba1], Sect. 4, [Be2], Def.1.8, it is an analogue for higher $d$ of the well-known construction of the free ordinary category generated by a 1-globular set, or, what is the same, by an oriented graph). One can easily see that 
\begin{equation}\label{eqthetamorph}
\Theta_d(S,T)=\Cat_n(\omega_d(S^*), \omega_d(T^*))
\end{equation}
see [Be2, Sect.1].

The following result is proven in [Be2, Prop. 2.2]:
\begin{prop}\label{bergerprop}
The category $\Disk_d(\Sets)$ is equivalent to the category $\Theta_d^\op$:
\begin{equation}\label{thetancat}
\Theta_d(S,T)=\Disk_d(\overline{T},\overline{S})
\end{equation}
\end{prop}

\qed

A functor $\Theta_d^\op\to\mathcal{E}$ is called {\it a $d$-cellular object} in $\mathcal{E}$, a functor $\Theta_d\to\mathcal{E}$ is called {\it a $d$-cocellular object} in $\mathcal{E}$.

\begin{remark}\label{remthetaordinals}
{\rm
This remark clarifies the statement that objects of the category $\Theta_d$ are $d$-level trees. Both the definition of a $d$-level tree \eqref{eqntree} and the wreath product definition of $\Theta_d$ involve 1-ordinals, and one has to specify how these ordinals are related to each other. Assume a $d$-level tree $T$ is given by a sequence of maps of ordinals \eqref{eqntree}. To $T$ is associated an object $D(T)\in\Theta_d$. Then $D(T)=([k]; D_1,\dots, D_k)$ where $D_i\in\Theta_{d-1}$. Then, by definition, the ordinal $t_1$ in \eqref{eqntree} is {\it the ordinal of elementary intervals} of $[k]$, and the $(d-1)$-level trees $T_1,\dots,T_k$, which are preimages of $k$ elements of $t_1$, are related to the objects $D_1,\dots,D_k\in\Theta_{n-1}$ as $D_i=D(T_i)$. By induction, is gives the association of an object of $\Theta_d$ to a $d$-level tree. 
}
\end{remark}

\begin{remark}\label{remthetagraphs}
{\rm
On the other hand, the category $\Tree_d$ is a subcategory of the category $\Disk_d$ having the same objects, but only those morphisms which map the inner vertices to inner vertices. Thus one can ask ourselves what is corresponded to 
$\Tree_d(T,S)\subset \Disk_d(\overline{T},\overline{S})$ in the l.h.s. of \eqref{thetancat}. The answer is that it is exactly maps in $\Theta_d(S,T)\simeq \Cat_d(\omega_d(S^*),\omega_d(T^*))$ which preserve all minima and all maxima.  (Thus, for $d=2$, it means, in notations of \eqref{thetamorph}, that $\phi(0)=0, \phi(n)=m$, and each $\phi_i^j$ maps the minimum (resp., the maximum) of the 1-ordinal $A_i$ to the minimum (resp., maximum) of the 1-ordinal $B_j$). It can be proven similarly to the proof of Proposition \ref{bergerprop} given in [Be2]. 

In fact, this statement provides an explanation why the category of $d$-level trees appears in the theory of Batanin $d$-operads. Namely, it shows that the morphisms of level trees correspond to the $d$-categorical pasting compositions. 
}
\end{remark}

Proposition \ref{bergerprop}, as well as its supplemental part stated in Remark \ref{remthetagraphs}, can be considered as the Joyal duality [J] for higher $d$.
Recall that for the category $\Delta$ ($d=1$ case) the Joyal duality is the following equivalence:
\begin{equation}
\Delta\simeq \Delta_{**,\ge 1}^\op
\end{equation}
where $\Delta_{**,\ge 1}$ has objects $[[1]],[[2]],[[3]],\dots$ (the ordinal $[[0]]$ is excluded), where $[[k]]$ is a finite interval of $k+1$ elements, and the morphisms are those morphisms in $\Delta$ which preserve the minima and maxima of the ordinals. 
The functor $\phi\colon \Delta\to \Delta_{**,\ge 1}$ is defined as $[k]\to \Delta([k],[1])=[[k+1]]$. The set $\Delta([k],[1])$ has a natural structure of ordinal with $k+2$ elements, its minimum and maximum are those maps $[k]\to [1]$ for which all elements of $[k]$ are mapped to $[0]$ and to $[1]$ , correspondingly. Clearly these extreme maps are preserved by any morphism $[k]\to [\ell]$ in $\Delta$. The $n=1$ case of Proposition \ref{bergerprop} is exactly this equivalence. It is instructive to think of elements of $[[k+1]]$ as the ordinal of elementary intervals of $[k]_+$ where $[k]_+$ is the ordinal $-1<0<1<\dots<k<k+1$, that is it is obtained from $[k]$ by adding the minimum and the maximum elements.

\subsection{\sc Generalised disks}
Let $C$ be a small ordinary category. We call an object $c\in C$ (generalised) minimum if there are no arrows distinct from $\id_c$  having $c$ as its target. Similarly we call an object $d\in C$ (generalised) maximum if there are no arrows distinct from $\id_d$ having $d$ as its source. A {\it generalised interval} is a connected small category having a unique minimum and a unique maximum. (By definition the minimum and the maximum are distinct objects, except for the case when a generalised interval is the final category, the latter case is referred to as the degenerate generalised interval).

In fact, in all our examples it will be enough to consider generalised intervals of the form $I_{n_1}\boxx\dots\boxx I_{n_k}$  where $\boxx$ is the funny product (see the definition below Definition \ref{gendiskdef}), and $I_{n_s}$ are ordinary intervals (that is, ordinary categories generated by a finite linear poset $0<1<\dots<n_s$). 

\begin{defn}\label{gendiskdef}{\rm
A generalised $n$-disk is a functor $D\colon \overline{\mathbb{G}}_d\to\Sets$, $D(k)=X_k$, such that $X_0$ is a singleton, and for any $x\in X_{k}$ the set $i_{k}^{-1}(x)\subset X_{k+1}$ has a  structure of generalised interval (which is a part of data), whose generalised minimum and generalised maximum are $s_k(x)$ and $t_k(x)$, correspondingly. The generalised interval $i_{k}^{-1}(x)$ is degenerate (if and) only if $x\in \partial X_{k-1}$.
A morphism $f$ of generalised $n$-disks is a natural transformation of functors $\overline{\mathbb{G}}_d\to\Sets$, whose restriction  to the corresponding fibres $i_k^{-1}(x)\to i_k^{-1}(f(x))$ is a functor, for any $x\in X_k$. (This functor preserves generalised minima and maxima from naturality). The category of generalised $n$-disks is denoted by $\GDisk_n$.
}
\end{defn}
It is clear that the $n$-disk $\overline{T}$ is a generalised disk, the fibres are standard interval categories associated to ordinals. 

Our motivation for introducing generalised disks here is existence of a funny product of generalised $n$-disks which is a generalised $n$-disk again, defined as follows.

Recall that the funny product $C\boxx D$ of ordinary categories is an ordinary category having $\Ob(C)\times \Ob(D)$ as objects, and the morphisms are generated by morphisms $f\boxx \id_d\in (C\boxx D)((c_1,d),(c_2,d))$ and $\id_c\boxx g\in (C\boxx D)((c,d_1),(c,d_2))$, where $c,d$ are objects of $C$ and $D$, correspondingly, $f\colon c_1\to c_2$, $g\colon d_1\to d_2$ are morphisms in $C$ and $D$, with only relations $(f_2\boxx\id_d)\circ (f_1\boxx \id_d)=f_2f_1\boxx\id_d$, $(\id_c\boxx g_2)\circ (\id_c
\boxx g_1)=\id_c\boxx g_2g_1$. The main point is that the following relation in $C\times D$ is {\it dropped}: 
$(\id_{c_2}\times g)\circ (f\times\id_{d_1})=(f\times\id_{d_2})\circ (\id_{c_1}\times g)$. 

Clearly the funny product endows the category $\Cat$ of ordinary categories with a symmetric monoidal structure. It has a right adjoint internal Hom, the category of functors and non-natural transformations. 

\begin{lemma}
Let $C,D$ be generalised intervals. Then the funny product $C\boxx D$ is again a generalised interval, which is degenerate iff $C$ and $D$ are both degenerate. Thus the funny product defines a symmetric monoidal structure on generalised intervals, whose unit is the degenerate generalised interval. 
\end{lemma}
It is clear.
\qed

Let now $X$ and $Y$ be two generalised $n$-disks. Define their funny product $X\boxx Y$ as the generalised $n$-disk with the underlying sets of objects
$(X\boxx Y)_k=X_k\times Y_k$, such that the generalised interval structure on $i_k^{-1}(x,y)=i_k^{-1}(x)\boxx i_{k}^{-1}(y)$ is the funny product of the corresponding generalised intervals. One easily shows that the funny product gives rise to a bifunctor on the category $\GDisk_n$, for any $n\ge 1$. 

\begin{prop}\label{propfunnyn}
The funny product defined as above endows the category of generalised disks $\GDisk_n$ with a symmetric monoidal structure, for any $n\ge 1$. Its unit is the degenerate $n$-disk. 
\end{prop}

\qed

These innocent generalisations are however crucial for the definition of the higher lattice path operads, and for the entire paper. 

\section{\sc Higher lattice path operads}\label{sectionlp}
\subsection{\sc The classical lattice path operad}\label{sectionlp1}

\subsubsection{}
Recall the definition of the (classical) lattice path operad $\mathcal{L}$ [BB]. It is a colored symmetric operad with the category of unary operations $\Delta$, whose components are defined as
\begin{equation}\label{latticepath1}
\mathcal{L}([n_1],\dots,[n_k]; [n])=\Cat_{**}([[n+1]], [[n_1+1]]\boxx[[n_2+1]]\boxx\dots\boxx[[n_k+1]])
\end{equation}
where in the r.h.s. $\Cat_{**}$ denotes the category of generalised intervals (so the morphisms preserve the source and the target objects). 

Fix an element $\omega\in \mathcal{L}([n_1],\dots,[n_k]; [n])$ and $1\le i<j\le k$. Define $\omega_{ij}\in \mathcal{L}([n_i],[n_j];[n])$ as the composition of $\omega$ with the projection $p_{ij}\colon [[n_1+1]]\boxx\dots\boxx [[n_k+1]]\to [[n_i+1]]\boxx[[n_j+1]]$. We say that $\omega_{ij}$ has {\it complexity} $\ell$ if the lattice path $\omega_{ij}$ changes the direction exactly $\ell$ times, and we say that $\omega$ has complexity $\le \ell$ if for any $i<j$ the lattice path $\omega_{ij}$ has complexity $\le \ell$. Denote by $\mathcal{L}_\ell\subset \mathcal{L}$ the subset of all lattice paths having complexity $\le \ell$. One checks $\mathcal{L}_\ell$ is a suboperad of $\mathcal{L}$. Note that the simplicial unary operations preserve or lower complexity, so we get a filtration by $\Delta$-colored symmetric operads 
\begin{equation}\label{lpfiltration}
\mathcal{L}_1\subset \mathcal{L}_2\subset\dots\subset \mathcal{L}
\end{equation}
It is proven in [BB], Sect.3 that the topological (resp., dg over any ring $\k$) condensation of the operad $\mathcal{L}_\ell$ is weakly equivalent to the little disk operad $E_\ell$ (resp., to its chain complex operad $C_\ldot(E_\ell,\k)$). It is proved using the {\it cellular 
structure} on the lattice path operad $\mathcal{L}$, defined in terms of Berger complete graph operad.

Each lattice path $\omega$ as above defines a permutation $\sigma\in \Sigma_k$, called {\it the first order movement}. 
It is defined as follows: $\sigma(i)<\sigma(j)$ if the lattice path moves along the $i$-th direction at first time earlier than it moves along the $j$-th direction. So the permutation $\sigma$ encodes the order in which the lattice move along the coordinate axes the first time (the lattice path since it moves along the $i$-th direction the first time can change direction and return to direction $i$ many more times, but only the first movement matters for $\sigma$). Note that a lattice path in $\mathcal{L}([0],\dots,[0];[0])$ is uniquely defined by its first order movement.

\subsubsection{\sc ``Contractible blocks'' of $\mathcal{L}_\ell$}\label{sectionposet}
The symmetric operad $\mathcal{L}$ admits a cellular structure in the sense of C.Berger [Be1]. Namely, Berger introduced a symmetric operad in posets $\mathcal{K}$, and studies topological symmetric operads which admit ``cell decomposition'' into closed subsets, such that the morphisms in the poset correspond to cofibrations, such that each particular cell (whose ``interior'' is non-empty) is contractible, and such that the cells with non-empty interior cover our operad. In this way he proves that the topological operad is $E_\infty$, and the canonical filtration on it, coming from the cellular structure, is a filtration by operads $E_n$. This beautiful theory is a powerful method to prove that a given operad is an $E_n$ operad: the computation reduces to the nerve of complete graph operad, which is computed for a cellular structure on the little $n$-cubes operad, and then any operad admitting $\mathcal{K}^{\le n-1}$-cellular structure is an $E_n$ operad. Although it is a proper context for what we call ``contractible blocks'' here (which are in turn the cells in Berger's cell decomposition), we do not go into detail on the complete graph operad here, referring the reader to [Be1] and [BeM] for more detail. See also remarks at the end of [Ba2, Sect.11] for a link between the complete graph operad and Batanin classificators, from which it follows that the relative Batanin classificator of internal $n$-operads in a categorical symmetric operad is mapped canonically to the complete graph operad.

Assume $k\ge 1$ is given. Consider a pair $(\mu,\sigma)$, where $\sigma\in\Sigma_k$ is a permutation, and $\mu=\{\mu_{ij}\}_{1\le i<j\le k}$ where $\mu_{ij}\ge 0$ are integral numbers. Let $i<j$, we define $\sigma_{ij}=1$ if $\sigma(i)<\sigma(j)$, and $\sigma_{ij}=-1$ if $\sigma(i)>\sigma(j)$.  

We say that 
\begin{equation}\label{bergerposet}
(\mu,\sigma)\le (\mu^\prime,\sigma^\prime)\text{ if for any }i<j\text{ one has either }\sigma_{ij}=\sigma^\prime_{ij}, \mu_{ij}\le \mu_{ij}^\prime, \text{  or  }\sigma_{ij}\ne\sigma^\prime_{ij}, \mu_{ij}<\mu^\prime_{ij}
\end{equation}
 For $k=2$ this is poset of closed hemispheres which form a cell decomposition of $S^\infty$. For given $k\ge 1$, we denote by $\mathcal{K}(k)$ the constructed poset. 

For a lattice path $\omega\in \mathcal{L}([n_1],\dots,[n_k]; [n])$ denote by $(\mu,\sigma)(\omega)\in\mathcal{K}(k)$ the element 
with $\mu_{ij}+1$ equal to the complexity of $\omega_{ij}$, and $\sigma$ equal to the first order movement permutation of $\omega$. 

Let $(\mu,\sigma)\in\mathcal{K}(k)$.
Define 
\begin{equation}\label{latticepathcell}
\mathcal{L}_{(\mu,\sigma)}(k)=\big\{
\omega\in \mathcal{L}([n_1],\dots,[n_k]; [n])|\ (\mu,\sigma)(\omega)\le (\mu,\sigma)\big\}
\end{equation}
One sees directly that in this way we obtain a functor $\mathcal{L}^{(\mu,\sigma)}\colon (\Delta^\op)^{\times k}\times \Delta\to\Sets$. 
For fixed $[n]\in\Delta$, denote by $\mathcal{L}_{(\mu,\sigma)}([n])$ the corresponding functor $(\Delta^\op)^{\times k}\to\Sets$.

The following statement (the contractibility of the cells) is fundamental:

\begin{prop}\label{propcelltop}
Fix $(\mu,\sigma)\in\mathcal{K}(k)$. The following statements are true:
\begin{itemize}
\item[(i)] For a fixed $n$, the topological realization of the polysimplicial set $\mathcal{L}^{(\mu,\sigma)}[n]$ is a contractible topological space, denoted by $|\mathcal{L}_{(\mu,\sigma)}[n]|$.
\item[(ii)] The topological totalization of the cosimplicial topological space $[n]\mapsto |\mathcal{L}_{(\mu,\sigma)}[n]|$, is contractible.
\end{itemize}
\end{prop}

The proof goes back to [MS3, Sect. 12-15]. The statement (i) follows from combination of [MS3], Prop. 12.7, Cor. 13.3, and Lemma 14.8. Then the statement (ii) is proved in [MS3], Sect. 15.

We provide a more direct approach, see Section \ref{topcond1}, which generalises to higher lattice path operads with the categories $\Theta_d$ as unary operations. For $d>1$ [MS3, Prop. 12.7] fails, so one had to find a way to overcome it. Our approach is much simpler and is based on a generalisation of [MS3, Lemma 14.8], see Proposition \ref{propms}. 

Note that, based on the scheme of proof given in [MS3], Batanin and Berger gave [BB, 3.5-3.9] a general homotopical approach for computing different condensations of the filtration components $\mathcal{L}_\ell$ of the lattice path operad.

The following statement is an analogue of Proposition \ref{propcelltop} for dg condensation. 

\begin{prop}\label{propcelldg}
Fix $(\mu,\sigma)\in\mathcal{K}(k)$. The following statements are true:
\begin{itemize}
\item[(i)] For a fixed $n$, the realization in $C^\udot(\mathbb{Z})$ of the polysimplicial set $\mathcal{L}_{(\mu,\sigma)}[n]$, denoted by $|\mathcal{L}_{(\mu,\sigma)}[n]|_{\dg}$, is quasi-isomorphic to $\mathbb{Z}[0]$.
\item[(ii)] The totalization in $C^\udot(\mathbb{Z})$ of the cosimplicial dg abelian group $[n]\mapsto |\mathcal{L}_{(\mu,\sigma)}[n]|_\dg$, is quasi-isomorphic to $\mathbb{Z}[0]$.
\end{itemize}
\end{prop}
The proof is given in [BB, 3.10(c)] and [BBM, Th.3.10].

Again, our approach provides a more direct argument, see Section \ref{dgcond1}, which also generalises to the higher lattice path operads. 

\subsubsection{\sc Lattice path operad $\mathcal{L}$ and a Batanin $n$-operads}\label{rembb}
We refer the reader to Appendix \ref{appendixa} and the references therein for basic definitions and facts on Batanin higher operads. Recall that a pruned $\ell$-tree $T$ is the same as an $\ell$-ordinal structure on the set of leaves $|T|$, and a map of pruned $\ell$-trees is the same as a map of $\ell$-ordinals (that is, \eqref{nordinals} holds), see Remark \ref{remprune}(2), (3). 

The $\ell$-th filtration component $\mathcal{L}_\ell$ of the lattice path operad is a colored symmetric operad. One can define similarly a Batanin colored $n$-operad. 

We construct a $\Delta$-colored pruned $(\ell-1)$-terminal $\ell$-operad $\mathcal{L}_{\ell, B}$ in $\Sets$, as follows.
Let $T$ be a pruned $\ell$-level tree (which is the same as an $\ell$-ordinal), whose vertices are canonically ordered. Let $T$ have $k$ leaves, the set of all leaves is denoted by $|T|$, the leaves are labelled by ordinals $[n_1],\dots,[n_k]$. Define $\mu(T)_{ab}=\ell-1-i$ if $a<_i b$, $a,b\in |T|$, $a<b$. Define
\begin{equation}\label{latticepathlop}
\mathcal{L}_{\ell,B}(T)([n_1],\dots,[n_k]; [n])=\mathcal{L}_{(\mu(T),\id)}([n_1],\dots,[n_k]; [n])
\end{equation}
\begin{prop}\label{proptamop}
$\{\mathcal{L}_{\ell,B}(T)([n_1],\dots,[n_k]; [n])\}$ are components of a $\Delta$-colored $(\ell-1)$-terminal $\ell$-operad. 
\end{prop}
For $\ell=2$ the 1-terminal 2-operad $\{\mathcal{L}_{2,B}=\mathbf{seq}$ is the Tamarkin 2-operad, introduced in [T3]. It plays an important role in the proof [T3] of the Deligne conjecture for Hochschild cochains. Moreover, the papers [BM1,2] show that the Tamarkin 2-operad is universal for questions like ``What do fg categories form?'' for different enrichments. With different condensations, the same 2-operad answers the questions like ``What do Gray categories form?'' or presumably even ``What do Crans categories form?''. 

On the other hand, some ``higher categorical'' problems of the same type require the higher operads $\mathbf{seq}_d$ we construct here. Thus, the $\Theta_2$-colored 3-operad $\mathbf{seq}_2$ plays a similar role for the question ``What do dg bicategories form?" [PS]. 

\begin{proof}
We only need to show that these components are closed under the $\ell$-operadic composition associated with a map $\phi\colon T\to S$ of (labelled, pruned) $\ell$-trees (in this case we assume that the input labels of $S$ are equal to the corresponding output labels of the fibres $\{\phi^{-1}(s)$), as well as under unary operations acting on labels. The latter case is a direct check, so we consider the former one. Let $a,b\in |T|$, $a<b$, there are two cases: (1) $a,b\in \phi^{-1}(s), s\in |S|$, (2) $a\in \phi^{-1}(s_1), b\in \phi^{-1}(s_2)$, $s_1\ne s_2$. Take $\omega\in \mathcal{L}_{\ell,\ns}(S), \omega_s\in \mathcal{L}_{\ell,\ns}(\phi^{-1}(s))$. Denote by $\tilde{\omega}$ the operadic composition of $\omega,\{\omega_s\}$ (given by the corresponding superposition of the lattice paths). In case (1), the complexity of $\tilde{\omega}_{ab}$ is equal to the complexity of $(\omega_s)_{a^\prime b^\prime}$, where $a^\prime$ and $b^\prime$ are the corresponding leaves of $\phi^{-1}(s)$; at the same time, $a<_ib$ in $T$ iff $a^\prime<_i b^\prime$, so this case is clear. In case (2), we consider two subcases: (2.1) the points $a<b$ in $T$ and $s_1<s_2$ in $S$, (2.2) $a<b$ in $T$, $s_2<s_1$ in $S$. In case (2.1), $a<_ib$ implies $s_1<_j s_2$ for $j\ge i$ by \eqref{nordinals2}, at the same time the complexity of $\tilde{\omega}_{ab}$ is equal to the complexity of $\omega_{s_1s_2}$ (and the first order movements are the same), so this case is clear. In case (2.2), $a<_ib$, $s_2<_js_1$ for $j>i$ by \eqref{nordinals2}, the complexities of $\tilde{\omega}_{ab}$ and $\omega_{s_1s_2}$ are the same, but the first order movement are different. So $j>i$ guarantees the claim in this case, despite the fact that the first order movements are different, by \eqref{bergerposet}.

\end{proof}

\begin{coroll}
The topological and the dg condensations of the $\Delta$-colored $\ell$-operad $\mathcal{L}_{\ell,\ns}$ are contractible $\ell$-operads. 
\end{coroll}
\begin{proof}
For topological condensation, it follows immediately from the contractibility of the components, which is Proposition \ref{propcelltop}. For dg case, one has to argue why the contractibility of the components (proved in Proposition \ref{propcelldg}) implies the contractibility of the operad.

For a $\mathbb{Z}$-graded complex $K$ denote by $\tau{\le 0}(K)$ the $\mathbb{Z}_{\le 0}$-graded complex whose cohomology in degrees $\le 0$ coincide with those of $K$. Explicitly,
$$
\tau_{\le 0}(K)=\dots\to K^{-2}\xrightarrow{d}K^{-1}\xrightarrow{d}Z^0\to 0
$$
where $Z^0=\mathrm{Ker}(d: K^0\to K^1)$. 

Applying $\tau_{\le 0}(-)$ to the components of the operad $|\mathcal{L}_{\ell,\ns}|_\dg(k)$, we get a dg {\it suboperad} $\tau_{\le 0}|\mathcal{L}_{\ell,\ns}|_\dg$. By Proposition \ref{propcelldg}, the embedding
$$
i\colon \tau_{\le 0}|\mathcal{L}_{\ell,\ns}|_\dg\to |\mathcal{L}_{\ell,\ns}|_\dg
$$
is a quasi-isomorphism od dg operads. 

On the other hand, the same Proposition \ref{propcelldg} and a simple direct computation in degree 0 show that the projection to $H^0(-)$
$$
p\colon \tau_{\le 0}|\mathcal{L}_{\ell,\ns}|_\dg\to H^0(|\mathcal{L}_{\ell,\ns}|_\dg)\simeq \mathbb{Z}
$$
is a quasi-isomorphism of dg operads.

We get a zig-zag of quasi-isomorphisms 
$$
\xymatrix{
&\tau_{\le 0}|\mathcal{L}_{\ell,\ns}|_\dg\ar[ld]_{i}\ar[rd]^p\\
\mathbb{Z}[0]&&|\mathcal{L}_{\ell,\ns}|_\dg
}
$$
which gives the claim. 
\end{proof}

\subsection{\sc Higher lattice path operads}\label{sectiongenlp1}
The funny product of generalised disks makes it possible to define higher generalisations of the lattice path operad. Then we define  a colored $(d+1)$-operad whose category of unary operations (also called the category of colors) is $\Theta_d$ sitting inside it. 
For $d=1$ we get a 2-operad with the category of colors $\Delta$, which is the Tamarkin 2-operad $\mathbf{seq}$ [T3], defined via shuffles, in its lattice path form it is the 2-operad $\mathcal{L}_{2,\ns}$ (see Remark \ref{rembb}).

Let $T_1,\dots,T_k; T$ be pruned $d$-level trees. Recall that to them are associated the corresponding objects of the category $\Theta_d$ (the object of $\Theta_d$ corresponded to $T$ is denoted by $T^*$), and of the category $\Disk_d$ (the $d$-disk corresponded to $T$ is denoted by $\overline{T}$). Note that we have a fully faithful embedding $\Disk_d(\Sets)\to \GDisk_d$, with the ordinal category structure on the fibres. Therefore, we may consider each $d$-disk $\overline{T}$ as a generalised disk. 

Define the components of a symmetric $\Theta_d$-colored operad (called the $d$-higher lattice path operad) by
\begin{equation}\label{latticepath2}
\mathcal{L}^{d}(T_1,\dots, T_k; T)=\GDisk_d(\overline{T}, \overline{T}_1\boxx \overline{T}_2\boxx\dots\boxx \overline{T}_k)
\end{equation}
The operadic compositions are straightforward. Note that for $d=1$ we get back to the lattice path operad, see \eqref{latticepath1}. We refer to elements of \eqref{latticepath2} as $d$-lattice paths. 

Next, we introduce a concept of complexity for $\omega\in \mathcal{L}^d(T_1,\dots, T_k; T)$ as an integral vector $(c_1(\omega),\dots,c_d(\omega))$ where $c_i(\omega)$ is roughly the complexity at level $i$, defined as follows.

Here $c_1(\omega)$ is the ``classical'' complexity of the corresponding lattice path at level 1, defined as the maximum of the 
$c_1(i,j)$, where $c_1(i,j)$ is the complexity at level 1 of the $(i,j)$-projection of the lattice path, that is
the number of times the $(i,j)$-projection of the lattice path changes the direction (also equal to the number of angles).

To define the higher complexities $c_i(\omega)$ we use induction and presentation of a classical $d$-disk $\overline{T}$ as 
$\overline{T}=([[m]], t^1,\dots, t^{m-1})$, where $[[m]]$ is the level 1 ordinal of $\overline{T}$, $t^i$ are classical $d-1$ disks, corresponded to levels $2,\dots,d$ of $T$,  $t^i$ is attached to the element $i\in [[m]]$ (in this notation, we ignore some boundary elements). 

Let $\omega$ be as above, use notations $\overline{T}_i=([[m_i]], t_i^1,\dots, t_i^{m_i-1})$. Then $\omega$ is given by generalised disk map $\phi\colon \overline{T}\to \overline{T}_1\boxx\dots\boxx \overline{T}_k$, denote by $\phi_i$ the level $i$ component of $\phi$. In particular, $\phi_1\colon [[m]]\to [[m_1]]\boxx\dots\boxx [[m_k]]$ is a usual lattice path, and $c_1(\omega)=c_1(\phi_1)$. 
The map $\phi_1$ is given by its projections $\phi_1^i\colon [[m]]\to [[m_i]]$, $1\le i\le k$.

For each $a=1,\dots,m-1$, we get a generalised lattice path $\omega_a\colon t^a\to t_1^{\phi_1^1(a)}\boxx t_2^{\phi_1^2(a)}\boxx\dots\boxx t_{n-1}^{\phi_1^{n-1}(a)}$.
Define by induction $c_i(\omega)=\max_{a}c_{i-1}(\omega_a)$, $a\ge 2$.

In this way the total complexity $c(\omega)$ becomes a string of integral numbers $(c_1(\omega),\dots,c_d(\omega))$.

\begin{prop}\label{symglpop}
Let $d\ge 1$ be fixed. The sequence of sets $\{\mathcal{L}^{d}(T_1,\dots, T_k; T)\}$ are components of a $\Theta_d$-colored symmetric operad $\mathcal{L}^d$. For each complexity $c=(c_1,\dots,c_d)$ the subset $$\mathcal{L}^{d}_c(T_1,\dots, T_k; T)\subset \mathcal{L}^{d}(T_1,\dots, T_k; T)$$ formed by lattice paths of total complexity component-wise $\le c$, are components of a symmetric suboperad 
$\mathcal{L}^d_c\subset \mathcal{L}^d$. 
\end{prop}
The only thing one has to check is closeness of $\mathcal{L}^d_c$ under the operadic composition, for any total complexity $c$.
It is enough to show for the compositions having two non-identity arguments, which is clear for operations with $k\ge 2$, and is checked directly for the case when one operation is an unary operation. 

\qed

\subsection{\sc Products of Berger posets and contractible blocks}\label{sectiongenlp2}
Recall the poset $\mathcal{K}(k)$, see Section \ref{sectionposet}. Fix $d\ge 1$ and $k\ge 1$. 
Consider the poset $\mathcal{K}(k)^{\times d}$. Its element is a tuple 
$$
(\boldsymbol{\mu},\boldsymbol{\sigma})=\big((\mu^1,\sigma_1),\dots,(\mu^d,\sigma_d)\big)
$$
where $\mu^s=\{\mu^s_{ij}\}_{1\le i<j\le n}$, $\sigma_s\in \Sigma _k$. The $\le$ relation is defined component-wise via \eqref{bergerposet}. 

Define ``blocks'' $\mathcal{L}^{d}_{(\boldsymbol{\mu},\boldsymbol{\sigma})}$ as follows. 

For a generalised lattice path $\omega\in 
\GDisk_d(\overline{T}, \overline{T}_1\boxx \overline{T}_2\boxx\dots\boxx \overline{T}_k)$ we get
a single ordinary lattice path at level 1, and in general more ordinary lattice paths at higher levels, which number depends on the level, as it is explained in Section \ref{sectiongenlp1}.  For level $1\le \ell\le d$, we get several ordinary lattice paths $\omega^\ell_a$ at level $\ell$, $a \in S_\ell$, the set $S_\ell$ depends on $\ell$, each of which has its own pair $(\mu^\ell_a(\omega), \sigma_{\ell,a})(\omega)$, defined as $\mu^\ell_a(\omega)_{ij}+1=c((\omega^\ell_a)_{ij})$, $\sigma_{\ell,a}$ is defined as the first order movement of the ordinary lattice path $\omega^\ell_a$. We require that that level $\ell$ component $(\mu^\ell,\sigma_\ell)$ of $(\mu,\sigma)$ should be $\ge (\mu^\ell_a(\omega),\sigma_{\ell,a}(\omega))$ in the sense of \eqref{bergerposet}:

\begin{equation}
\mathcal{L}^{d}_{(\boldsymbol{\mu},\boldsymbol{\sigma})}(T_1,\dots,T_k; T)=\{\omega\in \mathcal{L}^{d}(T_1,\dots,T_k; T)| \forall \ell, \forall a\in S_\ell, (\mu^\ell_a(\omega),\sigma_{\ell,a}(\omega))\le (\mu^\ell,\sigma_\ell)\}
\end{equation}

One checks directly that $\mathcal{L}^{d}_{(\boldsymbol{\mu},\boldsymbol{\sigma})}$ gives rise to a functor 
$$
\mathcal{L}^{d}_{(\boldsymbol{\mu},\boldsymbol{\sigma})}\colon (\Theta_d^\op)^{\times k}\times\Theta_d\to\Sets
$$
For $T\in\Theta_d$, denote by $\mathcal{L}_{(\boldsymbol{\mu},\boldsymbol{\sigma})}(T)$
the functor 
$$
\mathcal{L}_{(\boldsymbol{\mu},\boldsymbol{\sigma})}(-,\dots,-;T)\colon (\Theta_d^\op)^{\times k}\to \Sets
$$

Similarly to the case $d=1$, considered in Propositions \ref{propcelltop} and \ref{propcelldg}, one has the following contractibility results:

\begin{theorem}\label{propthetatop}
Fix $(\boldsymbol{\mu},\boldsymbol{\sigma})\in\mathcal{K}(k)^{\times d}$. The following statements are true:
\begin{itemize}
\item[(i)] For a fixed $T$, the topological realization of the poly-$d$-cellular set $\mathcal{L}_{(\boldsymbol{\mu},\boldsymbol{\sigma})}(T)$ is a contractible topological space, denoted by $|\mathcal{L}_{(\bmu,\bsigma)}(T)|$.
\item[(ii)] The topological totalization of the $d$-cocellular topological space $T\mapsto |\mathcal{L}_{(\bmu,\bsigma)}(T)|$, is contractible.
\end{itemize}
\end{theorem}

\begin{theorem}\label{propthetadg}
Fix $(\boldsymbol{\mu},\boldsymbol{\sigma})\in\mathcal{K}(k)^{\times d}$. The following statements are true:
\begin{itemize}
\item[(i)] For a fixed $T$, the realization in $C^\udot(\mathbb{Z})$ of the poly-$d$-cellular set $\mathcal{L}_{(\boldsymbol{\mu},\boldsymbol{\sigma})}(T)$, denoted by $|\mathcal{L}_{(\bmu,\bsigma)}(T)|_\dg$, is quasi-isomorphic to $\mathbb{Z}[0]$.
\item[(ii)] The totalization in $C^\udot(\mathbb{Z})$ of the cosimplicial topological space $T\mapsto |\mathcal{L}_{(\bmu,\bsigma)}(T)|_\dg$, is quasi-isomorphic to $\mathbb{Z}[0]$.
\end{itemize}
\end{theorem}

We have to explain what we mean by the realization of a functor $(\Theta_d^\op)^{\times k}\to Sets$ and by the totalization 
of a functor $\Theta_d\to \Top (C^\udot(\mathbb{Z}))$. We do that, and prove Theorems \ref{propthetatop} and \ref{propthetadg}, in Section \ref{sectioncontr2} below.

\section{\sc The case $d=1$: a revision of the McClure-Smith approach}\label{sectioncontr1}
Here we prove Propositions \ref{propcelltop} and \ref{propcelldg}. The statements are known (see [MS3] for the topological case, and [BB], [BBM] for the dg case), however, we prove them by a new method, which, unlike the proofs in loc.cit., works for the case of higher lattice path operad, see Theorems \ref{propthetatop} and \ref{propthetadg}. The approach (in the topological case) is as follows:
\begin{itemize}
\item[(a)] It is rather straightforward that the proof of Lemma 14.8 of [MS3] can be generalised to a proof of the analogous statement statement for any cosimplicial degree, see  Proposition \ref{propms}, that is, that for a fixed $n$ the realization $|\mathcal{L}_{(\mu,\sigma)}[n]|$ is contractible.
\item[(b)] After that, the problem is to compute the totalization of the cosimplicial space $[n]\mapsto |\mathcal{L}_{(\mu,\sigma)}[n]|$ (each of which component is contractible by (a)), and to prove its contractibility. It is enough to show that this cosimplicial space is Reedy fibrant (Proposition \ref{propreedyfibrantdelta}), which we establish by showing in Proposition \ref{reedytot} the Reedy fibrancy of the corresponding cosimplicial object in polysimplicial sets, that is, before taking the realization. Then well-known commutativity of the realization with finite limits and preservation of fibrations by the realization gives the result.
\end{itemize}
The main point is that we avoid using ``whiskering'' argument [MS3, Prop. 12.7, Prop. 13.4], which seemingly is special for the case of $\Delta$, and is hardly generalised to the case of $\Theta_d$, $d>1$.

We prove general Theorems \ref{propthetatop} and \ref{propthetadg} for $\Theta_d$ in Section \ref{sectioncontr2}, by a similar method.

\subsection{\sc The topological condensation}\label{topcond1}
We start with the topological case of Proposition \ref{propcelltop}.

Consider the functor $C\colon \Delta\to \Delta$, defined as $C([n])=[n+1]$, for $\phi\colon [m]\to [n]$, $C(\phi)(i)=\phi(i-1)+1$ for $i>0$, $C(\phi(0))=0$. There is a natural transformation $\alpha\colon \Id\Rightarrow C\colon \Delta\to\Delta$, such that $\alpha_{[n]}\colon [n]\to[n+1]$ is the map $i\mapsto i+1$. This map is given by the extreme face map $\partial_0$, and the naturality is clear.

Consider the standard cosimplicial topological space, $\Delta_\top([n])=\Delta^n$. One has the composition 
$\Delta_\top(C\circ -)$, which is a contractible topological space. Its contraction can be chosen to commute with all cosimplicial operators. The natural transformation $\alpha$ gives a map of cosimplicial spaces $\alpha_*\colon \Delta_\top\to \Delta_\top\circ C$.
\comment
There is another natural transformation $\beta\colon C\Rightarrow\Id\colon \Delta\to \Delta$, given by the extreme degenerace operator $\sigma_0$. For a simplicial object $X$, define $\tilde{X}=X\circ C$, then $\beta$ induces a simplicial map
$\beta^*\colon X\to \tilde{X}$.
\endcomment

The following lemma is well-known:
\begin{lemma}\label{lemmac}
The following statements are true:
\begin{itemize}
\item[1.] Let $X\colon \Delta^\op\to \Sets$ be a simplicial set. 
Then 
$$
X\otimes_{\Delta} (\Delta_\top\circ C)\simeq \pi_0(X)=X\otimes_{\Delta}*
$$
\item[2.] Let $Y\colon (\Delta^\op)^{\times k}\to \Sets$ be a polysimplicial set. Then
$$
Y\otimes_{\Delta^{\times k}}(\Delta_\top\circ C)\times \Delta_\top^{\times (k-1)}=|X|_\top
$$
where $X\colon (\Delta^\op)^{\times (k-1)}\to \Sets$ is
$$
X([n_1]\times\dots\times [n_{k-1}])=\pi_0(Y([-]\times [n_1]\times\dots\times [n_{k-1}])
$$
\end{itemize}
\end{lemma}
\qed

Now we prove Proposition \ref{propcelltop} (i). We can assume that $\sigma=\id$, as the polysimplicial set $\mathcal{L}_{(\mu,\sigma)}[n]$ is isomorphic to $\mathcal{L}_{(\sigma^{-1}(\mu),\id)}[n]$. 

The $n=0$ case of the following Proposition is [MS3, Lemma 14.8]. The only our originality here is the remark that the same argument is applied for any $n$.

\begin{prop}\label{propms}
The topological realization $\mathcal{L}_{(\mu,\id)}[n]$ is contractible for any $n$.
\end{prop}
\begin{proof}
Let $\omega\in \mathcal{L}_{(\mu,\id)}[n]$ be a lattice path, given by a map $\omega\colon [[n+1]]\to [[n_1+1]]\boxx[[n_2+1]]\boxx\dots\boxx[[n_k+1]]$. Define $$c(\omega)\colon [[n+1]]\to [[n_1+2]]\boxx[[n_2+1]]\boxx\dots\boxx[[n_k+1]]$$
with $n_1$ increased by 1, such that its first movement goes along the first coordinate axis, see Figure \ref{figlp}. A simple and crucial remark is that if $\omega\in \mathcal{L}_{(\mu,\id)}[n]$, $c(\omega)\in \mathcal{L}_{(\mu,\id)}[n]$ as well (here we use the assumption $\sigma=\id$). 

\sevafigc{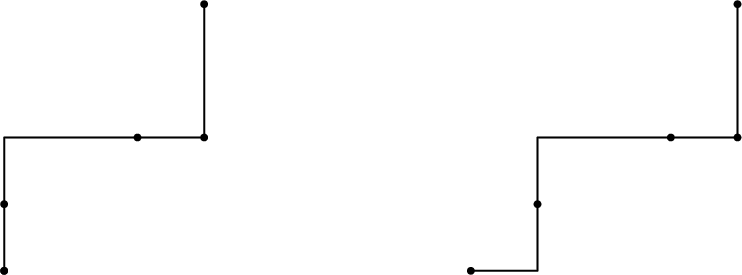}{80mm}{0}{A lattice path $\omega$ (left) and the lattice path $c(\omega)$ (right) \label{figlp}}

Note that $\omega\rightsquigarrow c(\omega)$ extends to a natural transformation of functors denoted by $\beta\colon \mathcal{L}_{(\mu,\id)}[n]\to\mathcal{L}_{(\mu,\id)}[n]$.
Note also that 
\begin{equation}\label{eqnicems}
\partial_0^{(1)}(c(\omega))=\omega
\end{equation}
where $\partial_0^{(1)}$ is the 0-th face map $[n_1]\to [n_1+1]$ (or, equivalently, $[[n_1+2]]\to[[n_1+1]]$), acting on the first argument.

We use induction by number $k$ of simplicial arguments. 

The natural transformation $\alpha\colon \Id\Rightarrow C$ gives a map
\begin{equation}
|\mathcal{L}_{\mu,\id}[n]|_\top=\mathcal{L}_{(\mu,\id)}[n]\otimes_{\Delta^{\times k}}(\Delta_\top\times\dots\times \Delta_\top)\xrightarrow{\alpha_*}\mathcal{L}_{(\mu,\id)}[n]\otimes_{\Delta^{\times k}}((\Delta_\top\circ C)\times \Delta_\top^{\times (k-1)}
\end{equation}
\comment
Denote by $\tilde{\mathcal{L}}_{(\mu,\id)}[n]$ the polysimplicial set whose first argument is precomposed with $C$, so the map $\beta^*$ applied to the first argument gives a map of polysimplicial sets $\beta^*\colon \mathcal{L}_{\mu,\id}[n]\to \tilde{\mathcal{L}}_{(\mu,\id)}[n]$.
\endcomment

Consider the composition
\begin{equation}
\begin{aligned}
\ &|\mathcal{L}_{\mu,\id}[n]|_\top=\mathcal{L}_{(\mu,\id)}[n]\otimes_{\Delta^{\times k}}(\Delta_\top\times\dots\times \Delta_\top)\xrightarrow{\alpha_*}\mathcal{L}_{(\mu,\id)}[n]\otimes_{\Delta^{\times k}}((\Delta_\top\circ C)\times \Delta_\top^{\times (k-1)})
\xrightarrow{\beta_*}\\
&{\mathcal{L}}_{(\mu,\id)}[n]\otimes_{\Delta^{\times k}}((\Delta_\top\circ C)\times \Delta_\top^{\times (k-1)})
\end{aligned}
\end{equation}
By \eqref{eqnicems} and the relations in the coend, the composition is the identity map. Therefore, $|\mathcal{L}_{\mu,\id}[n]|_\top$ is a retract of 
$\mathcal{L}_{(\mu,\id)}[n]\otimes_{\Delta^{\times k}}((\Delta_\top\circ C)\times \Delta_\top^{\times (k-1)})$, which is the realization of $(k-1)$-polysimplicial set obtained by application of $\pi_0(-)$ to the first simplicial argument of $\mathcal{L}_{(\mu,\id)}[n]$. 
This $(k-1)$-polysimplicial set is of the form $\mathcal{L}_{(\mu^\prime,\sigma^\prime)}[n]$, $(\mu^\prime,\sigma^\prime)\in\mathcal{K}(k-1)$, and $|\mathcal{L}_{(\mu^\prime,\sigma^\prime)}[n]|$ is contractible by induction assumption. Thus its retract is also contractible. 

In this argument, when $n>0$ one has to check ``by hands'' that $\pi_0(-)$ consists of a single element. It reduces, by induction on $k$, to the following check. Let $n$ be fixed, then  $\pi_0$ of the simplicial set $\Delta(?,[n])$ is a single element set. 
\end{proof}

It remains to prove Proposition \ref{propcelltop} (ii). 

We have to prove that the topological totalization 
\begin{equation}
\underline{\Hom}_{\Delta}(\Delta^-_\top,|\mathcal{L}_{(\mu,\sigma)}[-]|)
\end{equation}
is contractible. 

We know that $|\mathcal{L}_{(\mu,\sigma)}[n]|$ is contractible for any $n$. By Proposition \ref{propbem}(1), it is enough to prove that the cosimplicial topological space 
$$
[n]\mapsto |\mathcal{L}_{(\mu,\sigma)}[n]|
$$
is Reedy fibrant, which we show in Proposition \ref{propreedyfibrantdelta} below.

\begin{prop}\label{propreedyfibrantdelta}
The cosimplicial topological space 
$$
[n]\mapsto |\mathcal{L}_{(\mu,\sigma)}[n]|
$$
is Reedy fibrant. 
\end{prop}
\begin{proof}
Recall that the statement that a cosimplicial topological space $X^\udot$ is Reedy fibrant means the following:

For any $n$, consider the $n$-th matching map
$$
m_n\colon X^{n+1}\to M^n X^\udot
$$
where 
\begin{equation}\label{eqmatching}
M^nX^\udot=\mathrm{eq}\big(\prod_{i=0}^nX^n\rightrightarrows\prod_{0\le i<j\le n} X^{n-1}\big)
\end{equation}
where the maps two landing at $i<j$ factor are 
$$
\prod_{i=0}^nX^n\xrightarrow{p_j}X^n\xrightarrow{\sigma^i}\text{   and   }\prod_{i=0}^n X^n\xrightarrow{p_i}X^n\xrightarrow{\sigma^{j-1}}X^{n-1}
$$
($p_i$ is the projector to the $j$-th factor, and $\sigma^i$ is the $i$-th codegeneracy map).

The $j$-th component of the map $m_n\colon X^{n+1}\to M^n X^\udot$ is given by the $j$-th codegeneracy map. 

The cosimplicial space $X^\udot$ is called Reedy fibrant if, for any $n\ge -1$, the map $m_n$ is a fibration of topological spaces.

Note that our cosimplicial topological space $|\mathcal{L}_{(\mu,\sigma)}[-]|$ comes as the realization of a polysimplicial cosimplicial  set. The realization functor preserves finite limits [May, Th.14.3] and fibrations [GoJa, Th.10.10], so it is enough to investigate the matching maps $m_n$ for the cosimplicial polysimplicial set $\mathcal{L}_{(\mu,\sigma)}[-]$. Then we have:

\begin{prop}\label{reedytot}
For the cosimplicial polysimplicial space $\mathcal{L}_{(\mu,\sigma)}[-]$ the following statements are true:
\begin{itemize}
\item[(1)] the matching map $m_{-1}$ is the projection of $\mathcal{L}_{(\mu,\sigma)}[0]$ to the constant polysimplicial space, equal to a point,
\item[(2)] the matching map $m_0\colon \mathcal{L}_{(\mu,\sigma)}[1]\to \mathcal{L}_{(\mu,\sigma)}[0]$ is the map forgetting the only marked point at the lattice path; its restriction to the diagonal $\Delta^\op\hookrightarrow (\Delta^\op)^{\times k}$ is a Kan fibration. 
\item[(3)] for $\ell\ge 1$, the matching map $m_\ell$ is an isomorphism. 
\end{itemize}
\end{prop}
\begin{proof}
(1) is a tautology, (3) follows from the description of the cosimplicial operators acting on the lattice paths, given in [BBM, 2.4-2.5].
Namely, the the composition of the unique map $[[1]]\to [[\ell]]$ with a lattice path $[[\ell]]\to [[n_1]]\boxx\dots\boxx [[n_k]]$ results in forgetting the ``marked points'', and produces what we call a ``geometric lattice path''. The cosimplicial maps do not change the underlying geometric lattice path and only affect the marked points, the $i$-th coface maps increases the multiplicity of the $i$-th marked point by 1, the $i$-th codegeneracy decreases the multiplicity of the $i$-th marked point by 1, or removes it if the multiplicity is equal to 1. It follows that $M^\ell\mathcal{L}_{(\mu,\sigma)}[-]$ is given by a tuple of $\ell+1$ lattice paths with $\ell$ marked points each, such that the removing of $j$-th marked point from $i$-th factor (or decreasing its multiplicity by 1, if it is $>1$) equals to removing of $i$-th marked point from the $j$-th factor (or decreasing of its multiplicity). It follows that the geometric lattice paths in the tuple are equal, and if $\ell\ge 1$, there is a unique lattice path $\omega$ with $\ell+1$ marked points, whose underlying geometric lattice path is equal to the ones from the tuple, such that $\sigma^i\omega$ is equal to the $i$-th component of the tuple. Therefore, the matching map $m_\ell$ is the identity map, for $\ell\ge 1$. 

(2): For the case $\ell=0$ we have to show that the restriction $\delta^*$ to the diagonal of the map of polysimplicial sets $m_0\colon \mathcal{L}_{(\mu,\sigma)}[1]\to \mathcal{L}_{(\mu,\sigma)}[0]$  is a Kan fibration of simplicial sets.
Consider the diagram 
$$
\xymatrix{
\Lambda^n_i\ar[r]\ar@{^{(}->}[d]&\delta^*\mathcal{L}_{(\mu,\sigma)}[1]\ar[d]\\
\Delta^n\ar[r]\ar@{.>}[ur]&\delta^*\mathcal{L}_{(\mu,\sigma)}[0]
}
$$
One has to prove it has a filling shown by dashed arrow. 

First of all, the projections $\omega_s$ of the lattice path $\omega\colon [[2]]\to [[n]]^{\boxx k}$ to the $k$ components reduce the lifting problem to the case $k=1$. In the Joyal dual language, we are given a map
$\alpha\colon \Lambda^n_i([n])\to [1]$ and a map $\beta\colon[n]\to [0]$, one has to lift $\beta$ to a map $\hat{\beta}\colon 
[n]\to [1]$ such that $\hat{\beta}|_{\Lambda^n_i}=\alpha$. We are given an element $i\in [n]$, and any subset of $[n]$ containing  $n$ elements including  $i$ is mapped to the two-element ordinal $[1]$, in compatible way. That is, for any $a\in [n], a\ne i$ there is an ordinal map $\alpha_a\colon [n]\setminus\{a\}\to [1]$ such that for $a\ne b\ne i$, one has $\alpha_a|_{[n]\setminus\{a\sqcup b\}}=\alpha_b|_{[n]\setminus\{a\sqcup b\}}$, one has to define $\hat{\beta}\colon [n]\to [1]$ extending all $\alpha_a$. When $n\le 1$ one checks it by hand, when $n\ge 2$ for any $j\in [n]$ there is some $\alpha_a$ defined at $j$ (it is so for any $j\ne a$), if $\alpha_b$ is also defined 
at $j$ ($j\ne a,b$), the compatibility condition tells us that $\alpha_a(j)=\alpha_b(j)$, so we can define $\hat{\beta}(j)$ as any of such $\alpha_a(j)$. It is clear that $\hat{\beta}$ is a map of ordinals solving the lifting problem.

\end{proof}

Now Proposition \ref{propreedyfibrantdelta} follows from 
Proposition \ref{reedytot} and the following well-known facts:
\begin{itemize}
\item[(a)]  the commutativity of the realization of a polysimplicial set with finite limits (it already proves the Corollary for $\ell\ne -1,0$),
\item[(b)] any topological space is fibrant (it completes the case $\ell=-1$),
\item[(c)] the realization of a Kan fibration is a Serre fibration,
\item[(d)] the realization of the diagonal of a polysimplicial set is homeomorphic to the successive realization by each simplicial argument ((c) and (d) are used for the case $\ell=0$).
\end{itemize}

\end{proof}

It follows from these results that the totalization of the cosimplicial space $|\mathcal{L}_{(\mu,\sigma)}[-]|$ is weakly equivalent to the to the totalization of the constant cosimplicial space at a point. The latter totalization is equal to a point, therefore, the totalization of $|\mathcal{L}_{(\mu,\sigma)}[-]|$ is contractible. 

Proposition \ref{propcelltop} is proved.

\begin{remark}{\rm
Regarding the case (2) of Lemma \ref{reedytot}, notice that [MS3, Prop.12.7 and Cor. 13.3] prove that the realization of the map of polysimplicial sets $\mathcal{L}_{(\mu,\sigma)}[1]\to \mathcal{L}_{(\mu,\sigma)}[0]$ is a trivial bundle over the realization $|\mathcal{L}_{(\mu,\sigma)}[0]|$ with fibre $\Delta^1_\top$. Moreover, loc.cit. proves more generally that the realization of the projection $\mathcal{L}_{(\mu,\sigma)}[m]\to \mathcal{L}_{(\mu,\sigma)}[0]$ (defined via the final map $[[1]]\to [[m+1]]$, that is, it is the map forgetting all marked points at the lattice path) is the trivial bundle over  $|\mathcal{L}_{(\mu,\sigma)}[0]|$ with fibre $\Delta^m$. 
On the other hand, it seems that [MS3, Prop.12.7] can not be generalised (or even stated) for the case of $\Theta_d$-lattice paths with $d>1$. It motivated us to find an alternative approach, as we show it here for the case $d=1$. We employ it for the case of general $d$ in Section \ref{sectioncontr2}.
}
\end{remark}

\subsection{\sc The dg condensation}\label{dgcond1}
In this Subsection we apply results of Section \ref{topcond1} to prove Proposition \ref{propcelldg}. The argument is very short now. 
Although we state the results below for complexes in $\mathbb{Z}$-$\Mod$, the same results with the same arguments hold true for modules over any commutative ring $R$. 

We have:
\begin{prop}\label{propmsdg}
For a fixed $n$, the dg realization of $\mathcal{L}_{(\mu,\sigma)}[n]$ is quasi-isomorphic to $\mathbb{Z}[0]$. 
\end{prop}
\begin{proof}
\comment

For a simplicial set $X$, the adjunction unit $X\to\Sing|X|$ is a weak equivalence. That is, they have homotopically equivalent realizations: $|X|\sim |\Sing|X||$. The dg realization of $X$ is isomorphic to the CW chain complex of $|X|$. It follows that the chain complexes of $X$ and of $\Sing|X|$ are quasi-isomorphic, in particular, the dg realization of $X$ is quasi-isomorphic to $\mathbb{Z}[0]$ if $|X|$ is contractible. The same argument is applied to polysimplicial sets. 
Thus, the statement follows directly from Proposition \ref{propms}.

v2:
\endcomment
It is well-known that for a simplicial set $X$, the realization $|X|$ has a CW-complex structure, whose cells of dimension $i$ are in 1-to-1 correspondence with non-degenerate simplices in $X_i$. That is, the CW cell complex of $|X|$ is identified with the dg realization of $X$. On the other hand, the CW homology of a CW complex is isomorphic to its singular homology, which is homotopy invariant. Thus, the statement follows directly from Proposition \ref{propms}.
\end{proof}

To compute the dg totalization of the cosimplicial complex $[n]\mapsto |\mathcal{L}_{(\mu,\sigma)}[n]|_\dg$, we need to know that this cosimplicial complex is Reedy fibrant (with respect to the projective model structure on $C^\udot(\mathbb{Z})$).
It is, as well as in the topological case, a corollary of Proposition \ref{reedytot}.

\begin{coroll}\label{corfibdg}
The cosimplicial complex $|\mathcal{L}_{(\mu,\sigma)}[-]|_\dg$ is Reedy fibrant with respect to projective model structure on $C^\udot(\mathbb{Z})$. 
\end{coroll}
\begin{proof}
The dg realization of a polysimplicial set is the total sum complex of its chain polycomplex. We have to show that the total sum complex of the chain polycomplex of a polysimplicial set commutes with finite limits of polysimplicial sets. In fact, the total sum complex is a left adjoint and the total product complex is right adjoint. As in our situation the total sum complex is the same as the total product complex (because at each diagonal $i_1+\dots +i_k=N$ our polycomplex restricted to this diaginal is finitely-generated free abelian group, and in each polydegree the terms in positive degrees are 0\footnote{Recall that by the convention we adopt the differential has degree +1}, our total complex commutes with finite limits. Then the statement for the matching maps $m_\ell$ with $\ell\ge 1$ follows from Proposition \ref{reedytot} (3). The case $\ell=-1$ follows tautologically. 
The case $\ell=0$ follows from the fact that the map of polysimplicial sets $\mathcal{L}_{(\mu,\sigma)}[1]\to\mathcal{L}_{(\mu,\sigma)}[0]$ is surjective in each simplicial  polydegree, so the map of their chain complexes is surjective, which is a fibration in the projective model structure. 

\end{proof}

From the results of Proposition \ref{propmsdg} and Corollary \ref{corfibdg} is follows that the dg totalization of the cosimplicial complex $[n]\mapsto |\mathcal{L}_{(\mu,\sigma)}[n]|_\dg$ is quasi-isomorphic to the dg totalization of the constant cosimplicial complex equal to $\mathbb{Z}[0]$ in any cosimplicial degree. The totalization of the latter complex is $\mathbb{Z}$. 

Proposition \ref{propcelldg} is proved. 

\section{\sc The contractibility of a single block: the case of general $d$}\label{sectioncontr2}
The goal of this Section is to prove Theorems \ref{propthetatop} and \ref{propthetadg}, by a method analogous to the one employed for the case $d=1$ in Section \ref{sectioncontr1}. 

We need some preparation as we have to use homotopy properties of the topological and the dg realization (resp. totalization) of presheaves (resp., copresheaves) on $\Theta_d$. We work with Quillen model categories, and, in particular, with Reedy model structures on diagrams indexed by a Reedy category. 
The Reedy category structure on $\Theta_d$ is due to C.Berger [Be2, Sect. 2], we also refer to [BR] for alternative treatment. 
The Reedy category structure on $\Theta_d$ is {\it elegant}, which implies that the $d$-cocellular topological space $T\mapsto C_T$, where $C_T$ is the standard Joyal topological cell [J], see Section \ref{sectionjoyalcells}, is Reedy cofibrant. This is the system of cells used in topological realization (resp. totalization) of a $d$-cellular (resp., $d$-cocellular) topological space. 
A general formalism of [BeM, Sect. 9] is applied for realizations (resp., totalizations) of presheaves (resp., precosheaves) on $\Theta_d$ with values in a monoidal model category.
We also essentially use a result of Berger [Be3, Prop. 3.9] saying that the topological realization of a presheaf on $\Theta_d$ is homeomorphic to the standard multi-simplicial realization of the restriction from $\Theta_d$ to $\Delta^{\times d}$. 

\subsection{\sc A Reedy category structure on $\Theta_d$}\label{sectionreedy}
Recall the wreath product definition of the category $\Theta_d$, see \eqref{thetadef}.

Recall that a {\it Reedy category structure} $\mathcal{R}$ is a small category, equipped with two subcategories $\mathcal{R}^+$ and $\mathcal{R}^-$, containing all objects (they are called the direct and the inverse subcategories), and with a degree function 
$\deg\colon \mathcal{R}\to\mathbb{N}$ such that the following conditions hold:
\begin{itemize}
\item[(1)] every morphism $\alpha$ in $\mathcal{R}$ admits a unique factorization $\alpha=\alpha^+\alpha^-$, with $\alpha^+\in \mathcal{R}^+$, $\alpha^-\in\mathcal{C}^-$,
\item[(2)] for every morphism $\alpha\colon c\to d$ in $\mathcal{R}^+$ one has $\deg c\le \deg d$, with $=$ only when $\alpha=\id$,
for every morphism $\alpha\colon c\to d$ in $\mathcal{R}^-$  one has $\deg c\ge \deg d$, with $=$ only 
when $\alpha=\id$.
\end{itemize}

As a consequence, $\mathcal{R}^+\cap\mathcal{C}^-$ consists of the identity morphisms of all objects, and any isomorphism in $\mathcal{R}$ is an identity morphism.  (Indeed, assuming $\alpha=f_+f_-$ is an isomorphism, $f_\pm\in\mathcal{C}^\pm$, one finds $\alpha^{-1}=g_+g_-$, $g_\pm\in\mathcal{C}^\pm$, that is $f_+f_-g_+g_-=\id=f_+g_+^\prime f_-^\prime g_-=\id\circ\id$, and from the uniqueness of the decomposition it follows that $f_+$ and 
$g_-$ are identity morphisms, similarly $f_-$ and $g_+$ also are). 

The standard example of a Reedy category is the simplicial category $\Delta$, where $\deg [n]=n$, and $\Delta^+$ (resp. $\Delta^-$) is the subcategory generated by the face (resp., the degeneracy) maps. 

Also if $\mathcal{R}$ is a Reedy category, $\mathcal{R}^\op$ is a Reedy category, with the same $\deg$ function, and with $(\mathcal{R}^\op)^\pm=\mathcal{R}^\mp$. Also, if $\mathcal{R}_1$ and $\mathcal{R}_2$ are Reedy categories, the product $\mathcal{R}_1\times\mathcal{C}_2$ also is, with the degree function $\deg(c_1,c_2)=\deg c_1+\deg c_2$, $(\mathcal{R}_1\times\mathcal{C}_2)^\pm=\mathcal{R}_1^\pm\times \mathcal{R}_2^\pm$.

A less trivial example is the category $\Theta_2=\Delta\wr\Delta$.
An object of $\Theta_2$ is a tuple $([m]; [\ell_1],\dots,[\ell_m])$ and a morphism 
$\Phi\colon ([m]; [\ell_1],\dots, [\ell_m])\to ([n]; [k_1],\dots, [k_n])$ is given by $(\phi, \phi_i^j)$ where 
$\phi\colon [m]\to [n]$ a morphism in $\Delta$, and for a given $1\le i\le m$, $\phi_i^j\colon [\ell_i]\to [k_j]$ a morphism in $\Delta$, where $j=\phi(i-1)+1,\dots,\phi(j)$. 

We say that $\Phi\in \Theta_2^-$ if $\phi\in \Delta^-$, and for each $i$ such that $\phi(i-1)\ne \phi(i)$, the maps $\phi_i^j\in \Delta^-$.
The direct subcategory $\Theta_2^+$ is more tricky: we require $\phi\in \Delta^+$, and for each $i$ the maps $\{\phi_i^j\}_{j=\phi(i-1)+1\dots \phi(i)}$ are {\it jointly injective}, meaning that for any two element $a\ne b\in [\ell_i]$ there is at least one $j$, $\phi(i-1)+1\le j\le \phi(i)$ for which $\phi_i^j(a)\ne \phi_i^j(b)$. The fact that it is a Reedy category is firstly proven in [Be2, Sect.2].

Berger [Be2] gives a similar construction for $\Theta_d$ for any $d$. Bergner and Rezk [BR] developed an approach via {\it multi-Reedy categories} which makes this construction more transparent. The answer is as follows.

We have $\Theta_d=\Delta\wr(\Delta\wr\dots (\Delta\wr\Delta)\dots)$ ($d$ factors). 

An object of $\Theta_d$ is given by tuple $\underline{m}=([m]; \{[m_i]\}; \{[m_{ij}]\};\dots,\{[m_{i_1\dots i_{d-1}}]\})$. 
Its degree is equal to
$$
\deg\underline{m}=m+\sum_{i}m_i+\sum_{i,j}m_{ij}+\dots+\sum_{i_1,\dots,i_{d-1}}m_{i_1i_2\dots i_{d-1}}
$$
A morphism $\Phi\colon \underline{m}\to\underline{n}$ is given by the following maps in $\Delta$:
\begin{description}
\item a morphism $\phi\colon [m]\to [n]$,
\item (multi-)morphisms $\phi_i\colon [m_i]\to\{[n_j\}_{j=\phi(i-1)+1\dots\phi(i)}$ (here a multi-morphism is just a sequence of morphisms from the same source but with different targets), denote the $j$-th component of the multi-morphism by $\phi_i^j$,
\item (multi-)morphisms $\phi_{i_1i_2}\colon [m_{i_1i_2}]\to \{[n_{j_1j_2}]\}_{j_1=\phi(i_1-1)+1\dots\phi(i_1), j_2=\phi_{i_1}^{j_1}(i_2-1)+1\dots\phi_{i_1}^{j_1}(i_2)}$, denote the $(j_1,j_2)$-component of the multi-morphism by $\phi_{i_1i_2}^{j_1j_2}$,
\item $\dots$
\item (multi-)morphisms $\phi_{i_1\dots i_{d-1}}\colon [m_{i_1\dots i_{d-1}}]\to\{[n_{j_1\dots j_d}]\}_{\dots}$
\end{description}
We say that such a morphism $\Phi$ belongs to $\Theta_d^-$ if each of the morphisms $\phi_{i_1\dotsi_k}^{j_1\dots j_k}\colon [m_{i_1\dots i_k}]\to [n_{j_1\dots j_k}$ belongs to $\Delta^-$ unless the range for $j_1,\dots,j_k$ is not empty, $k=0,\dots, d-1$.
We say that $\Phi\in \Theta_d^+$ if for each $k=0,\dots,d-1$ and each multi-index $\{i_1,\dots, i_k\}$ the components of the multimorphism $\phi_{i_1\dots i_k}$ are {\it jointly injective}, in the sense explained above.

For a proof that in this way one gets a Reedy category structure the reader is referred to [Be2, Sect. 2] or [BR, Sect. 2]. 

Next, the Reedy category $\Theta_d$ is {\it elegant}, [Be2, Sect.2], [BR, Sect. 3,4]. This property says that for any presheaf $X$ on an elegant Reedy category $\mathcal{R}$ and for any element $x\in X(c)$ there is a {\it unique} morphism $\phi\colon c\to d$ in $\mathcal{R}^-$ and a unique element $y\in X(d)$ such that $x=\phi^*y$, and $y$ is {\it non-degenerate}, meaning that it is not an image of $z\in X(e)$ under a morphism $d\to e$ in $\mathcal{R}^-$ unless the morphism $d\to e$ is identity. This property is very useful when one considers (topological) realizaions of presheaves, and want they to be $CW$-complexes. 
The property follows from a generalised Eilenberg-Zilber lemma [BR, Sect. 4] (see [GZ, Sect. II.3] for the classical case of $\Delta$), which in general Reedy category is a condition which implies the elegancy. 

Let $\mathcal{M}$ be a Quillen model category [Q], [Hi], [Ho], $\mathcal{R}$ a Reedy category. Then the diagram category $\mathcal{M}^{\mathcal{R}}$ is again a model category whose weak equivalences are point-wise weak equivalences of the $\mathcal{R}$-diagrams, called the Reedy model structure [Ree], [Hi, Ch.15]. The cofibrations and the fibrations for Reedy model structure on $\mathcal{M}^\mathcal{R}$ are as follows. 

For an object $X\in \mathcal{M}^{\mathcal{R}}$ and an object $c\in \mathcal{R}$ define the {\it latching object} $L_cX$ as
$$
L_cX=\colim_{d\to c\in \partial(\mathcal{R}^+/c)}X_d
$$
where $\partial(\mathcal{R}^+/c)$ is the category $\mathcal{R}^+/c$ with excluded identity morphism.  The universal property of colimit morphism 
$$
L_cX\to X_c
$$
called {\it the latching map} at $c\in\mathcal{C}$.

Similarly, define the {\it matching object} $M_c$ as
$$
M_cX=\lim_{c\to d\in \partial({c/\mathcal{R}^-})}X_d
$$
and the universal property of limit gives a morphism
$$
X_c\to M_cX
$$
called the {\it matching map} at $c$. 

Let $f\colon X\to Y$ be a map in $\mathcal{M}^{\mathcal{R}}$.

The map $f$ is called {\it a Reedy cofibration} if for any $c\in\mathcal{C}$ the relative latching map
$$
X_c\amalg_{L_cX}L_cY\to Y_c
$$
is a cofibration in $\mathcal{M}$.

The map $f$ is called {\it a Reedy fibration} id for any $c\in\mathcal{C}$ the relative matching map
$$
X_c\to Y_c\Pi_{M_cY}M_cX
$$
is a fibration in $\mathcal{M}$.

For any Reedy category $\mathcal{R}$, the category of diagrams $\mathcal{M}^\mathcal{R}$ with object-wise weak equivalences, Reedy cofibrations, and Reedy fibrations is a Quillen model category.

\comment
??? Below in ??? we prove a more explicit formulas for the latching and the matching objects for $\mathcal{R}=\Theta_d$, analogous to the well-known formula for the latching (resp., the matching objects )??? for the case of $\mathcal{R}=\Delta$, which uses only the face (resp., the degeneracy) maps changing the degree by 1 and 2. 
\endcomment

We will need a cofinal subcategory $\mathcal{S}^\vee_T\subset \partial(T/\Theta_d^-)$ so that the matching object $M_TX$ of a functor $X\colon \Theta_d\to\mathcal{M}$ can be expressed as a limit
$$
M_TX=\lim_{T\to T^\prime\in \mathcal{S}^\vee_T}X_{T^\prime}
$$
over a ``smaller'' category. 

Let $T\in \Theta_d$ be a $d$-level tree, see Definition \ref{defntree}. In general it can be degenerate, that is, some of ordinals $t_j$ in \eqref{eqntree} can be empty ordinals. Denote by $S^\vee_T$ the full subcategory in $\partial(T/\Theta_d^-)$ whose objects 
$\phi\colon T\to T^\prime$ are the morphisms in $\Theta_d^-$ defined as follows (here we think of such morphisms as maps of disks $\overline{T}^\prime\to\overline{T}$; the morphisms in $\Theta_d^-$ correspond to {\it face maps} $\overline{T}^\prime\to\overline{T}$, so they can be considered as face maps $T^\prime\to T$ of the original $d$-level trees).  The objects $\phi$ we consider lower $\Theta_d$-degree by 1 or by 2. If $a$ is a leaf of $T$ at level $k$ (that is, in notations of Definition \ref{defntree}, $i_k^{-1}(a)$ is an empty set), we allow $\phi$ which just removes $a$ (getting $T^\prime$) and embeds $\overline{T}^\prime\to \overline{T}$ (in the dual language, it is the contraction of the corresponding elementary interval of the corresponding ordinal). Such morphisms for different $a$ exhaust all objects $\phi\colon T\to T^\prime$ in $\mathcal{S}^\vee_T$ lowering $\Theta_d$-degree by 1 (note however that different leaves $a$ may be at different levels). The objects $\phi\colon T\to T^\prime$ lowering the $\Theta_d$-degree by 2 are obtained by such removal of two different leaves $a,a^\prime$ of $T$ (note that saying that is not the same as removing a leaf $a$ followed by removing a leaf of $T\setminus\{a\}$). 

\comment
Assume that $0\le j\le d$ is the maximal number such that the ordinal $t_j$ is non-empty. Consider the map of ordinals $i_{j-1}\colon t_j\to t_{j-1}$. Thus we have the set of 1-ordinals $\{i_{j-1}^{-1}(a)\}_{a\in |t_{j-1}|}$, some of which may be empty (but not all, by the assumption made on $j$). These ordinals are {\it the ordinals of elementary intervals} of 1-ordinals $[\ell_1],\dots, [\ell_N]$, see Remark \ref{remthetaordinals}. Consider the maps $\Phi\colon T\to T^\prime$ in $\partial(T/\Theta_d^-)$ whose image under the projection $\Theta_j\to\Theta_{j-1}$ is the identity map, and which are thus $\Delta$-degeneracy operators on each ordinal $[\ell_s]$ at level $j$. Define the category $\mathcal{S}_T^\vee$ as the full subcategory of $\partial(T/\Theta_d^-)$ whose objects are maps $\Phi\colon T\to T^\prime$ of this type, such that the total shift of degree by $\Phi$ is equal to 1 or 2. That is, a morphism $\Phi$ is given either (1) by {\it elementary} degeneracy maps in two distinct ordinals $[\ell_a], [\ell_b]$, $a\ne b$, and the identity maps in $[\ell_s]$ for $s\ne a,b$, or (2) by a composition of {\it two elementary} degeneracy maps in some $[\ell_a]$, and the identity maps in $[\ell_s]$ for $s\ne a$. 
\endcomment

\begin{prop}\label{propmatchingcofinal}
The subcategory $i_T\colon \mathcal{S}_T^\vee\subset \partial(T/\Theta_d^-)$ is cofinal, that is, for any functor $X\colon \Theta_d\to\mathcal{M}$ the natural map 
$$
\lim_{\partial(T/\Theta_d^-)}X\to \lim_{\mathcal{S}_T^\vee}i^*X
$$
is an isomorphism.
\end{prop}
\begin{proof}
By [ML, IX.3], we have to prove the following:

For any object $A=(T\to T^\prime)\in \partial(T/\Theta_d^-)$, the comma category $\mathcal{S}_T^\vee/A$ is non-empty and connected. 

This is elementary and is left to the reader.
\end{proof}

\begin{remark}{\rm
The statement of Proposition \ref{propmatchingcofinal} would be false if we instead defined $\tilde{S}^\vee_T$ with the same objects lowering the degree by 1, but for the objects lowering degree by 2 we allowed removing a leaf $a$ of $T$ followed by removing a leaf in $T\setminus\{a\}$. In this way, some categories $\mathcal{S}_T^\vee/A$ may be not connected. 
}
\end{remark}

For further reference, note a simple lemma, which describes all $T$ for which the category $S^\vee_T$ has a single object. 
\begin{lemma}\label{lemmaoneobject}
Assume $T$ is a $d$-level tree such that the category $S^\vee_T$ has a single object. Then $T$ is a linear tree, having all ordinals $[0]$ in Definition \ref{defntree}, up to some level $k\le d$, and empty above level $k$.
\end{lemma}

It is clear.

\qed

\subsection{\sc Standard system of $d$-cells in $\Top$ and $C^\udot(\mathbb{Z})$, a realization and a totalization of a $d$-(co)cellular space}\label{sectionrt}

\subsubsection{\sc General theory}\label{sectionbem}
It is well-known that for a simplicial topological space $X$ its topological realization 
$$
|X|=X_n\otimes_{\Delta}\Delta_n^\top
$$
has the following homotopy invariance: for two {\it Reedy cofibrant} simplicial topological spaces $X,Y$ and a map of simplicial topological spaces $f\colon X\to Y$ such that each component $f_n\colon X_n\to Y_n$ is a weak equivalence, the induced map of realizations 
$$
|f|\colon |X|\to |Y|
$$
is a weak equivalence as well, see [Ree].

Recall that the totalization $\Tot X$ of a cosimplicial topological space $X^\udot$ is defined as
$$
\Tot X=\underline{\Hom}_{\Delta}(\Delta^n_\top,X)
$$
where $\underline{\Hom}_\Delta$ denotes the $\Top$-enriched natural transformations. 
Assume that $X,Y$ are {\it Reedy fibrant} cosimplicial topological spaces, and $f\colon X\to Y$ is a component-wise weak equivalence. Then the corresponding map
$$
\Tot(f)\colon \Tot X\to\Tot Y
$$
is a weak equivalence as well.

A key point in the statements like these two is that the assignment $[n]\mapsto \Delta^n_\top$ is a Reedy cofibrant cosimplicial space. It simply the fact that for each $n\ge 0$ the embedding $\partial\Delta^n\to\to\Delta^n_\top$ is a topological cofibration. 

Similar statements hold when one considers the model category $C^\udot(\mathbb{Z})$ with projective model structure instead of $\Top$. In fact, when we work with (co)simplicial objects, the target model category $\mathcal{M}$ has to be a simplicial model category. This point of view is adopted in [Ree], [Hi]. 

However, when one works with diagrams indexes by a general Reedy category, it is natural to require that $\mathcal{M}$ is a monoidal model category in the sense of Hovey [Ho]. The corresponding formalism on homotopy properties of a realization (resp., a totalization) in this setting was developed in [BeM, Sect. 9]. Here we briefly recall the main statements. 

Let $\mathcal{M}$ be a monoidal model category, $\mathcal{R}$ a Reedy category, $C\colon \mathcal{R}\to \mathcal{M}$ a functor, for $a\in\mathcal{R}$ denote by $C^a$ the value of $C$ at $a$. For any functor $X\colon\mathcal{R}^\op\to\mathcal{M}$ denote by $X_a$ the value of $X$ at $a\in\mathcal{R}$. Define the realization $|X|_C$ as
\begin{equation}\label{realbem}
|X|_C=X\otimes_{\mathcal{R}}C=\mathrm{coeq}\big(\coprod_{a\to b\in\mathcal{R}}X_b\otimes C^a\rightrightarrows \coprod_{a\in \mathcal{R}}X_a\otimes C^a\big)
\end{equation}
where $\otimes$ in the r.h.s. denote the monoidal product at $\mathcal{M}$. 

Similarly, for $Y\colon\mathcal{R}\to\mathcal{M}$, define the totalization $\Tot_CY$ as
\begin{equation}\label{totbem}
\Tot_CY=\underline{\Hom}_\mathcal{R}(C,Y)=\mathrm{eq}\big( \prod_{a\in\mathcal{R}}\underline{\Hom}(C^a,Y^a)\rightrightarrows \prod_{a\to b\in\mathcal{R}}\underline{\Hom}(C^a,Y^b)  \big)
\end{equation}
where $\underline{\Hom}$ stands for the internal Hom which takes values in $\mathcal{M}$. 

Berger and Moerdijk [BeM, Sect.9] give the following definition.
\begin{defn}\label{defbem}{\rm
Let $\mathcal{R}$ be a Reedy category, $\mathcal{M}$ a monoidal model category, $C\colon \mathcal{R}\to\mathcal{M}$. One says that $C$ is strong monoidal (resp., $h$-monoidal) if the realization functor $X\mapsto |X|_C$ is strong symmetric monoidal (resp. $h$-monoidal, that is it is a symmetric monoidal functor which defines an isomorphism on the monoidal units, and for any two Reedy cofibrant objects $X,Y\in\mathcal{M}^{\mathcal{R}^\op}$ the map $|X|_C\otimes |Y|_C\to |X\otimes Y|_C$ is a weak equivalence.}
\end{defn}

For a presheaf $X\colon \mathcal{R}^\op\to\Sets$ denote by $X_\mathcal{M}\colon \mathcal{R}^\op\to\mathcal{M}$ the functor defined component-wise on sets, and for a set $S$ as $\amalg_{s\in S}I$ where $I$ is a monoidal unit. Assume that the unit $I$ is cofibrant, then the functor $X\mapsto X_\mathcal{M}$ is strong monoidal. 
One has:
\begin{prop}\label{propbem}
The following statements are true:
\begin{itemize}
\item[(1)] Assume that $C\colon \mathcal{R}\to\mathcal{M}$ is Reedy cofibrant. Then there is a Quillen adjunction 
$$
|-|_C\colon \mathcal{M}^{\mathcal{R}^\op}\rightleftarrows\mathcal{M}\colon \underline{\Hom}(C,-)
$$
with $|-|_C$ the left adjoint. Similarly, there is a Quillen adjunction
$$
-\otimes_\mathcal{R}C\colon\mathcal{M}\rightleftarrows \mathcal{M}^{\mathcal{R}}\colon \Tot_C(-)
$$
with $\Tot_C(-)$ the right adjoint.
\item[(2)] Assume $C$ is Reedy cofibrant and strong monoidal (resp., $h$-monoidal). Then the functor $|-|_C\colon\mathcal{M}^{\mathcal{R}^\op}\to\mathcal{M}$ is strong monoidal (resp., $h$-monoidal) left Quillen.
\item[(3)] Assume the monoidal unit is cofibrant. Then $C$ is strong monoidal (resp., $h$-monoidal) if and only if for any two representable sheaves $h_a=\mathcal{R}(-,a)$, $h_b=\mathcal{R}(-,b)$, $a,b\in\mathcal{R}$ there are associative and symmetric compatible isomorphisms (resp. weak equivalences)
$$
|(h_a)_\mathcal{M}|_C\otimes |(h_b)_\mathcal{M}|_C\to |(h_a)_\mathcal{M}\otimes (h_b)_\mathcal{M}|_C
$$
\end{itemize}
\end{prop}
\begin{proof}
(1): see [BeM], Lemma 9.8. (2): it is a direct consequence of Definition \ref{defbem} and of (1). (3): see [BeM], prop. 9.3.
\end{proof}

\begin{remark}\label{rembem}{\rm
In fact, [BeM] gives also a definition of an $h$-constant cosimplicial object $C\colon\mathcal{R}\to\mathcal{M}$ as an object on which the cosimplicial operators act as weak equivalence. They call $C$ {\it a standard system of cells} a cosimplicial object $C$ as above which is Reedy cofibrant, $h$-monoidal, and $h$-constant. Also the system of cells in $\Top$ for the Reedy category $\Theta_d$, considered below, is standard, we do not its full consequence, see [BeM, Prop. 9.13]. 
}
\end{remark}

\subsubsection{\sc The case of the Reedy category $\Theta_d$, $d\ge1$}\label{sectionjoyalcells}
Now we introduce the system of cells $C\colon \Theta_d\to\Top$ which we in definition of the topological realization of objects in $\Sets^{\Theta_d^\op}$ and of the topological totalization of objects in $\Sets^{\Theta_d}$. We show that this system of cells is Reedy cofibrant and strongly monoidal, so all statement of Proposition \ref{propbem} are applied. 

Let $T\in \Theta_d$, we recall a convex contractible topological subspace $C_T$ of a Euclidean space, such that $T\mapsto C_T$ gives rise to a functor $C\colon \Theta_d\to\Top$. The construction is due to Joyal [J] and was further studied by Berger [Be1,2]. 

We consider $T$ as a $d$-level tree, to which one associates a $d$-disk $\overline{T}$ in $\Sets$, see Section \ref{section11}.
By Proposition \ref{bergerprop}, $\Disk_d\simeq \Theta_d^\op$. Recall also a topological $\omega$-disk $B$, see Section \ref{section11}, which we may consider as a topological $d$-disk by truncating the higher levels, which we also denote by $B$.

Define $$C_T=\underline{\Hom}_{\Disk_d}(\overline{T},B)$$
where $\underline{\Hom}$ stands for $\Top$-enriched $\Hom$, so $C_T$ can be interpreted as enriched natural transformations.
By definition, $C_T$ gives rise to a functor $\Disk_d^\op\to\Top$, so by Proposition \ref{bergerprop}, it gives rise to a functor $C\colon \Theta_d\to \Top$.

We consider the realization $|X|_C$ of a $d$-cellular set $X$, defined via the system of cells $C$, see \eqref{realbem}.

One can describe cells $C_T$ explicitly [Be2, Prop. 2.6, Ex.2.7], as follows. 

For a $d$-level graph $T$, denote by $e(T)$ the set of edges of $T$ and by $v(T)$ the set of vertices of $T$ excluding the root.  One can consider a point of $[-1,1]^{|e(T)|}$ as a map 
$$
t\colon e(T)\to [-1,1]
$$
For $\alpha\in e(T)$ denote $t(\alpha)$ by $t_\alpha$. For $\alpha,\beta\in e(T)$ we say that $\alpha$ precedes $\beta$ and denote it by $\alpha<\beta$, if $\alpha,\beta$ belong to the same level, and $\alpha<\beta$. For a vertex $x$ of $T$ denote by $[x]\subset e(T)$ the linear set of edges which are below $x$, denote by $\height(x)$ the number of elements in $[x]$.
Then $C_T$ can be seen to be the following subspace of the cube $[-1,1]^{|e(T)|}$:
\begin{equation}
C_T=\{t\colon e(T)\to [-1,1]|\ t_\alpha\le t_\beta \text{  if  }\alpha<\beta, \ (t_\alpha)_{\alpha\in [x]}\in B^{\height(x)}={
\{\sum_{\alpha\in [x]}t_\alpha^2\le 1\} }\text{  for any  }x\in v(T)\}
\end{equation}
\begin{example}\label{examplecell}{\rm
\begin{itemize}
\item[1.]
For the $d$-level tree $T$ given by the chain of maps of ordinals
$$
[n-1]\to [0]\to\dots\to [0]
$$
one has 
$$
C_T=\{(t_\alpha)\in  [-1,1]^{n+d-1}|\  t_1\le t_2\dots\le t_n, t_i^2+t_{n+1}^2+\dots+t_{n+d-1}^2\le 1, 1\le i\le n\}
$$
\item[2.] For $d=2$ and $T$ as in Figure \ref{figex} one has
$$
C_T=\{(t_\alpha)\in [-1,1]^4|\ t_1\le t_2, t_3\le t_4, t_1^2+t_3^2\le 1, t_2^2+t_3^2\le 1\}
$$
\sevafigc{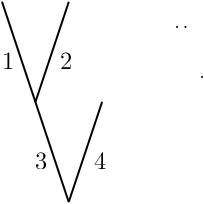}{30mm}{0}{\label{figex}}

\end{itemize}
}
\end{example}
It follows easily that for a representable presheaf $\Theta_d[T]:=\Theta_d(-, T)$ one has 
\begin{equation}\label{realyon}
|\Theta_d[T]|_C=C_T
\end{equation}

The following Proposition is due to C.Berger [Be2, Prop. 2.6, Prop. 2.8]:
\begin{prop}\label{propberger2}
The following statements are true:
\begin{itemize}
\item[(1)] For $T\in \Theta_d$, the cell $C_T$ is a convex closed subspace of $[-1,1]^{|e(T)|}$ with non-empty interior and thus is homeomorphic to a closed disc $B^{|e(T)|}$. There is a CW complex structure on $C_T$ which cells of codimension $i$ are in 1-to-1 correspondence with objects in $\Theta_d^+/T$ shifting the degree by $i$. In particular, the boundary $\partial C_T\simeq  S^{|e(T)|-1}$ has a structure of a CW-complex, whose cells of dimension $|e(T)|-i$ are in 1-to-1 correspondence which objects in the category $\partial(\Theta_d^+/T)$ raising the degree by $i$. A cell $\sigma_1$ is a subset of a cell $\sigma_2$ iff there is a morphism in $\partial(\Theta_d^+/T)$ between to the corresponding objects. For any two objects of $\partial(\Theta_d^+/T)$ there is at most one morphism. Thus, $\Theta_d^+/T$ is a poset, and the realization of the poset $\partial(\Theta_d^+/T)$ is homeomorphic to $S^{|e(T)-1|}$. Any morphism in $\Theta_d^+/T$ is a composition of morphisms raising degree by 1. 
\item[(2)] The $d$-cocellular topological space $T\mapsto C_T$ is Reedy cofibrant.
\item[(3)] For $T_1,T_2\in \Theta_d$, one has associative and commutative isomorphisms 
$$
|\Theta_d[T_1]|_C\times |\Theta_d[T_2]|_C\simeq |\Theta_d[T_1]\times \Theta_d[T_2]|_C
$$
In particular, the system of cells $C$ is strong monoidal, in the sense of Definition \ref{defbem}.
\end{itemize}
\end{prop}
\begin{proof}
We sketch the arguments, due to C.Berger. 

(1): The method employed in the proof is beautiful as it does not rely on any computation of relations between the ``elementary face maps" for $\Theta_d$, mimicking the ones for the case of $\Delta=\Theta_1$, which would be really heavy (for $d=2$, we provided such relations in [PS]). Contrary, the proof mixes up categorical and geometrical considerations. 

First of all, the statement that $C_T$ is a convex subset of the cube is clear, so the boundary is indeed homeomorphic to a sphere $S^{|e(T)|-1}$. Next, $C_T=|\Theta_d(T)|_C$. The further argument essentially uses the fact that $\Theta_d$ is an elegant Reedy category, which is proven by an analogue of the Eilenberg-Zilber lemma [GZ, II.3], [BR, Prop. 4.2, Prop. 4.4]. 
It follows that for any presheaf $X$ on $\Theta_d$, and for any element $x\in X(T)$  there is a {\it unique} morphism $\phi\colon T\to T_1$ in $\Theta_d^-$, and a unique $y\in X(T_1)$ which is non-degenerate and such $x=\phi^*y$. It follows from this property that  for any $X$ and for any $C$ such that each $C(T)$ is a closed subset homeomorphic to a disk, such that any point of $\partial C_T$  is the image $\phi_*y$ by some operator of $\phi\colon T^\prime\to T\in \Theta_d^+$ and $y\in C_{T^\prime}$, any point in the realization $|X|_C$ has a unique representative $(x\in X(T), \omega\in C(T))$ such that $x$ is non-degenerate and $\omega$ belongs to the interior. It gives a stratification of $X_T$ by open strata corresponded to the non-degenerate-interior points. 

Consider this stratification for $X=\Theta_d[T]$. The non-degenerate elements in $X(T^\prime)$ are in 1-to-1 correspondence with morphisms $\phi\colon T^\prime\to T$ in $\Theta_d^+$. It corresponds to an open stratum homeomorphic to $\Int(C_{T^\prime})$ of codimension $s$ if the morphism $\phi$ raises the degree by $s$. The incidence of the (closed) cells are in 1-to-1 correspondence with the morphisms between the corresponding objects in $\Theta_d^+/T$. It follows that for any two objects in $\Theta_d^+/T$ there is at most one morphism. The boundary $\partial C_T\simeq S^{|e(T)-1|}$ is corresponded to the morphisms in $\partial(\Theta_d^+/T)$. 

For (2), the one has to prove that for any $T\in \Theta_d$ the latching map at $T$ 
$$
\colim_{T^\prime\to T\in\partial(\Theta_d^+/T)}|\Theta_d[T^\prime]|_C\to |\Theta_d[T]|_C
$$
is a cofibration. As the realisation $|-|_C$ is a left adjoint by Proposition \ref{propbem}(1), it commutes with all colimits and the source object is 
$$
|\colim_{T^\prime\to T\in \partial(\Theta_d^+/T)}\Theta_d[T^\prime]|_C=|\partial(\Theta_d^+/T)|_C
$$
which is homeomorphic $S^{|e(T)|-1}$ by (1). The latching map is thus the embedding $\partial C_T\to C_T$, which is a closed embedding of $CW$-complexes. 

(3) follows from Proposition \ref{propbem} (3) and from [Be2, Prop. 2.8].

\end{proof}

\comment

From Proposition \ref{propberger2}(i) one easily deduces
\begin{coroll}\label{corthetalatching}
Let $T\in\Theta_d$. The following statements are true:
\begin{itemize}
\item[(1)] Consider the full subcategory $i\colon \mathcal{S}_T\subset \partial(\Theta_d^+/T)$ whose objects are given by morphisms $T^\prime\to T$ raising the degree by 1 or 2. Then it is a final subcategory. Consequently, for any functor $Y\colon \partial(\Theta_d^+/T)\to \mathcal{E}$ where $\mathcal{E}$ is a finitely cocomplete category, the canonical map
$$
\colim_{\mathcal{S}_T}i^*Y\to \colim_{\partial(\Theta_d^+/T)} Y
$$
is an isomorphism.
\item[(2)] Let $Y\colon \Theta_d\to \mathcal{M}$ be an objects of $\mathcal{M}^{\Theta_d}$. Then the latching object $L_TY$
is identified with the coequalizer
$$
\mathrm{coeq}\left(\coprod_{\substack{{T^{\pprime}\to T\in \Theta_d^+}\\ {\codim_TT^\pprime=2}}}Y_{T^\pprime}\rightrightarrows
\coprod_{\substack{{T^{\prime}\to T\in \Theta_d^+}\\ {\codim_TT^\prime=1}}}Y_{T^\prime}\right)
$$
where the two arrows are corresponded to the two cells of $\codim=1$ which are incident with a given cell of $\codim=2$, and the latching map
$$
L_TY\to Y
$$
is defined via the $\codim=1$ maps $T^\prime\to T$ by the universal property of colimit. 
\end{itemize}
\end{coroll}
\begin{proof}
Prove (1). According to ???, it is enough to prove that for any $t\colon T^\prime\to T\in \partial(\Theta_d^+/T)$ the comma-category $t/\mathcal{S}_T$ is non-empty and connected. 

We interpret this comma-category as follows.
Its object is an embedding of the cell $T^\prime\to T$ in $\partial C_T$ to a cell $Q\to T$ of $\partial C_T$ of codimension 1 or 2. 
A morphism $Q_1\to Q_2$ over $T$ is a map of cells of codimension 1 or 2 under the cell $T^\prime\to T$.

From Proposition \ref{propberger2}(1) and from the fact that $C_T$ has a single top dimension cell it follows that any cell of codimension 2 is a boundary of exactly two cells of codimension 1. 

Consider three cases: (i) $\codim_TT^\prime=1$, (ii) $\codim_TT^\prime=2$, (iii) $\codim_TT^\prime>2$.

In case (i) it is clear that the comma category $t/\mathcal{S}_T$ has a single object and only identity morphism of it. In case (ii), 
$T^\prime$ is a cell of codimension 2 and thus it is a boundary of exactly two cells of codimension 1. The graph of the poset $t/\mathcal{S}_T$ looks like 
$$
\xymatrix{\mb &\mb\ar[l]\ar[r]&\mb}
$$
where the object in the center is $t$, which shows that it is non-empty and connected.

In case (iii), consider all cells $Q\to T$ of codimension 2 for which there exists a map $T^\prime\to Q$ over $T$. It is clear that the set of such cells is non-empty. We have to prove that any two such maps $Q_1\to Q_2$ over $T$ and under $T^\prime$ of cells of codimension 2 they can be connected by a zig-zag of maps each of which is an embedding of a cell of $\codim=2$ into a cell of $\codim=1$. 

(2) is a direct consequence of (1). 
\end{proof}

???

Similar statements holds for matching objects as well, but the argument is easier:
\begin{prop}\label{propthetamatching}
The following statements are true:
\begin{itemize}
\item[(1)]
Let $T\in \Theta_d$, consider the full subcategory  $i\colon \mathcal{S}_T^\vee\subset \partial(T/\Theta_d^-)$ whose objects are morphisms $T\to T^\prime$ which decrease the degree by 1 or by 2. Then the embedding $i$ is cofinal. Consequently, for any functor $X\colon \partial(T/\Theta_d^-)\to \mathcal{E}$, where $\mathcal{E}$ is a finitely complete category, the canonical map
$$
\lim_{\partial(T/\Theta_d^-)}X\to \lim_{\mathcal{S}_T^\vee}i^*X
$$
is an isomorphism.
\item[(2)] For any functor $X\colon \Theta_d\to\mathcal{M}$ the matching object $M_TX$ is isomorphic to the equalizer
$$
\mathrm{eq}\left(\prod_{\substack{{T\to T^\prime\in\Theta_d^-}\\{\codim_TT^\prime=1}}}X_{T^\prime}\rightrightarrows 
\prod_{\substack{{T\to T^\prime\in\Theta_d^-}\\{\codim_TT^\prime=2}}}X_{T^\prime}\right)
$$
and the matching map $X\to M_TX$ is defined via the universal limit property. 
\end{itemize}
\end{prop}
\begin{proof}
We only need to prove (1). 

As we said, the argument is more straightforward than for the case of latching maps, Corollary \ref{corthetalatching}, because the morphisms in $\Theta_d^-$ are defined via any multi-morphisms and the joint injectivity of them; they are defined directly in terms of degeneracy maps in $\Delta$. 
\end{proof}

\endcomment

Now we consider dg realization of $d$-cellular sets, for which we need a system of dg cells $C^\dg\colon \Theta_d\to\C^\udot(\mathbb{Z})$. We define it as follows. 

Let $T\in \Theta_d$. Define $C^\dg_T=C_\ldot^{\mathrm{CW}}(C_T,\mathbb{Z})$ be the $CW$ chain complex of $C_T$. 

We want to make $T\rightsquigarrow C_T^\dg$ a functor $C^\dg\colon \Theta_d\to C^\udot(\mathbb{Z})$.
We know from Proposition \ref{propberger2}(1) that the $i$-cells of $C_T$ are maps $\rho\colon T^\prime\rightarrowtail T$ in $\Theta_d^+$ such that $\deg T^\prime=i$. 
Such $\rho$ has degree $-i$ in $C_T^\dg$ (recall the by the convention we adopt all differentials have degree +1).

Assume $\alpha\colon T\to T_1$ is a morphism in $\Theta_d$. As $\Theta_d$ is an elegant Reedy category, Section \ref{sectionreedy}, see also [Be2, Sect. 2], [BR, Sect. 3.4], the composition 
$$
T^\prime \overset{\rho}{\rightarrowtail}T\xrightarrow{\alpha}T_1
$$
can be decomposed as 
$$
T^\prime \overset{s}{\twoheadrightarrow} T^\prime_1\overset{\rho_1}{\rightarrowtail} T_1
$$
Define $\alpha_*\colon C_T^\dg\to C_{T_1}^\dg$ by mapping $\rho\colon T^\prime \rightarrowtail T$ to
$\rho_1\colon T^\prime_1\to T_1$ if $s$ is the identity morphism, and to 0 otherwise. One has to check that this assignment gives a map of complexes. Indeed, $T^{\pprime}$ is a summand in $\partial(T^\prime)$ (taken with an appropriate sign) 
if there is a map $T^\pprime\rightarrowtail T^\prime$ in $\Theta_d^+$ and $\deg T^\pprime=\deg T^\prime-1$. But if $s=\id$ it gives a map $T^\pprime\rightarrowtail T^\prime_1=T^\prime$, which shows that $\alpha_*$ is a map of complexes.

\begin{prop}\label{propcofthetadg}
The cocellular object $C^\dg$ in $C^\udot(\mathbb{Z})$ is Reedy cofibrant.
\end{prop}
\begin{proof}
We have to show that for any $T\in\Theta_d$ the latching map
\begin{equation}\label{latchingthetadg}
\colim_{T^\prime\to T\in \partial(\Theta_d^+/T)}C^\dg_{T^\prime}\to C_T^\dg
\end{equation}
is a cofibration in the projective model structure on $C^\udot(\mathbb{Z})$. Recall [Ho, 2.3] that a cofibration is a map of complexes which is term-wise injective with projective cokernels in each degree. 

The map \eqref{latchingthetadg} is easily identified with the map
$$
C^{\mathrm{CW}}_\ldot(\partial C_T)\to C^{\mathrm{CW}}_\ldot (C_T)
$$
for which the statement is clear: it is a term-wise embedding, and the cokernel is $\ne 0$ only in degree $-\deg T$, in which it is $\mathbb{Z}[T\xrightarrow{\id}T]$. 

\end{proof}

For $X\colon \Theta_d^\op\to\Sets$ define 
$$
|X|_{\Theta_d,\dg}:=C_T^\dg\otimes_{T\in \Theta_d}X_T
$$
We will need
\begin{lemma}\label{lemmarealthetadgtop}
Let $X$ be as above. Then $|X|_{\Theta_d,\dg}$ is isomorphic to the CW chain complex $C_\ldot^{\mathrm{CW}}(|X|_{\Theta_d,\top})$ of the topological realization of $X$. 
\end{lemma}
\begin{proof}
It simply follows from presentation of $X$ as a colimit of Yoneda presheaves $\Theta_d[T]=\Theta_d(?,T)$ and commutativity with colimits in both formulas; for $X=\Theta_d[T]$ both sides give $C_T^\dg$. 
\end{proof}
We often denote the topological realization (resp., the topological totalization) of a $d$-cellular (resp., $d$-cocellular) set $X$ by $|X|_{\Theta_d,\top}$ (resp., $\Tot_{\Theta_d,\top}(X)$). Similarly, we denote the dg realization (resp., totalization) by $|X|_{\Theta_d,\dg}$ (resp., $\Tot_{\Theta_d,\dg}(X)$). The same notations are used for (co)cellular topological spaces and complexes. 

Now we can state for the case of $\mathcal{R}=\Theta_d$ the general statement of Proposition \ref{propbem}.
\begin{prop}\label{propbemtheta}
The following statements are true:
\begin{itemize}
\item[(1)] There is a Quillen adjunction 
$$
|-|_{\Theta_d,\top}\colon \Top^{\Theta_d^\op}\rightleftarrows \Top\colon \underline{\Hom}(C_-,=)
$$
where $|-|_{\Theta_d,\top}$ is the left adjoint. Similarly, there is a Quillen adjunction 
$$
-\otimes_{\Theta_d}C_-\colon \Top\rightleftarrows \Top^{\Theta_d}\colon \Tot_{\Theta_d,\top}(=)
$$
where $\Tot_{\Theta_d,\top}$ is the right adjoint.
\item[(2)] Endow $C^\udot(\mathbb{Z})$ with the projective model structure. There is a Quillen adjunction 
$$
|-|_{\Theta_d,\dg}\colon (C^\udot(\mathbb{Z}))^{\Theta_d^\op}\rightleftarrows C^\udot(\mathbb{Z})\colon \underline{\Hom}(C^\dg_-,=)
$$
where $|-|_{\Theta_d,\dg}$ is the left adjoint. 
Similarly, there is a Quillen adjunction 
$$
-\otimes_{\Theta_d}C^\dg_-\colon C^\udot(\mathbb{Z})\rightleftarrows (C^\udot(\mathbb{Z}))^{\Theta_d}\colon \Tot_{\Theta_d,\dg}(=)
$$
where $\Tot_{\Theta_d,\dg}(=)$ is the right adjoint.
\item[(3)] The functor $|-|_{\Theta_d,\top}$ is strong monoidal. 
\end{itemize}
\end{prop}
\begin{proof}
(1) and (2) follow from Proposition \ref{propbem} and Propositions \ref{propberger2}(2) and \ref{propcofthetadg}.
(3) follows from Propostion \ref{propberger2}(3) and the computation in [Be2, Prop. 2.8].
\end{proof}

\vspace{2mm}

We will need the following surprising statement, proven in [Be3, Prop. 3.9]. It provides an alternative way to compute the realization $|-|_{\Theta_d,\top}$. 

For any category $\mathcal{A}$, there is a functor 
$$
\delta_{\mathcal{A}}\colon \Delta\times \mathcal{A}\to \Delta\wr\mathcal{A}
$$
defined by
$$
\delta_{\mathcal{A}}([n], A)=([n]; \underset{n \text{ times }}{\underbrace{A,\dots,A}})
$$
$$
\delta_\mathcal{A}(\phi\colon [m]\to [n], f\colon A\to B)=(\phi, \phi_i^j)
$$
where for any $1\le i\le n$, $\phi_i^j=f\colon A\to B$, $j=\phi(i-1)+1\dots\phi(i)$, unless $\phi(i)\ne\phi(i+1)$.

As $\Theta_d=\Delta\wr\Theta_{d-1}$ and $\Theta_1=\Delta$, one gets a functor
\begin{equation}\label{wreathdiag}
\delta_d^\wr\colon \Delta^{\times d}\to\Theta_d
\end{equation}s
The following statement is a direct corollary of [Be3, Prop. 3.9], we state it in the cases we use, see loc. cit. for the general statement.
\begin{prop}\label{propberger3}
Let $X\colon \Theta_d^\op\to\mathcal{E}$ where $\mathcal{E}=\Sets$ or $\mathcal{E}=\Top$. Then the realization $|(\delta_d^\wr)^*X|_{\Delta^{\times d}}$ of the polysimplicial set $(\delta_d^\wr)^*X$ is homeomorphic to the cellular realization $|X|_{\Theta_d}$. Both realizations commute with all colimits and with finite limits. 
\end{prop}

\qed

Clearly the two realizations have, by their definitions, different combinatorial (CW complex) structures, but Proposition states that as topological spaces they are {\it homeomorphic}.

\comment
\vspace{2mm}

Finally, we will use the result [Be2, Th. 3.9] on Quillen model structure on the category $d$-cellular sets.
\begin{prop}\label{propbergermodel}
There is a closed model structure on the category of $d$-cellular sets $\hat{\Theta}_d$, whose weak equivalences are realization weak equivalences and whose cofibrations are monomorphisms, and a Quillen pair
$$
L\colon \hat{\Theta}_d\rightleftarrows \Top\colon R
$$
where $L(X)=|X|_{\Theta_d}$ and $R(Y)(T)=\Hom(C_T,Y)\in \Sets$. Moreover, this Quillen pair is a Quillen equivalence.
\end{prop}
Proposition implies that $\Theta_d$ is a {\it test category}. We denote $R(Y)=\Sing_{\Theta_d}(Y)$. 

We will need the following consequence from the statement:
\begin{coroll}\label{corolthetareal1}
Let $X\in \hat{\Theta}_d$. Then the adjunction morphism $\epsilon\colon X\to \Sing_{\Theta_d}(|X|_{\Theta_d})$ is a weak equivalence of $d$-cellular sets.
\end{coroll}
\begin{proof}
Any $d$-cellular set is cofibrant and any topological space is fibrant. Then the statement is equivalent to to the statement that the  adjoint to $\epsilon$ map $|X|_{\Theta_d}\xrightarrow{\id}|X|_{\Theta_d}$ is a weak equivalence of topological spaces. 
\end{proof}

\begin{coroll}\label{corolthetareal2}
The map of $d$-cellular sets $X\to \Sing_{\Theta_d}|X|_{\Theta_d, \top}$ induces a quasi-isomorphism on their dg realizations. 
\end{coroll}
\begin{proof}
The statement of  Corollary \ref{corolthetareal1} means that the topological realizations $$|X|_{\Theta_d,\top}\sim |\Sing_{\Theta_d}|X|_{\Theta_d,\top}|_{\Theta_d,\top}$$ are homotopy equivalent. Then the statement follows from Lemma \ref{lemmarealthetadgtop}.
\end{proof}

\endcomment

\subsection{\sc Proofs of Theorems \ref{propthetatop} and \ref{propthetadg}}
We apply the same strategy which was employed for the case $d=1$ (of the category $\Delta$) in Section \ref{sectioncontr1}.
However, the proofs are technically more involved, and rely on the preliminary results stated or proven in Sections \ref{sectionreedy}-\ref{sectionrt} above. 

\subsubsection{\sc The topological condensation}
Let $d\ge 1$, $n\ge 1$ be fixed, $(\bmu,\bsigma)\in \mathcal{K}(k)^{\times d}$,

\begin{equation}\label{eqbmu}
(\bmu,\bsigma)=((\mu^1,\sigma_1),(\mu^2,\sigma_2),\dots,(\mu^d,\sigma_d))
\end{equation}
(see Section \ref{sectiongenlp2} for notations). 

The following statement is Theorem \ref{propthetatop}(i). 
\begin{prop}\label{propmsd}
For any $T\in \Theta_d$, the topological realization $|\mathcal{L}_{(\bmu,\bsigma)}(T)|$ is contractible.
\end{prop} 
\begin{proof}
We prove contractibility level by level, starting with the highest levels, using Lemma \ref{lemmac} and the argument similar to the one in Proposition \ref{propms}. Thus, we employ induction on the number $d$ of levels, inside which there is another induction by the number $k$ of arguments. 

We can assume that $\sigma_d=1$ in \eqref{eqbmu}. Indeed, one has an isomorphism of functors
$$
\mathcal{L}_{(\bmu,\bsigma)}\simeq\mathcal{L}_{(\bmu^\prime,\bsigma^\prime)}
$$
where $(\bmu^\prime,\bsigma^\prime)=((\sigma_d^{-1}\mu^1,\sigma_d^{-1}\sigma_1),(\sigma_d^{-1}\mu^2,\sigma_d^{-1}\sigma_2),\dots, (\sigma_d^{-1}\mu^d,\id))$.

Next, by Proposition \ref{bergerprop3}, the realization $|\mathcal{L}_{(\bmu,\bsigma)}(T)|$ is homeomorphic to the standard topological realization of the
polysimplicial set $(\delta_d^{\times k})^*\mathcal{L}_{(\bmu,\bsigma)}(T)$, which is the restriction along the functor $(\delta_d)^{\times k}\colon ((\Delta^\op)^{\times d})^{\times k}\to(\Theta_d^\op)^{\times k}$. 

We start with contraction along the factor $(\Delta^\op)^{\times k}$ corresponded to the highest level of the disks. Here the argument just repeats the one in Proposition \ref{propms}. 
The only difference is in computation of $\pi_0(-)$, it is reduced, by induction on $k$, to the elementary computation $\pi_0(-)=*$ for the following simplicial set: 
$\Delta(?, [a_1])\times \Delta(?,[a_2])\times \dots\times\Delta(?,[a_\ell])$, where $a_1,\dots, a_\ell$ correspond to the 1-ordinals at the highest level $d$. 

The same speculation, by induction on the number $k$ of arguments, reduces the statement to the contractibility of the realization
$|\mathcal{L}_{(\mu^1,\sigma_1),\dots, (\mu^{d-1}, \sigma_{d-1})}(\tau_{\le d-1}T)|$
of $\Theta_{d-1}$-generalised lattice path, where $\tau_{\le d-1}T$ is the $(d-1)$-level tree obtained by the truncation of the highest level. Then we reduces to the case $\sigma_{d-1}=\id$, as above, and induction by $d$ completes the proof. 
\end{proof}
\begin{remark}\label{remcruciald}{\rm
A crucial point in the proof is that the assignment $\omega\rightsquigarrow c(\omega)$ defined as in the proof of Proposition \ref{propms}, leaves the generalised lattice path inside the block $\mathcal{L}_{(\bmu,\bsigma)}(T)$, and thus defines a natural transformation $\beta\colon \mathcal{L}_{(\bmu,\bsigma)}(T)\to \mathcal{L}_{(\bmu,\bsigma)}(T)$. The preservation of $(\bmu,\bsigma)$ is clear, because we defined $\mu^i$ and $\sigma_i$ ``globally'' at level $i$, that is, as the same parameters for all lattice paths ``in the fibres'' at this level. The requirement of preservation of a cell by $c$ dictates us, in fact, how the cells should be defined. 
}
\end{remark}

Next, we prove Theorem \ref{propthetatop}(ii). We need to compute the totalization of the $d$-cocellular topological space 
$$
T\mapsto |\mathcal{L}_{(\bmu,\bsigma)}(T)|
$$
and prove its contractibility.

We know from Proposition \ref{propmsd} that for each $T$ the corresponding topological space is contractible. Then, by Proposition \ref{propbemtheta}(1), it is enough to prove
\begin{prop}\label{propreedyfibranttheta}
The $d$-cocellular topological space $T\mapsto |\mathcal{L}_{(\bmu,\bsigma)}(T)|$ is Reedy fibrant, with respect to the Reedy category structure on $\Theta_d$ defined in Section \ref{sectionreedy}.
\end{prop}
\begin{proof}
We use the statement of Proposition \ref{propmatchingcofinal} which provides a cofinal subcategory $\mathcal{S}^\vee_T\subset 
\partial(T/\Theta_d^-)$. We will see that the proof goes rather similarly to the proof of the case $d=1$ in Proposition \ref{propreedyfibrantdelta}. 

Let $T$ be an object of $\Theta_d$. We have to prove that the matching map 
$$
X_T\to M_TX=\lim_{T\to T^\prime\in \mathcal{S}^\vee_T}X_{T^\prime}
$$
is a fibration in $\Top$, where $X_T=|\mathcal{L}_{(\bmu,\bsigma)}(T)|$. 

\comment
Let $T$ is empty above the level $j$, and is non-empty at level $j$. Let 
$[\ell_1],\dots, [\ell_N]$ be the 1-ordinals at level $j$ of $T$ (which are preimages of $N$ vertices at level $j-1$). By assumption, not all $\ell_i=0$. Denote by $s_i^aT$ the tree equal to $T$ at levels $\le j-1$, and obtained from $T$ by action of the elementary codegeneracy $s^a$ acting on $\ell_i$, $0\le a\le \ell_i-1$ (the other 1-ordinals $[\ell_s]$ remain unchanged). By Proposition \ref{propmatchingcofinal} and its proof, the matching object $M_TX$ is equal to 
\begin{equation}\label{matchingexpl}
M_TX=\mathrm{eq}\left(\prod_{a,i}X_{s_i^aT}\rightrightarrows \prod_{i_1\ne i_2}X_{s_{i_1}^as_{i_2}^bT}\ \times\ \prod_{i; a<b}X_{s_{i}^as_{i}^bT}\right)
\end{equation}
The two arrows are given by suitable operators $s_i^a$, likewise for the case of $\Delta$. The new feature for the case $d>1$ is that in the second product pairs such operators may act on different 1-ordinals $[\ell_i]$ as well as on the same $[\ell_i]$. 
\endcomment

Like in Proposition \ref{reedytot}, we analyse the matching map for $T\mapsto \mathcal{L}_{(\bmu,\bsigma)}(T)$ {\it before} the realization. Recall that we denote by $\delta_d^{\wr}\colon \Delta^{\times d}\to\Theta_d$ the map defined in \eqref{wreathdiag}; denote by $\delta_{k}^{\Theta_d}\colon \Theta_d\to\Theta_d^{\times k}$ the diagonal embedding. 
\begin{prop}\label{reedytottheta}
Let $T\in\Theta_d$, $(\bmu,\bsigma)\in\mathcal{K}(k)^{\times d}$. The following statements are true.
\begin{itemize}
\item[(1)] If $T$ has $\ge$ 2 leaves (possibly at different levels), the matching map
$$
\mathcal{L}_{(\bmu,\bsigma)}(T)\to M_T\mathcal{L}_{(\bmu,\bsigma)}(-)
$$
is an isomorphism.
\item[(2)] If $T$ has a single leaf, that is, $T$ is a linear $d$-level tree truncated at level $\ell\le d$,
the restriction of the matching map to 
$$
(\delta_d^{\wr})^*(\delta_k^{\Theta_d})^*\mathcal{L}_{(\bmu,\bsigma)}(T)\to (\delta_d^{\wr})^*(\delta_k^{\Theta_d})^*M_T\mathcal{L}_{(\bmu,\bsigma)}(-)
$$
is a product $p_1\times\dots\times p_d$ of $d$ maps of simplicial sets $p_i\colon B_i\to A_i$, where $p_i$ corresponds to the level $i$. 
The maps $p_i$ are identity maps for $i<\ell$, the map $p_\ell$ is a Kan fibration of simplicial sets, (for the case $\ell<d$) the map $p_{\ell+1}$ is a projection to a point, (for the case $\ell<d-1$) the maps $p_i$, $i>\ell+1$ are empty. 
\item[(3)] If $T=([0];\varnothing,\dots)$ is the initial object of $\Theta_d$, the matching map is a projection to a point (this case may be considered as the case (2) for $\ell=0$).
\end{itemize}
\end{prop}
\begin{proof}
(1): the argument is similar to the one in the proof of Proposition \ref{reedytot}(3), in particular, we use the interpretation of the action of the codegeneracy operators in a lattice path by ``removing the marked points'', see [BBM, 2.4-2.5]. However, for the case $d>1$ there is a difference (which though does not affect the argument). The difference is that, assuming a leaf $a$ of $T$ is at level $\ell<d$, there is also a ``geometric'' lattice path at level $\ell+1$ without marked points ``growing out of $a$''. The removal of the market points corresponded to $a$ results in two things: (a) removal of the marked point at the corresponding lattice path at level $\ell$, (b) removal of the geometric lattice path without marked points at level $\ell+1$ growing out of the marked point corresponded to $a$.

\comment
As $T$ is empty above level $j$, a generalised lattice path $\omega\colon \overline{T}\to \overline{T}_1\boxx\dots\boxx\overline{T}_k$ is as well non-empty up to level $j$. The operators in \eqref{matchingexpl} act on the level $j$ and do not affect the levels $<j$, that is, each map in \eqref{matchingexpl} is the identity map when restricted to the truncation $\tau_{\le (j-1)}T$. We consider the ``fibre'' lattice paths at level $j$. The cosimplicial operators do not affect the ``geometric lattice paths'' (see the proof of Proposition \ref{reedytot}(3) for the terminology), and the action results in removing a marked point at some of the  lattice paths at level $j$ (or n decreasing its multiplicity by 1 if it was greater than 1). The equalizer condition says that we count only that elements in 
$$
\prod_{a,i}\mathcal{L}_{(\bmu,\bsigma)}(s_i^aT)
$$
which agree when we remove two points (which may be at the same lattice path at level $j$, or at distinct lattice paths at level $j$). 
If the total number of marked points (counted with their multiplicities) is $\ge 2$, one reconstructs all marked points at all lattice paths, so an element in 
$
\prod_{a,i}\mathcal{L}_{(\bmu,\bsigma)}(s_i^aT)
$
obeying the equalizer condition lifts uniquely to a point in $\mathcal{L}_{(\bmu,\bsigma)}(T)$.
\endcomment

If $T$ has at least 2 leaves, the description of the category $S^\vee_T$ shows that the matching map 
$\mathcal{L}_{(\bmu,\bsigma)}(T)\to M_T\mathcal{L}_{(\bmu,\bsigma)}(-)$ for $T$ has at least two factors each of which forgets 1  point corresponded to a leaf $a$ at the  lattice path at level $\ell$, and forgets the geometric lattice path without marked at level $\ell+1$ growing out of this marked point, in compatible way (when we remove two leaves of $T$. If $T$ has at least two leaves, the generalised lattice path is uniquely reconstructed from these projections. 

(2): By Lemma \ref{lemmaoneobject}, we know that $T$ is a linear $d$-level tree truncated at some level $\ell\le d$. 
The statement about factorization of the restriction of thematching map into a product of simplicial maps is clear. Namely,
if $(\bmu,\bsigma)=((\mu_1,\sigma_1),\dots,(\mu_d,\sigma_d))$,  the map $p_i$ are
$$
\begin{cases}
p_i\colon \mathcal{L}_{(\mu_i,\sigma_i)}[1]\to \mathcal{L}_{(\mu_i,\sigma_i)}[1]&\text{  if  }i<\ell\\
p_\ell\colon\mathcal{L}_{\mu_\ell,\sigma_\ell}[1]\to\mathcal{L}_{(\mu_\ell,\sigma_\ell)}[0]&\text{  if  }i=\ell\\
p_{\ell+1}\colon \mathcal{L}_{(\mu_{\ell+1},\sigma_{\ell+1})}[0]\to *&\text{  if  }i=\ell+1
\end{cases}
$$
The statement about $p_\ell$ is the same as in Proposition \ref{reedytot}(2), the statements about other $p_i$ are clear. 
\comment
 We assume that the tree $T\in \Theta_d$ is empty above the level $j$. It follows that any generalized lattice path parametrized by $T$ is empty above level $j$ as well. In fact, we can assume that $j=d$ and $T$ has the ordinal $[0]$ at level $j=d$ (then the only non-trivial fibre of $\overline{T}$ at level $d$ is $[[1]]$). Denote by $T_{(d-1)}$ the restriction of $T$ to the levels up to $d-1$.
\endcomment

\comment
The matching map $p\colon \mathcal{L}_{(\bmu,\bsigma)}(T)\to M_T\mathcal{L}_{(\bmu,\bsigma)}(-)$ is identified, by \eqref{matchingexpl}, with the projection 

\begin{equation}\label{mapprop2}
p\colon \mathcal{L}_{(\bmu,\bsigma)}(T)\to \mathcal{L}_{(\bmu,\bsigma)}(T_{(d-1)})
\end{equation}
It is a map of functors $(\Theta_d^\op)^{\times k}\to\Sets$. First of all, we restrict both sides to $\Theta_d^\op$ by the diagonal embedding $\Delta_k^{\Theta_d}$, and then to $(\Delta^\op)^{\times d}$. We write $(\Delta^\op)^{\times d}=(\Delta^\op)^{\times (d-1)}\times \Delta^\op$, where the first factor is corresponded to the levels from 1 to $d-1$, and the last factor is corresponded to level $d$. 

We prefer to use the Moerdijk model structure on the category of $d$-polysimplicial sets (transferred from the Quillen-Kan model structure on $\Sets^{\Delta^\op}$ via the diagonal embedding, see [M, Sect.1] for $d=2$ case, the general case is similar) for the fibration property. Logically it is the same as checking the Kan fibration property for the restriction to $\Delta^\op$ via $\delta_d^{\Delta}$, but in author's opinion this way of presentation makes the argument more clear.

Recall that the generating acyclic cofibrations for the Moerdijk model structure on $\Sets^{(\Delta^\op)^{\times d}}$.
Denote by $A^{k_1,\dots,k_d}[n]$ the $d$-polysimplicial set whose $(\ell_1,\dots,\ell_d)$-components is
$$
A^{k}[n]_{(\ell_1,\dots,\ell_d)}=\{\alpha_1\colon [\ell_1]\to [n],\dots,\alpha_d\colon [\ell_d]\to [n]|\text{  there is  }j\ne k, \text{  such that  }j\not\in \mathrm{Im}\alpha_i\text{  for any  }i\}
$$
where $k\in [n]$. 

The Moerdijk model structure on $d$-polysimplicial sets is cofibrantly generated, with the set of generating acyclic cofibrations given by the imbeddings $i_{k,n}\colon A^k[n]\to \Delta^{\times d}(-,[n]\times \dots\times [n])$ for all $k, n$. 

We have to show that there is a filler in the following diagram
$$
\xymatrix{
A^k[n]\ar[d]\ar[r]&\mathcal{L}_{(\bmu,\bsigma)}(T)\ar[d]^{p}\\
\Delta^{\times d}(-,[n]\times \dots\times [n])\ar[r]\ar@{.>}[ur]&\mathcal{L}_{(\bmu,\bsigma)}(T_{(d-1)})
}
$$
Now up to level $d-1$ the two $d$-polysimplicial sets on the right are equal, that means, their restrictions to $(\Delta^{\op})^{\times (d-1)}\hookrightarrow(\Delta^\op)^{\times d}$, corresponded to the levels from 1 to $d-1$, are equal. That is, the filler tautologically exists up to level $d-1$. It turns our that we have to find a filler at level $d$ only, which is the diagram
$$
\xymatrix{
\Lambda^n_k\ar[d]\ar[r]&\Delta(-,[1])\ar[d]\\
\Delta(-,[n])\ar[r]\ar@{.>}[ur]&\Delta(-,[0])
}
$$
But this case was considered in the proof of Proposition \ref{reedytot}(2).

\endcomment

(3): it is clear. 

\end{proof}
Now the statement of Proposition \ref{propreedyfibranttheta} follows from Proposition \ref{reedytottheta}, because
\begin{itemize}
\item[(a)] the cellular topological realization functor is strong monoidal and, more generally, preserves finite limits (it follows from Proposition \ref{propbemtheta}(1),(3), Proposition \ref{propberger3}, and from the fact that (poly)simplicial realization preserves finite limits). It already proves Proposition 
\ref{propreedyfibranttheta} for the case when $T$ has at least two leaves. 
\item[(b)] For a polycellular set $X\colon (\Theta_d^\op)^{\times k}\to\Sets$ the topological polyrealization $|X|_{(\Theta_d^{\op})^{\times k}}$ is homeomorphic to the topological realization of the restriction $|(\delta_k^{\Theta_d})^*X|_{\Theta_d^\op}$; it follows from Proposition \ref{propberger2}(3) by the standard argument of representing any presheaf as a colimit of the Yoneda presheaves and the fact that the realization is left adjoint and thus commutes with all colimits).
\item[(c)] For a polysimplicial set $Y\colon (\Delta^\op)^{\times d}\to\Sets$ which is the external direct product of $d$ simplicial sets $Y_i\colon \Delta^\op\to \Sets$, $Y=Y_1\times\dots\times Y_d$, the realisation $|Y|_{(\Delta^\op)^{\times d}}$ is homeomorphic to the product $|Y_1|_{\Delta^\op}\times\dots |Y_d|_{\Delta^\op}$ of realizations.

\comment
the topological polyrealization $|Y|_{\Delta^{\times d}}$ is homeomorphic to the topological realization $|(\diag_k^\Delta)^*Y|_{\Delta^\op}$, it is well-known and is proven similarly to (b), using $|\Delta(-,[n])\times\Delta(-,[m])|\simeq |\Delta(-,[n])|\times |\Delta(-,[m])|$. 
\endcomment
\item[(d)] the topological realization of a Kan fibration of simplicial sets is a Serre fibration in $\Top$.
\item[(e)] the topological realization $|Z|_{\Theta_d}$ of a $d$-cellular set $X\colon \Theta_d^\op\to\Sets$ is homeomorphic to the topological reallization $|(\delta^\wr)^*Z|_{\Delta^{\times k}}$ of the polysimplicial set $(\delta^\wr)^*Z\colon (\Delta^\op)^{\times d}\to\Sets$, it is the statement of Proposition \ref{propberger3}.
\item[(f)] any topological space is fibrant. 
\end{itemize}

\end{proof} 

Now we prove Proposition \ref{propcelltop}(ii) as follows. By Proposition \ref{propbem}(1) for a Reedy cofibrant $C$ the totalization $\Tot_C(-)$ is a right Quillen for the Reedy model structure. It implies that a term-wise weak equivalence $X_1\simeqto X_2$ of two Reedy fibrant objects $X_1,X_2$ induces a weak equivalence $\Tot_C(X_1)\simeqto\Tot_C(X_2)$ of the totalizations (by the Ken Brown lemma). We know that $T\mapsto C_T$ is Reedy cofibrant for $\mathcal{R}=\Theta_d$, by Proposition \ref{propberger2}(2). 
We take $X_1=|\mathcal{L}_{(\bmu,\bsigma)}(-)|$, which is Reedy fibrant by Proposition \ref{propreedyfibranttheta}, and $X_2=*$ the constant cosimplicial space. 

\qed

\subsubsection{\sc The dg condensation}
Here we prove Theorem \ref{propthetadg}.

The statement of Theorem \ref{propthetadg}(i) follows from Proposition \ref{propmsd} and from Lemma \ref{lemmarealthetadgtop},
as $$|\mathcal{L}_{(\bmu,\bsigma)}(T)|_{\Theta_d,\dg}\overset{\text{ by Lemma \ref{lemmarealthetadgtop}}}{\simeq} C_\ldot^{\mathrm{CW}}(|\mathcal{L}_{(\bmu,\bsigma)}(T)|_{\Theta_d,\top},\mathbb{Z})\sim \mathbb{Z}[0]$$
where the last quasi-isomorphism follows from homotopy invariance of CW-homology and Theorem \ref{propthetatop}(i).

Next, we prove Theorem \ref{propthetadg}(ii). We have to compute $\Tot_{\Theta_d,\dg}(|\mathcal{L}_{(\bmu,\bsigma)}(-)|_{\Theta_d,\dg})$ and to prove it is quasi-isomorphic to $\mathbb{Z}[0]$. By Theorem \ref{propthetadg}(i), each complex 
$|\mathcal{L}_{(\bmu,\bsigma)}(T)|_{\Theta_d,\dg}$ is quasi-isomorphic to $\mathbb{Z}[0]$. If we knew that the cosimplicial complex $T\mapsto |\mathcal{L}_{(\bmu,\bsigma)}(T)|_{\Theta_d,\dg}$ is Reedy fibrant, Proposition \ref{propbemtheta}(2) would imply that 
the map on dg realizations induced by the natural map $\mathcal{L}_{(\bmu,\bsigma)}(T)\to *$ induces a quasi-isomorphism 
$\Tot_{\Theta_d,\dg}(|\mathcal{L}_{(\bmu,\bsigma)}(-)|_{\Theta_d,\dg})\simeqto \Tot_{\Theta_d,\dg}(\mathbb{Z}[0])$. The r.h.s. is quasi-isomorphic to $\mathbb{Z}[0]$ by Lemma \ref{lemmatotthetaconst} below.

Thus, it remains to show 

\begin{prop}\label{propreedyfibrantthetadg}
The cosimplicial complex $T\mapsto |\mathcal{L}_{(\bmu,\bsigma)}(T)|_{\Theta_d,\dg}$ is Reedy fibrant with respect to the projective model structure on $C^\udot(\mathbb{Z})$. 
\end{prop}
\begin{proof}
The argument is similar to the one in the case of $\Delta$ (Corollary \ref{corfibdg}). The components $ |\mathcal{L}_{(\bmu,\bsigma)}(T)|_{\Theta_d,\dg}$ for fixed $T$ are total sum complexes of a polycomplex $L^{i_1,\dots,i_k}$. Moreover, it is also the total product complex, because the ``polydiagonal'' spaces $\oplus_{i_1+\dots+i_k=N}L^{i_1,\dots,i_k}$ are finite-dimensional. Therefore, the total complex functor is in fact right adjoint, and, as such, commutes with (finite, because we need to maintain finite dimension of the ``polydiagonal'' components)  limits. It implies that the matching map for $|\mathcal{L}_{(\bmu,\bsigma)}(T)|_{\Theta_d,\dg}$ for the case of Proposition \ref{reedytottheta}(1) is an isomorphism. For the cases of Proposition \ref{reedytottheta}(2), (3), the map $p\colon \mathcal{L}_{(\bmu,\bsigma)}(T)\to M_T\mathcal{L}_{(\bmu,\bsigma)}(-)$ is a component-wise surjective 
map of poly-$d$-cellular sets, thus it remains surjective after dg realization, and thus is a fibration for the projective model structure 
on $C^\udot(\mathbb{Z})$.
\end{proof}

\begin{lemma}\label{lemmatotthetaconst}
The totalization $\Tot_{\Theta_d,\dg}(\underline{\mathbb{Z}})$ is quasi-isomorphic to $\mathbb{Z}[0]$.
\end{lemma}
\begin{proof}
\comment
The statement is trivial and well-known for $d=1$ (the case of category $\Delta$), where it can be shown by a simple direct computation. 
However, for $d>1$ a direct computation is non-trivial, so we provide the following argument. We only sketch the proof, a detailed account will appear in [Sh4]. 

The idea is to make use of the following fact: let $C$ be a small category, $A\colon C\to \Ab$ a functor which is an isomorphism on all arrows of $C$ (e.g. a constant functor), then $A$ defines a local system $\tilde{A}$ on $BC$, and
$H_\ldot(\hocolim_C A)=H_\ldot(BC,\tilde{A})$ and similarly $H^\udot(\holim_C A)=H^\udot(BC,\tilde{A})$. An advantage of this formula is that the r.h.s. is homotopy invariant, so it is 0 in non-zero degrees if $C$ is contractible. In our example, $\Theta_d$ has a final object and thus is contractible. So we will be done if we prove that there is a weak equivalence (in $C^\udot(\Ab)$)
\begin{equation}\label{eqbk}
\Tot_{\Theta_d}\mathbb{Z}\sim \holim_{\Theta_d}\mathbb{Z}
\end{equation}
Recall that for $d=1$ (the case of category $\Delta$) the Bousfield-Kan map [BK, XI, 2.6 ???], [Hi, 19.8.5], is a map of (Reedy cofibrant) cosimplicial simplicial sets
\begin{equation}\label{eqbk}
N(\Delta/[n])\to \Delta(?,[n])
\end{equation}
which associates to a $k$-chain $[n_0]\to[n_1]\to\dots\to [n_k]\to [n]$ the images of the maximal elements $n_0\in [n_0],\dots, n_k\in [n_k]$ in $[n]$, considered as a map $[k]\to [n]$. For any standard system of simplices in a monoidal model category $\mathcal{M}$, see ???, the realization of cosimplicial objects in $\SSets$ to cosimplicial objects in $\mathcal{M}$ is a left Quillen functor, so the realizations of both sides of \eqref{eqbk} remain Reedy cofibrant cosimplical objects. Therefore, for any Reedy fibrant $X\colon \Delta\to \mathcal{M}$ one has a {\it weak equivalence} induced by the Bousfield-Kan map:
\begin{equation}\label{eqbk2}
\holim_\Delta X=\underline{\Hom}(|N(\Delta/[n])|,X)\to \underline{\Hom}(|\Delta(?,[n])|,X)=\Tot_\Delta X
\end{equation}
Here we are interested in the case $\mathcal{M}=C^\udot(\mathbb{Z})$. As a constant cosimplicial abelian group is Reedy fibrant, we get:
\begin{equation}\label{eqbk2}
\Tot_\Delta\mathbb{Z}\sim\holim_\Delta\mathbb{Z}\sim H^\udot(B\Delta, \mathbb{Z})=\mathbb{Z}[0]
\end{equation}
This statement is trivial, but the method can be generalised for $\Theta_d$. 

We need an analogue of the Bousfield-Kan map for $\Theta_d$. 

A naive generalisation would be a map $N(\Theta_d/T)\to \Theta_d(?/T)$. It is not really nice as the lhs is a functor $\Delta^\op\times\Theta_d\to \Sets$ and the rhs is a functor $\Theta_d^\op\times\Theta_d\to\Sets$. So we have to upgrade the lhs. 

The idea is to upgrade the comma-category $\Theta_d/T$ to a $d$-category, and the nerve to Joyal nerve $N_d(-)$ of a strict $d$-category, which is a functor $\Theta_d^\op\to\Sets$. 

For instance, for $d=2$, its objects are maps $\phi_1\colon T_1\to T$ in $\Theta_2$. 1-morphisms from $T_1\to T$ to $T_2\to T$ are pairs $(f,\alpha)$ where $f\colon T_1\to T_2$ is a map of $\Theta_2$, and $\alpha\colon \phi_1\to \phi_2\circ f$ a strict natural transformation between the corresponding strict 2-functors. Finally, a 2-morphism 
$\theta\colon (f,\alpha)\to (g,\beta)$ is a natural transformation $\theta\colon f\to g$ such that the composition (pasting)
$( \theta\circ_h \id_{\phi_2})\circ_v \alpha=\beta$. The compositions are clear. 

A similar definition exists for any $d$ and gives rise to a strict $d$-category, which we denote by $(\Theta_d/T)_d$. 

We can construct a map 
\begin{equation}\label{eqbk3}
N_d((\Theta_d/T)_d)\to\Theta_d(?/T)
\end{equation}
Both sides are Reedy cofibrant objects in the category $\Fun(\Theta_d, \hat{\Theta}_d)$ of $d$-cocellular $d$-cellular objects.

On the other hand, there is a functor of $d$-categories $(\Theta_d/T)_d\to\Theta_d/T$, where in the r.h.s. all $i$-morphisms for $i>1$ are the identity morphisms. It gives rise to map
$$
N_d((\Theta_/T)_d)\to N_d(\Theta_d/T)
$$
The r.h.s. nerve has many degenerate cells, and one easily shows that the $\Theta_d$-dg realisation of $N_d(\Theta_d/T)$ is isomorphic to the $\Delta$-dg realisation of $N(\Theta_d/T)$ (of the usual nerve of 1-category).

Let $\mathcal{M}$ be a monoidal model category. 
For $X\colon \Theta_d\to\mathcal{M}$ denote by 
$$X_{\mathrm{lim}}=\underline{\Hom}_{\Theta_d}(|N_d((\theta_d/T)_d)|,X)$$ Then for a fibrant $X$ one gets a zig-zag of quasi-isomorphisms
\begin{equation}
\begin{aligned}
\ &\holim_{\Theta_d}X=\underline{\Hom}_{\Theta_d}(|N(\Theta_d/T)|,X)\sim X_{\mathrm{lim}}\eqto
\underline{\Hom}_{\Theta_d}(|\Theta_d/T|,X)=\Tot_{\Theta_d}X
\end{aligned}
\end{equation}
Finally, $$\Tot_{\Theta_d}\mathbb{Z}=\holim_{\Theta_d}\mathbb{Z}=H^\udot(B\Theta_d,\mathbb{Z})=\mathbb{Z}[0]$$
because $\Theta_d$ has a final object. 
\endcomment
The statement is dual to the statement on dg realization: 
$|\underline{\mathbb{Z}}|_{\Theta_d,\dg}=\mathbb{Z}[0]$. On the other hand,
$|\underline{\mathbb{Z}}|_{\Theta_d,\dg}=|*|_{\Theta_d,dg}\overset{\text{Lemma \ref{lemmarealthetadgtop}}}{=}C_\ldot^{CW}(|*|_{\Theta_d,\top})=\mathbb{Z}[0]$

\end{proof}

Theorem \ref{propthetadg}(ii) is proved. 

\qed

\section{\sc The $\Theta_d$-colored $(d+1)$-operad $\mathbf{seq}_d$ and its contractibility}\label{sectioncontr3}

\subsection{\sc How to detect $n$-operads inside the generalised lattice paths}
Recall our notation $\mathcal{L}^d$, see \eqref{latticepath2}. We further denote by $\mathcal{L}^d(k)$ all generalised lattice paths 
with $k$ arguments. We consider $\mathcal{L}^d(k)$ as a functor 
$$
\mathcal{L}^d(k)\colon (\Theta_d^\op)^{\times k}\times \Theta_d\to \Sets
$$
Fix $n\ge 2$. We want to describe a rather general way define $(n-1)$-terminal $n$-operads whose arity $S$ components are {\it subfunctors }
of $\mathcal{L}^d(k)$, where $S$ is a pruned $n$-tree with $k$ leaves. 

We know some subfunctors in $\mathcal{L}^d(k)$, namely, $\mathcal{L}_{(\bmu,\bsigma)}^d(k)$, for any element $(\bmu,\bsigma)\in \mathcal{K}(k)^{\times d}$, see Section \ref{sectionposet}, which we refer to as ``single blocks". So the first guess was would be to associate a single block to each pruned $n$-level tree $S$, such that all together they form a $\Theta_d$-colored $n$-operad. 

It is indeed the case for the Tamarkin 2-operad $\mathbf{seq}=\mathcal{L}_{2,B}$ (that is, $d=1$, $n=2$), and its higher cousins $\mathcal{L}_{n,B}$, which we recalled in Section \ref{rembb}, see Proposition \ref{proptamop}. 
On the other hand, the higher operads in $\mathcal{L}^d$ we define in this Section are {\it not}
of this type: for $d>1$ their arity components are ``generated by more than one block'' $\mathcal{L}^d_{(\bmu,\bsigma)}$.

In fact, a rough picture is as follows: the arity components $\mathbf{seq}_d(S)\subset \mathcal{L}^d(|S|)$ consist of a finite number of single blocks, organised in a poset by inclusion. This poset is contractible, as well as any of single blocks. A slight modification of the proofs in Section \ref{sectioncontr2} shows then the contractibility of the operads $\mathbf{seq}_d$ in the topological and in the dg condensations.

So an immediate general question is as follows. Assume we want to define an $n$-operad $\mathcal{O}$ whose arity components are subfunctors of $\mathcal{L}^d$. Let $T$ be a pruned $n$-level graph, and assume we are given several elements 
$$
(\bmu_1(T),\bsigma_1(T)),\dots, (\bmu_s(T),\bsigma_s(T))\in \mathcal{K}^{\times d}(|T|)
$$
(the number $s$ also depends on $T$).

For each arity $T$, define $\mathcal{O}(T)$ as follows:
\begin{equation}\label{manyblocks}
\mathcal{O}(T)=\{\omega\in \mathcal{L}^d(|T|)|\ \exists 1\le i\le s, (\bmu,\bsigma)(\omega)\le (\bmu_i(T),\bsigma_i(T))\}
\end{equation}
(It is a precise expression of what we mean saying that the ``arity components consist of several blocks''). Clearly the components $\mathcal{O}(T)$ in \eqref{manyblocks} are subfunctors. 

Now the question is: which conditions on the {\it labels} $(\bmu_1(T),\bsigma_1(T)),\dots, (\bmu_s(T),\bsigma_s(T))$ guarantee that $\{\mathcal{O}(T)\}$ are components of an $n$-operad?

A simple but very important result of Proposition \ref{proptwoleaves} below says that one has to fix labels only for pruned $n$-trees with 2 leaves, which are subject to some compatibility. Out of this data we define labels for all pruned $n$-trees with arbitrary number of leaves such that they are components of an $n$-operad.

Let $T_1,T_2\in \Tree_n$ be pruned $n$-trees.
Recall that a morphism $\phi\colon T_1\to T_2$ in the category $\Tree_n$ is called {\it a quasi-bijection} if $\phi$ defines an isomorphism on the sets of leaves (in particular, $|T_1|=|T_2|$). If we consider pruned $n$-trees whose leaves are labelled from 1 to $|T|=\ell$, the quasi-bijections act on such labelled $n$-trees. We get a poset $\mathcal{M}^n_\ell$ defined as follows: its objects are pruned $n$-trees with $\ell$ leaves, whose leaves have additional numbering from 1 to $\ell$ not necessarily the standard one, as a part of the data, and a morphism from $T_1\to T_2$ is a quasi-bijection of the underlying $n$-trees. One easily sees that the opposite poset $(\mathcal{M}^n_\ell)^\op$ is canonically a subposet of the Berger poset $\mathcal{K}(\ell)$. In fact, if we consider the filtration component $\mathcal{K}^n(\ell)$, formed by objects $(\mu,\sigma)$ with $\mu\le \ell-1$, the natural embedding $(\mathcal{M}^n_\ell)^\op\to \mathcal{K}^n(\ell)$ is a homotopy equivalence, as it follows from results of [Be1] (though we will not use this fact). 

\sevafigc{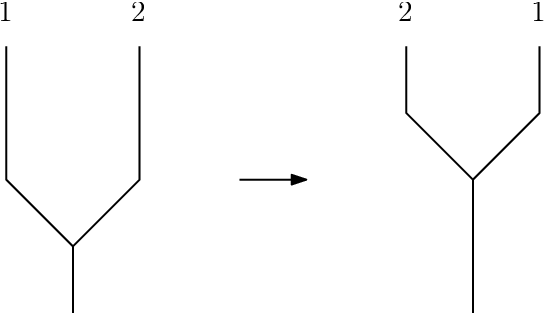}{50mm}{0}{A quasi-bijection $(1,\id)\to (2,(12))$ in $\mathcal{M}^4_2$\label{figquasi}}

In the statement below, we need only the poset $\mathcal{M}^n_2$, which is very simple, see Figure \ref{figquasi}.
Its objects are pruned $n$-trees with two leaves, whose leaves are labelled 1 and 2 in any of the two ways.
Denote a pruned $n$-tree with two leaves and vertices $1$ and $2$ such that $1<_a2$ by $T^n_a$, $0\le a\le n-1$. 
There are morphisms $(T^n_a,\sigma_1)\to (T^n_b,\sigma_2)$ if either $\sigma_1=\sigma_2$ and $a\le b$, or $\sigma_1\ne\sigma_2$ and $a<b$. The reader easily sees that the realisation of this poset is the sphere $S^{n-1}$. 

In the Proposition below we consider the {\it diagonal} action on $\Sigma_2$ on elements of $(\mathcal{K}^n(2))^{\times d}$, denoted by $(\bmu,\bsigma)\mapsto (\rho^{-1}\bmu,\rho^{-1}\bsigma)$.

\begin{prop}\label{proptwoleaves}
Let $d\ge 1$, $n\ge 2$. Assume for each element $((T^n_a,\id)\in \mathcal{M}^n_2$, $0\le a\le n-1$ we are given a set of elements 
$\{(\bmu_1(a),\bsigma_1(a)),\dots,(\bmu_{s(a)}(a),\bsigma_{s(a)}(a))\}$ of the poset $(\mathcal{K}^n(2))^{\times d}$ such that the following condition holds:
\begin{equation}\label{maincondition}
\parbox{5,5in}{\rm 
For any quasi-bijection $(T_a^n,\id)\to (T^n_b,\rho)$ in $\mathcal{M}^n_2$ and any $(\bmu_i(b),\bsigma_i(b))$ there exists $1\le j\le s(a)$ such that $(\bmu_j(a),\bsigma_j(a))\ge (\rho^{-1}\bmu_i(b),\rho^{-1}\bsigma_i(b))$, where $\ge$ is the dominance relation in the poset 
$(\mathcal{K}^n(2))^{\times d}$. 
}
\end{equation}

Define for each $n$-level tree $T$ the subfunctor $\mathcal{O}(T)$ as follows:
\begin{equation}\label{eqlatticenop}
\mathcal{O}(T)=\{\omega\in \mathcal{L}^d(|T|)|\ \forall i,j\in |T|, i<_cj\Rightarrow  (\bmu,\bsigma)(\omega_{ij}) \le (\bmu_q(c),\bsigma_q(c)) \text{  for some }1\le q\le s(c)\}
\end{equation}
Then $\{\mathcal{O}(T)\}$ are the components of a $\Theta_d$-colored $(n-1)$-terminal pruned $n$-operad.
\end{prop}
We consider the case $d=2$ in Section \ref{sectiond=2} as an illustration how this statement works, and the case $d=3$ in Section \ref{sectiond=3}. Then we discuss the case of general $d$ in Section \ref{sectiongen}. 
We prove that the desired $(d+1)$-operad $\mathbf{seq}_d$ exists. Here we state and prove two combinatorial Key-Lemmas \ref{lemmaconj1} and \ref{lemmaconj2}. The statement that $\mathbf{seq}_d$ follows directly from Proposition \ref{proptwoleaves} and Key-Lemma \ref{lemmaconj1}, whence its contractibility (provided in Section \ref{sectioncontr})  relies on Key-Lemma \ref{lemmaconj2}. 

For further use, denote by $P_{\mathcal{O}}(T)$ the subposet of $\mathcal{K}(|T|)^{\times d}$ defined as
\begin{equation}\label{posett}
P_{\mathcal{O}}(T)=\{(\bmu,\bsigma)| (\bmu,\bsigma)_{ij}\le  (\bmu_q(c),\bsigma_q(c)) \text{  for some }1\le q\le s(c)\text{  if  }i<_cj\}
\end{equation}

Denote also by $P_{\mathcal{O}}(T)_{ij}$ the following subposet of $\mathcal{K}(2)^{\times d}$ defined as
\begin{equation}\label{posettij}
P_{\mathcal{O}}(T)_{ij}=\{(\bmu,\bsigma)\le (\bmu_q(c),\bsigma_q(c)) \text{  for some }1\le q\le s(c)\text{  if  }i<_cj\}
\end{equation}
and by $P_{(\bmu,\bsigma)}$ an ``elementary'' subposet of $\mathcal{K}(-)^{\times d}$
\begin{equation}\label{posetsimple}
P_{(\bmu,\bsigma)}=\{(\bmu^\prime,\bsigma^\prime)| \ (\bmu^\prime,\bsigma^\prime)\ge (\bmu,\bsigma)\}
\end{equation}

We see from \eqref{eqlatticenop} that, as an abstract poset,
\begin{equation}\label{posetgent}
P_{\mathcal{O}}(T)=\prod_{(i,j)\in |T|, i<j}P_{\mathcal{O}}(T)_{ij}
\end{equation}

\begin{proof}
One has to prove that the conditions we impose on $\mathcal{O}(T)$ are preserved under $n$-operadic compositions. In $n$-operads, operadic compositions are associated to maps of $n$-trees. Let $T,S$ be pruned $n$-trees, $\phi\colon T\to S$ a morphism in $\Tree_n$. The operadic composition associated to $\phi$ is
$$
m_\phi\colon \mathcal{O}(S)\otimes \bigotimes_{i\in |S|}\mathcal{O}(P(\phi^{-1}(i)))\to\mathcal{O}(T)
$$
Here the fibres $\phi^{-1}(i)$ are typically not pruned $n$-trees, and $P(-)$ denotes the prunisation, see Definition \ref{prunedop}. 

The operadic composition is clear from the generalised lattice path definition \eqref{latticepath2}.
Let $\omega^S$ and $\{\omega^{P(\phi^{-1}(i))}\}_{i\in |S|}$ be generalised lattice paths we like to compose by $\phi$. Denote by $\omega$ the composition generalised lattice path. Let $s,t\in |T|, s\ne t$. Assume that $s<_at$ for some $a$. 
We easily see that 
\begin{equation}
(\bmu,\bsigma)(\omega_{st})=\begin{cases}
(\bmu,\bsigma)(\omega^{P(\phi^{-1}(i))}_{st})&\text{  if  }\phi(s)=\phi(t)=i\\
\rho^{-1}(\bmu,\bsigma)(\omega^S_{\phi(s),\phi(t)})&\text{  if  }\phi(s)\ne \phi(t), \rho=\id\text{  if  }\phi(s)<_b\phi(t), \rho=(21) \text{  if  }\phi(t)<_b\phi(s)
\end{cases}
\end{equation}
Assume $s<_at$ in $T$, then $s<_at$ in $\phi^{-1}(i)$ as well (here we denote by $s$ and $t$ the leaves of $\phi^{-1}(S)$ whose images under the natural inclusion are $s,t$), so in the first case $\omega_{st}$ satisfies \eqref{eqlatticenop} because $\omega_{st}^{\phi^{-1}(i)}$ does by assumption.

For the second case, the restriction of $\phi$ on the $n$-tree with two leaves $s,t$ gives a quasi-bijection to its image.
Thus, we have a quasi-bijection $(T^n_a,\id)\to (T^n_b,\rho)$. It implies that either $\rho=\id$ and $b\le a$, or $\rho=(21)$ and $b<a$. 
In both cases $\omega_{st}$ satisfies \eqref{eqlatticenop}, because $\omega_{\phi(s),\phi(t)}^S$ does by assumption, and due to the fact that our systems of labels for two-leaves graphs satisfy \eqref{maincondition}.

\end{proof}

\subsection{\sc The case $d=2$}\label{sectiond=2}
It is worth to start with revising the case $d=1$, considered in Section \ref{rembb}, from the point of view of Proposition \ref{proptwoleaves}.

In the sequel, we use notation $\underset{a\text{  symbols  }}{\underbrace{12121\dots}}$ for $(a,\id)$, and 
$\underset{a\text{  symbols  }}{\underbrace{21212\dots}}$ for $(a, (21))$.

Our conditions for the operad $\mathcal{L}_{n,B}$ are: $i<_a j\Rightarrow \omega_{ij}\le (n-a-1, \id)$ (recall that by convention the $(i,j)$-complexity $\mu_{ij}$ corresponds to $\mu_{ij}-1$ in the Berger poset).

The condition one has to check is: for a quasi-bijection $(T^2_a,\id)\to (T^2_b,\rho)$ of 2-level trees one has
$(n-a-1,\id)\ge (n-b-1,\rho)$. Definitely $a\ge b$ and $a>b$ if $\rho=(21)$.  Therefore, $n-a-1\ge n-b-1$, and the inequality is strict if $\rho=(21)$. But it is exactly the relation in the Berger poset $\mathcal{K}(2)$, see \eqref{bergerposet}. 

In our symbolic notation it reads
$$
\underset{n-a-1\text{  symbols  }}{\underbrace{12121\dots}}\ge \underset{n-b-1\text{  symbols  }}{\underbrace{12121\dots}}
$$
for $a\ge b$, and
$$
\underset{n-a-1\text{  symbols  }}{\underbrace{12121\dots}}\ge \underset{n-b-1\text{  symbols  }}{\underbrace{21212\dots}}
$$
for $a>b$. 

So this case is really trivial. 

Now we turn the case $d=2$. We are interested in a particular 3-operad denoted by $\mathbf{seq}_2$, so $n=3$.

We postulate the following conditions defining $\mathbf{seq}_2(T)$ for 3-graphs with 2 leaves:

\begin{equation}\label{conditionsd=2}
\begin{aligned}
\ & i<_2j\Rightarrow (\bmu,\bsigma)(\omega_{ij})\le (12)|(12)\\
&i<_1j\Rightarrow (\bmu,\bsigma)(\omega_{ij})\le (121)|(12)\\
&i<_0j\Rightarrow (\bmu,\bsigma)(\omega_{ij})\le (121)|(121)\text{  or  } (\bmu,\bsigma)(\omega_{ij})\le (1212)|(21)
\end{aligned}
\end{equation}
Let us explain our notations. By bar symbol we separate the elements $(\mu,\sigma)$ at different levels, from the level 1 to $d$ rightwards. Thus our condition for $i<_0j$ can be rewritten as 
$(\bmu,\bsigma)(\omega_{ij})\le (\bmu,\bsigma)_1$ or $(\bmu,\bsigma)(\omega_{ij})\le (\bmu,\bsigma)_2$ where
$(\bmu,\bsigma)_1=[(121)|(121)]=((1,\id),(1,\id))$ and $(\bmu,\bsigma)_2=[(1212)|(21)]=((2,\id), (0,(12)))$. (Recall that by convention  $(\mu_{ij},\sigma)$ corresponds to $(\mu_{ij}-1,\sigma)$ in the Berger poset. 

The following remark is important.
\begin{remark}\label{remexpl}
{\rm
Assume at some level $i$ we have $\mu_i=0$, that is, the corresponding element is $(12,\sigma)$ or $(21,\sigma)$. Then there can not be any lattice paths at levels greater or equal than $i+1$, as can be easily seen from definitions. 
In particular, we can rewrite the condition for $i<_2j$ in \eqref{conditionsd=2} as $i<_2j\Rightarrow (\bmu,\bsigma)(\omega_{ij})\le (12)|-$. Here $-$ is considered as an initial object of the poset at the corresponding level. 
}
\end{remark} 
\begin{lemma}\label{lemmad=2}
The system \eqref{conditionsd=2} satisfies \eqref{maincondition}. 
\end{lemma}
\begin{proof}
Due to Remark \ref{remexpl}(2) the statement is clear for maps of trees $\phi\colon(1,\id)\to (2,\rho)$. Indeed, it holds as 
$[(121)|(12)]>\rho[(12)|-]$, the r.h.s. is equal to $[(12)|-]$ or $[(21)|-]$ dependently on $\rho$.

Assume that for the case $i<_0j$ in \eqref{conditionsd=2} we took only one inequality, namely $(\omega_{ij})\le (121)|(121)$. This block is considered as ``leading'' one, in applications it corresponds to ``two-dimensional braces''. We claim that in this case \eqref{maincondition} {\it fails}. Indeed, consider the quasi-bijection $\phi\colon (0,\id)\to (1,\id )$ which switches the leaves. The inequality one has to check is 
$$
[(121)|(121)]\overset{?}{\ge} \sigma_{12}[(121)|(12)]=[(212)|(21)]
$$
Clearly it fails in $(\mathcal{K}(2))^{\times 2}$ (recall that we consider $\mathcal{K}^{\times d}$ as the cartesian product of posets, thus $\ge$ relation means $\ge$ relation for all components).

This example is crucial: it shows that one has to consider more than 1 block $\mathcal{L}^2_{(\bmu,\bsigma)}$ for the case $i<_0j$. 

Now turn back to the definition \eqref{conditionsd=2}. One has
$$
[(1212)|(21)]\ge \sigma_{12}[(121)|(12)]=[(212)|(21)]
$$
as $(1212)\ge (212)$ and $(21)\ge (21)$. This computation shows that \eqref{maincondition} holds for \eqref{conditionsd=2}.

\end{proof}

By Proposition \ref{proptwoleaves}, \eqref{eqlatticenop} defines a $\Theta_2$-colored 3-operad, which we denote by $\mathbf{seq}_2$. We have to prove its contractibility, in the topological and in the dg condensations. Here we discuss the first step of the proof, namely, the contractibility of the posets $P_{\mathbf{seq}_d}(T)$, see \eqref{posett}, for any 3-level tree $T$. 

We use the contractibility of the posets $P_{\mathbf{seq}_2}(T)$ in the proof the contractibility of $\mathbf{seq}_2$ in the topological and in the dg condensations, see Section \ref{condtopfinal}.

We strart with 3-trees with two leaves. 
The only $T$ for which the poset $P_{\mathbf{seq}_2}(T)$ is distinct from elementary $P_{(\bmu,\bsigma)}$ (see \eqref{posetsimple}) is the 3-level tree $T^3_0$ with two leaves 1 and 2 such that $1<_02$. The posets $P_{(\bmu,\bsigma)}$ contain a final object and thus are contractible (have contractible nerves).

The idea is to ``cover'' $P_{\mathbf{seq}_2}(T_0^3)$ by contractible $P_{(\bmu,\bsigma)}$ such that the elements of this cover form a contractible (finite) poset. 

The poset for $P_{\mathbf{seq}_2}(T_0^3)$ is the opposite to the one shown in \eqref{posetd=2}. (The only reason for us to work with the poset opposite to $P_{\mathbf{seq}_2}(T_0^3)$ is that it is more natural to visualise, on the other hand, $|N(\mathcal{C})|=|N(\mathcal{C}^\op)|$ for any small category $\\mathcal{C}$). 

\begin{equation}\label{posetd=2}
\xymatrix{
&&\boxed{(121)|(12)}\ar[drr]\ar[dddrr]\\
\boxed{(121)|(121)}\ar[urr]\ar[drr]&&&&\boxed{(12)}\\
&&\boxed{(121)|(21)}\ar[urr]\ar[drr]\\
\boxed{(1212)|(21)}\ar[urr]\ar[drr]&&&&\boxed{(21)}\\
&&\boxed{(212)|(21)}\ar[urr]\ar[uuurr]
}
\end{equation}

\begin{lemma}\label{colimd=2}
The poset \eqref{posetd=2} is contractible.
\end{lemma}
\begin{proof}
The statement is rather elementary and certainly can be proved directly. We provide a more general approach, which will be employed below for higher $d$. 

Denote by $\overline{P}$ the poset \eqref{posetd=2} (opposite to $P_{\mathbf{seq}_2}(T^3_0)$). Denote by $\overline{P}_1$ the subposet of $\overline{P}$ of elements $\ge [(121)|(121)]$, by $\overline{P}_2$ the subposet of elements $\ge [(1212)|(21)]$, and by $\overline{P}_{12}$ the subposet of elements $\ge [(121)|(21)]$. The poset $\overline{P}$ is a colimit in $\Cat$ of the diagram indexed by the category $I$ of inclusions of posets $\overline{P}_1\leftarrow  \overline{P}_{12}\rightarrow \overline{P}_2$. The posets $\overline{P}_1,\overline{P}_2,\overline{P}_{12}$ are contractible as they have initial elements. We want to deduce from here that $\overline{P}$ is contractible.

Denote by $\tilde{\overline{P}}$ the Grothendieck construction of the above functor $I\to \Cat$. 
We use here the fibred Grothendieck construction, which is fibred over $I$. By the Thomason homotopy colimit theorem [Th],
$$
N(\tilde{\overline{P}})=\hocolim_I N(\overline{P}_d)
$$
where we denote by $d$ an object of $D$, that is, $1$, $2$, or $12$. 
The category $I$ is clearly contractible. It follows that 
$$
N(\tilde{\overline{P}})\sim \hocolim_I N(\overline{P}_d)\sim \hocolim_D *\sim N(D)\sim *
$$
is contractible.

On the other hand, there is a natural projection $p\colon \tilde{\overline{P}}\to \overline{P}=\colim_I \overline{P}_d$. We claim that $p$ is a homotopy equivalence. Indeed, by Quillen Theorem A it is enough to show that the homotopy fibres $\tilde{\overline{P}}/a$ are contractible, for any $a\in \overline{P}$, which is clear. 
\end{proof}

We can now describe the poset  $P_{\mathbf{seq}_2}(T)$ for any pruned 3-level tree $T$. As the reader expects, it is reduced to the case of threes with two leaves. 

From definition \eqref{posetgent} we see that the contractibility of the posets $P_{\mathbf{seq}_2}(T)$ for 3-level trees $T$ with 2 leaves implies the contractibility of the posets $P_{\mathbf{seq}_2}(T)$ in general. (The same reduction works for general $n$-operad $\mathcal{O}$, constructed as in Proposition \ref{proptwoleaves}). 

We have proved
\begin{prop}\label{propposetd=2}
For any pruned 3-level graph $T$, the poset $P_{\mathbf{seq}_2}(T)$ is contractible. 
\end{prop}

\qed

\subsection{\sc The case $d=3$}\label{sectiond=3}
Here we consider  the case of 4-operad $\mathbf{seq}_3$. 

The 3-terminal $\Theta_3$-colored pruned 4-operad $\mathbf{seq}_3$ is defined by the following conditions for the graphs with two leaves:
\begin{equation}\label{conditionsd=3}
\begin{aligned}
\ & i<_3j\Rightarrow (\bmu,\bsigma)(\omega_{ij})\le (12)\\
&i<_2j\Rightarrow (\bmu,\bsigma)(\omega_{ij})\le (121)|(12)\\
&i<_1j\Rightarrow (\bmu,\bsigma)(\omega_{ij})\le (121)|(121)|(12)\text{  or  } (\bmu,\bsigma)(\omega_{ij})\le (1212)|(21)\\
&i<_0j\Rightarrow (\bmu,\bsigma)(\omega_{ij})\le (121)|(121)|(121)\text{  or  }(\bmu,\bsigma)(\omega_{ij})\le (121)|(1212)|(21)\text{  or  }\\
&(\bmu,\bsigma)(\omega_{ij})\le (1212)|(212)|(21)\text{  or  }(\bmu,\bsigma)(\omega_{ij})\le(12121)|(12)
\end{aligned}
\end{equation}
We see that the the poset for $i<_aj$ for $a=1,2,3$ for $d=3$ is the same as the poset for $i<_{a-1}j$ for $d=2$ (see \eqref{conditionsd=2}). 
The only new poset is $P_{\mathbf{seq}_3}(T^4_0)$. The poset of the opposite to it is drawn in \eqref{bigdiagram}. 

\begin{lemma}\label{lemmad=3}
The system \eqref{conditionsd=3} satisfies \eqref{maincondition}.
\end{lemma}
The proof is straightforward and is left to the reader.

\qed

\begin{equation}\label{bigdiagram}
{\tiny
\xymatrix{
&&&\boxed{(121)|(121)|(12)}\ar[dddrrr]\ar[dddddrrr]\\
\\
&&&\boxed{(121)|(121)|(21)}\ar[drrr]\ar[dddrrr]\\
\boxed{(121)|(121)|(121)}\ar[uuurrr]\ar[urrr]&&&&&&\boxed{(121)|(12)}\ar[rrrdd]\ar[rrrdddd]\\
&&&\boxed{(121)|(212)|(21)}\ar[urrr]\ar[drrr]\\
\boxed{(121)|(1212)|(21)}\ar[uuurrr]\ar[urrr]&&&&&&\boxed{(121)|(21)}\ar[rrr]\ar[rrrdd]&&&\boxed{(12)}\\
&&&\boxed{(1212)|(21)}\ar[drrr]\ar[urrr]\\
\boxed{(1212)|(212)|(21)}\ar[drrr]\ar[dddrrr]\ar[rrruuu]\ar[rrru]&&&&&&\boxed{(212)|(21)}\ar[rrr]\ar[uurrr]&&&\boxed{(21)}\\
&&&\boxed{(212)|(212)|(21)}\ar[urrr]\ar[drrr]\\
\boxed{(12121)|(12)}\ar[drrr]\ar[dddrrr]&&&&&&\boxed{(212)|(12)}\ar[uurrr]\ar[uuuurrr]\\
&&&\boxed{(1212)|(12)}\ar[urrr]\ar[uuuuuuurrr]\\
\\
&&&\boxed{(2121)|(12)}\ar[uuuuuuuuurrr]\ar[uuurrr]
}
}
\end{equation}
Denote this poset by $\overline{P}$. For an object $a$ of $\overline{P}$, denote by $\overline{P}_a$ the subposet of $\overline{P}$ of elements $\ge a$.
Consider the poset $I$ whose graph is shown below:
$$
\mb\leftarrow\mb\rightarrow\mb\leftarrow\mb\rightarrow\mb\leftarrow\mb\rightarrow\mb
$$
Then the diagram below defines a functor $X\colon I\to \Cat$:

\begin{equation}\label{eqdfirst}
\begin{tikzpicture}[baseline=(current bounding box.center),scale=1]
{
\node (p1) at (0,0) {$\overline{P}_{(121)|(121)|(121)}$};
\node (p12) at (2,-1.5) {$\overline{P}_{(121)|(121)|(21)}$};
\node (p2) at (4,0) {$ \overline{P}_{(121)|(1212)|(21)}$};
\node (p23) at (6,-1.5) {$\overline{P}_{(121)|(212)|(21)}$};
\node (p3) at (8,0) {$\overline{P}_{(1212)|(212)|(21)}$};
\node (p34) at (10, -1.5) {$\overline{P}_{(1212)|(12)}$};
\node (p4) at (12,0) {$\overline{P}_{(12121)|(12)}$};

    \path[->]

  (p12)edge(p1)(p12)edge(p2)(p23)edge(p2)(p23)edge(p3)(p34)edge(p3)(p34)edge(p4);

}
\end{tikzpicture}
\end{equation}

Introduce shorter notations for the posets $\overline{P}_a$ standing at the vertices of $I$:
\begin{equation}\label{eqdsecond}
\begin{tikzpicture}[baseline=(current bounding box.center),scale=1]
{
\node (p1) at (0,0) {$\overline{P}_{1}$};
\node (p12) at (2,-1.5) {$\overline{P}_{12}$};
\node (p2) at (4,0) {$ \overline{P}_{2}$};
\node (p23) at (6,-1.5) {$\overline{P}_{23}$};
\node (p3) at (8,0) {$\overline{P}_{3}$};
\node (p34) at (10, -1.5) {$\overline{P}_{34}$};
\node (p4) at (12,0) {$\overline{P}_{4}$};

    \path[->]

  (p12)edge(p1)(p12)edge(p2)(p23)edge(p2)(p23)edge(p3)(p34)edge(p3)(p34)edge(p4);

}
\end{tikzpicture}
\end{equation}

\begin{lemma}\label{colimd=3}
One has $\colim_{a\in I} \overline{P}_a=\overline{P}$.
\end{lemma}
\begin{proof}
In notations of \eqref{eqdsecond}, it is enough to prove the following statement (which is specific for the poset $\overline{P}$):

Let $\alpha\in \overline{P}$ be an object. Assume $\alpha\in \overline{P}_a,\overline{P}_b$ for $a<b$. Then $\alpha\in \overline{P}_{i,i+1}$ for any $a\le i\le b-1$.

For this particular poset one checks it directly. 
\end{proof}

The posets $\overline{P}_a$ are contractible as they have initial objects, and the category $I$ is contractible as well.

Denote by $\tilde{\overline{P}}$ the Grothendieck construction of the functor $X\colon I\to\Cat$. 

From the Thomason homotopy colimit theorem [Th] one has
$$
N(\tilde{\overline{P}})=\hocolim_{a\in I}N(\overline{P}_a)\sim\hocolim_I*\sim N(I)\sim *
$$
Therefore, $\tilde{\overline{P}}$ is contractible. 

There is a natural functor $\tilde{\overline{P}}\to \overline{P}$ (just as generally there exists a functor $\hocolim F\to\colim F$), which is a homotopy equivalence. Indeed, by Quillen Theorem A, it is enough to show that the comma-categories $\tilde{\overline{P}}/a$ are contractible, for any $a\in \overline{P}$. It is clear.

We have proved the contractibility of the poset $P_{\mathbf{seq}_3}(T^4_0)$. The only non-trivial among the other posets $P_{\mathbf{seq}_3}(T^4_a)$ is the poset for $T^4_1$, which is, by our inductive definition, is the same as the poset $P_{\mathbf{seq}_2}(T^3_0)$, whose contractibility we proved in Lemma \ref{colimd=2}. 

By \eqref{posetgent}, we have proved

\begin{prop}\label{propposetd=3}
The poset $P_{\mathbf{seq}_3}(T)$ is contractible, for any pruned 4-level tree $T$. 
\end{prop}

\qed

\subsection{\sc The general case}\label{sectiongen}
Here we state two combinatorial Lemmas, Key-Lemma \ref{lemmaconj1} and Key-Lemma \ref{lemmaconj2}, and prove that these two Lemmas give rise 
to a construction of $\Theta_d$-colored $(d+1)$-operad $\mathbf{seq}_d$, for general $d$; its contractibility in the dg and topological condensations will be proven in Section \ref{sectioncontr} based on Key-Lemma \ref{lemmaconj2}.

First of all, define the following subset $V_d$ of objects of $\mathcal{K}(2)^{\times d}$.
It will be a totally ordered set (with respect to some ``external'' ordering, not the one of the poset $\mathcal{K}(2)^{\times d}$). 

The maximal element in $V_d$ is $$w_d=\underset{d \text{  factors  }(121)}{\underbrace{(121)|(121)|\dots|(121)}}$$
The other elements of $V^d$ are obtained from $w_d$ by the application of the  following procedure. 
At each step, called {\it an elementary move}, we can move the leftmost digit (1 or 2) from some factor to the factor next to the left, as its rightmost digit.
Here is an example for $d=3$:

\begin{equation}\label{moveex}
(121)|(121)|(\underline{1}21)\to (121)|(\underline{1}21\overline{2})|(21)\to (121\overline{2})|(\underline{2}12)|(21)\to
(1212\overline{1})|(12)
\end{equation}
In \eqref{moveex}, the digit we move is underlined at the source expression, and it is overlined at the target expression. In such an ``elementary move'', a digit 1 may become a digit 2, and vice versa, according to the rule that no two equal digits stand in turn in a single factor. For example, the first arrow in \eqref{moveex} takes (underlined) 1, moves it leftward, and replaces by 2 (otherwise, we would have two equal 1's in (1211)). The following three rules uniquely determine the entire process.

\begin{itemize}
\item[Rule 1.] When the leftmost digit 1 from some factor is moved to the next factor leftwards and becomes the rightmost digit in this factor, it becomes digit 2 if its neighbour digit is 1. Similarly 2 becomes 1.

\item[Rule 2.] If at some place we get a factor (12) or (21), all factors rightward are removed, and such operation is allowed only if the only rightward factor is (12) or (21). 

\item[Rule 3.] This rule specifies which elementary move is the next after the performed one. Assume that $\omega^d_1,\dots,\omega^d_i$ have been defined. We define $\omega^d_{i+1}$ as follows. We take the leftmost digit from a factor $A$ (specified below) and move it leftwards so this digit (after appropriate switch according to Rule 1) becomes the rightmost digit in the next to $A$ factor to the left. This operation is uniquely defined as soon as the factor $A$ is determined. The factor $A$ is determined as follows. For $i+1=1$, it is the rightmost factor $(121)$ of $w_d$. For $i+1>1$, we take the {\it rightmost} factor of $\omega^d_i$, distinct from (21) at the end, and such that the move does not create a new 2-element factor, if such a factor exists. 
Otherwise, $A$ is the factor next to the left to the rightmost factor $(21)$. This factor $A$ is then a 3-element factor, and, in agreement with Rule 2, the rightmost factor $(21)$ gets removed, and the total number of factors of $\omega_{i+1}^d$ is by 1 less than the number of factors of $\omega_i^d$. 
\end{itemize}

The necessity of Rule 3 is not visible for the examples for $d=2$ and $d=3$ considered above. 
It firstly emerges for the case $d=4$. Consider the following diagram of elementary moves. 

\begin{equation}\label{eqd4}
\begin{tikzpicture}[baseline=(current bounding box.center),scale=1]
{

\node (p1) at (0,0) {$(121)|(121)|(121)|(121)$};
\node (p2) at (0,-1) {$(121)|(121)|(1212)|(21)$};
\node (p3) at (0,-2) {$(121)|(1212)|(212)|(21)$};
\node (p4l) at (-2,-3) {$(121)|(12121)|(12)$};
\node (p4r) at (2,-3) {$(1212)|(212)|(212)|(21)$};
\node (p5) at (0,-4) {$(1212)|(2121)|(12)$};
\node (p6) at (0,-5) {$(12121)|(121)|(12)$};
\node (p7) at (0,-6) {$(121212)|(21)$};

    \path[->]
    
    (p1)edge(p2)(p2)edge(p3)(p3)edge(p4l)(p3)edge(p4r)(p4l)edge(p5)(p4r)edge(p5)(p5)edge(p6)(p6)edge(p7);

}
\end{tikzpicture}
\end{equation}
At $a=(121)|(1212)|(212)|(21)$ there were, without Rule 3, two options for the next elementary move, going to the left and to the right in 
\eqref{eqd4}. The left-hand moving path results in removing $\dots |(21)$ at the end of $a$, but the right-hand moving path shows an option without removal of a two-element factor. That is, {\it Rule 3 predicts that we go along the right-hand path}. (At the same time, the left-hand element, $(121)|(12121)|(12)$, does not belong to the path, we just take it out of $V_4$, as well as its incoming and outgoing edges). On the other hand, if we chose the left-hand path, the main conditions \eqref{maincondition} would fail, for the map of graphs $T^5_0\to T^5_1$ switching the leaves. 

The process is stopped when any elementary move is impossible. In this way, we get a set $V_d$, starting with $w_d$, of elements of $\mathcal{K}(2)^{\times d}$, which forms a linearly ordered set. We call this ordering {\it canonical}. Denote the elements in $V^d$, written in their canonical order, by $\{\omega_i^d\}_{i\in V^d}$. In particular, $\omega^d_1=w_d$.

Note also that what we get in \eqref{moveex} are exactly the expressions at the leftmost column of \eqref{bigdiagram}, where they stand in the canonical order downwards. 

In fact, the first $d$ terms of the ordered set $V_d$ can be written down in a very simple way.

\begin{lemma}\label{reallysimplel}
The element $\omega^d_i$ of $V^d$, $i\le d$ is equal to $w_d$ for $i=1$, and to
\begin{equation}
(121)|\dots|(121)|(1212)|(212)|\dots|(212)|(21)
\end{equation}
for $2\le i\le d$, 
where the factor $(1212)$ stands at the $i$-th position from the left.
\end{lemma}
\begin{proof}
It is clear from Rules 1-3 above. Note that a more sophisticated Rule 3 is essentially employed here. 
\end{proof}

We introduce an ordered set $\tilde{V}^\ell$, which has the same number of elements as $V^{\ell}$, denoted by $\{\tilde{\omega}^\ell_a\}$, as follows.  The first element $\tilde{\omega}_1^\ell$ is defined as $w_{\ell}|(12)$, it has $\ell+1$ factors. For $i>1$, the element $\tilde{\omega}^\ell_i$ is defined as $\omega^\ell_i$.  (We employ Rule 2, by which there are no two consecutive two-element factors). In particular, for $2\le i\le \ell$, $\tilde{\omega}^\ell_i$ is given by Lemma \ref{reallysimplel}.

\vspace{2mm}

We define ``labels'' for the $(d+1)$-operad $\mathbf{seq}_d$ as follows:
\begin{equation}\label{conditionsd=gen}
\begin{aligned}
\ & i<_{d}j\Rightarrow (\bmu,\bsigma)(\omega_{ij})\le (12)\\
&i<_{d-1}j\Rightarrow (\bmu,\bsigma)(\omega_{ij})\le (121)|(12)\\
&\dots\\
&i<_\ell j\Rightarrow (\bmu,\bsigma)(\omega_{ij})\le \text{  some  }\tilde{\omega}^\ell_a\in \tilde{V}^{d-\ell}\\
&\dots\\
&i<_1j\Rightarrow (\bmu,\bsigma)(\omega_{ij})\le \text{  some  }\tilde{\omega}_a^{d-1}\in \tilde{V}^{d-1}\\
&i<_0j\Rightarrow (\bmu,\bsigma)(\omega_{ij})\le \text{  some  } \omega^d_a\in V^d
\end{aligned}
\end{equation}
We have:

\begin{klemma}\label{lemmaconj1}
The system \eqref{conditionsd=gen} satisfies \eqref{maincondition}.
\end{klemma}
\begin{proof}
Note that, using induction on $d$, one has to check \eqref{maincondition} only for the two maps $\id,\sigma_{12}\colon T^{d+1}_0\rightrightarrows T^{d+1}_1$. Then \eqref{maincondition} for other maps in the poset $\mathcal{M}^{d+1}_2$ follows by induction on $d$ and our definition \eqref{conditionsd=gen}. Moreover, for $\id\in T^{d+1}$ the statement is clear by induction as well. The only case one has to prove is the following statement:

\centerline{For any $a$, there exists $b$ such that $\omega^d_b\le \sigma_{12}\tilde{\omega}^{d}_a$\ \ \  (*)}

We firstly prove (*) for $a=1$, that is, we find $b$ such that 
\begin{equation}\label{eqwow1}
\omega^d_b\le \sigma_{12}(w_{d-1}|(12))
\end{equation}
More precisely, we show that \eqref{eqwow1} holds for $b=d$. 

Indeed, by Lemma \ref{reallysimplel},
$$
\omega_d^d=(1212)|\underset{d-2 \text{  factors  }}{\underbrace{(212)|\dots|(212)}}|(21)
$$
and 
$$
\sigma_{12}(w_{d-1}|(12))=\underset{d-1\text{  factors  }}{\underbrace{(212)|(212)|\dots|(212)}}|(21)
$$
from which the statement is clear. 

Now we claim that for any $a\ge 1$ one has
$$
\omega_{d+a-1}^d\le \sigma_{12}(\tilde{\omega}^{d-1}_a)
$$
Indeed, denote by $M$ the operation $\omega^\ell_i\to \omega^\ell_{i+1}$, so that $M(\omega^\ell_i)=\omega^\ell_{i+1}$. 
Then one has:
$$
\omega^d_{d+a-1}=M^{a-1}\omega_1^d\le M^{a-1}\sigma_{12}\tilde{\omega}^{d-1}_1=\sigma_{12}M^{a-1}\tilde{\omega}^{d-1}_1=\sigma_{12}\tilde{\omega}^{d-1}_a
$$
\end{proof}

Key-Lemma \ref{lemmaconj1} and Proposition \ref{proptwoleaves} imply that $\mathbf{seq}_d$ is a $(d+1)$-operad.

\begin{coroll}\label{corollconj1}
The $d+1$-sequence $\mathbf{seq}_d$ form components of a $\Theta_d$-colored $(d+1)$-operad.
\end{coroll}
\qed

Now we pass to contractibility of $\mathbf{seq}_d$. 
The contractibility of the posets $P_{\mathbf{seq}_d}(T)$ is the only what we need to know from the combinatorics of $\mathbf{seq}_d$ for the contractibility of the $(d+1)$-operad $\mathbf{seq}_d$ in the topological and in the dg condensations, see Section \ref{sectionveryfinal}.
This contractibility of the poset $P_{\mathbf{seq}_d}(T)$ is deduced in Proposition \ref{proppropcontr}, by homotopy theoretical methods, from Key-Lemma \ref{lemmaconj2} below.

Denote by ${P}^{(d)}$ the subposet in $\mathcal{K}(2)^{\times d}$ which consists of elements $\le$ than one of $\omega^d_i\in V_d$. We need to know that the poset ${P}^{(d)}$ is contractible. As the poset ${P}^{(d)}$ does not have a final object for $d>1$, its contractibility is not straightforward. 

We know that the elements $\{\omega^d_i\in V^d\}$ are totally ordered. Each subposet ${P}^{(d)}_i:={P}_{\omega^d_i}$, which consists of all objects of $\mathcal{K}(2)^{\times d}$ which are $\le \omega^d_i$, has a final object and thus is contractible. Consider two ``neighbour'' posets ${P}^{(d)}_i, P^{(d)}_{i+1}$. Their intersection ${P}^{(d)}_{i,i+1}$ has as well a terminal objects which is easy to describe: one just removes from $\omega_i^d$ that digit which is moved leftwards so $\omega_i^d$ becomes $\omega^d_{i+1}$, denote $\omega_i^d$ with this digit removed by $a$; the image of the digit in $\omega^d_{i+1}$ can be removed as well, and the result of this operation is the same object $a$. Thus the object $a$ belongs to $\overline{P}^{(d)}_{i,i+1}$, and it is clear that any element in $P^{(d)}_{i,i+1}$ is $\le a$. Clearly $\overline{P}^{(d)}_{i,i+1}=\overline{P}^{(d)}_{i}\cap \overline{P}^{(d)}_{i+1}$. 

From now on, we discuss the posets {\it opposite} to ${P}^{(d)}$, ${P}^{(d)}_i$, ${P}^{(d)}_{i,i+1}$, for which we use notations $\overline{P}^{(d)}$, $\overline{P}^{(d)}_i$, $\overline{P}^{(d)}_{i,i+1}$. Thus, the posets $\overline{P}^{(d)}_i$ and $\overline{P}^{(d)}_{i,i+1}$ have initial objects. 

There is a diagram of posets in which all maps are inclusions:
\begin{equation}\label{eqdgen}
\begin{tikzpicture}[baseline=(current bounding box.center),scale=1]
{
\node (p1) at (0,0) {$\overline{P}^{(d)}_{1}$};
\node (p12) at (2,-1.5) {$\overline{P}^{(d)}_{12}$};
\node (p2) at (4,0) {$ \overline{P}^{(d)}_{2}$};
\node (pinf) at (6, -1.5) {$\dots$};
\node (pN-1) at (8,0) {$\overline{P}^{(d)}_{N-1}$};
\node (pNN-1) at (10,-1.5) {$\overline{P}^{(d)}_{N-1,N}$};
\node (pN) at (12, 0) {$\overline{P}^{(d)}_N$};

    \path[->]

  (p12)edge(p1)(p12)edge(p2)(pNN-1)edge(pN-1)(pNN-1)edge(pN);

}
\end{tikzpicture}
\end{equation}
where $N=\sharp V_d$. Denote the diagram \eqref{eqdgen} by $X\colon I^{(d)}\to\Cat$. Clearly $I^{(d)}$ is contractible. 

\begin{klemma}\label{lemmaconj2}
Let $d\ge 2$, $\alpha\in \overline{P}^{(d)}$ be an element. Assume $\alpha\in \overline{P}^{(d)}_i, \alpha\in \overline{P}^{(d)}_j$, $i<j$. Then $\alpha\in \overline{P}^{(d)}_{a,a+1}$ for any $i\le a\le j-1$. Equivalently, $\alpha\in \overline{P}^{(d)}_a$ for any $i\le a\le j$.
\end{klemma} 
\begin{proof}
Let $\alpha \in \overline{P}^{(d)}$ be such that $\alpha\in \overline{P}^{(d)}_i, \overline{P}^{(d)}_j$, for $i<j$. 

We prove the statement in the slightly different form: in these assumptions, $\alpha\in \overline{P}^{(d)}_{i+1}$ (**). It is clear that then the statement of Key-lemma follows by induction.

Recall that by Lemma \ref{reallysimplel} we have
$$
\omega_d^d=(1212)|(212)|\dots|(212)|(21)
$$
Consider two cases: $i<d$, and $i\ge d$. 

\vspace{5mm}

\noindent {\it Case 1:} $i\ge d$.

In this case the statement (**) easily follows by induction, from the statement for $d-1$. 

\vspace{5mm}

\noindent {\it Case 2:} $i<d$.
Recall that $\omega^d_1=(121)|\dots|(121)$, and
$$
\omega^d_i=(121)|\dots|(121)|(1212)|(212)|\dots|(212)|(21),\ \ 2\le i\le d
$$
where the factor $(1212)$ stands at the $i$-th position from the right. 

We prove (**), that is, 
$\alpha\in \overline{P}^{(d)}_{i+1}$, or $\alpha\le \omega_{i+1}^d$. For $i=1$ it can fail only if $\alpha$ has $d$ factors, and the rightmost factor of $\alpha$ is $(121)$. But in this case $\alpha$ can not belong to any $\overline{P}^{(d)}_j$ for $j>1$. So, consider the case $1<i<d$. If $\alpha\nleq$ then the $i$-th factor from the left of $\alpha$ is $(1212)$, as it is the only factor of $\omega^d_i$ such that the corresponding factor of $\omega^d_{i+1}$ is less than it. But in this case $\alpha$ can not belong to any $\overline{P}^{(d)}_j$ for $j>i$.

\end{proof}

\begin{coroll}\label{corollconj2}
One has:
$$
\colim_{\alpha\in I^{(d)}}\overline{P}^{(d)}_\alpha=\overline{P}^{(d)}
$$
\end{coroll}

It follows immediately from Key-Lemma \ref{lemmaconj2}.

\qed

\begin{prop}\label{proppropcontr}
The $\overline{P}^{(d)}$ is contractible.
\end{prop}

The proof repeats the argument used in the proofs of Lemma \ref{colimd=2} and Lemma \ref{colimd=3}. It uses the Thomason homotopy colimit theorem, by which $\hocolim_{\alpha\in I^{(d)}}N(\overline{P}^{(d)}_\alpha)$ is contractible. Denote by $\wtilde{\overline{P}}^{(d)}$ the fibred Grothendieck construction of the functor $X\colon I^{(d)}\to \Cat$. Then Quillen Theorem A applies to the canonical map $\wtilde{\overline{P}}^{(d)}\to\colim_{\alpha\in I^{(d)}}\overline{P}^{(d)}_\alpha\overset{\text{  by Corr  }\ref{corollconj2}}{ =}\overline{P}^{(d)}$. Here the crucial point is to know that $\colim_{\alpha \in I^{(d)}}\overline{P}^{(d)}_\alpha=\overline{P}^{(d)}$, which is the statement of Corollary \ref{corollconj2}. The comma-categories $\tilde{\overline{P}}^{(d)}/a$, $a\in \overline{P}^{(d)}$, are ``star-like'' and thus are contractible. We conclude that the poset $\overline{P}^{(d)}$ is contractible. 

\qed

\begin{prop}\label{propdoperad}
Define for each pruned $(d+1)$-level tree $T$ the arity components $\mathbf{seq}_d(T)$ by \ref{eqlatticenop}. Then the following statements are true:
\begin{itemize}
\item[(1)] The functors $\{\mathbf{seq}_d(T)\colon ((\Theta_d)^{\op})^{|T|}\times\Theta_d\to\Sets\}$ are arity components of a $\Theta_d$-colored $(d+1)$-operad $\mathbf{seq}_d$.
\item[(2)] For a pruned $(d+1)$-level tree $T$ the subposet $P_{\mathbf{seq}_d}(T)$ of $\mathcal{K}(|T|)^{\times d}$ $($defined in \eqref{posett}$)$ is contractible. 
\end{itemize}
\end{prop}

\begin{proof}
(1): it follows from Proposition \ref{proptwoleaves}, see Corollary \ref{corollconj1}. (2): By \eqref{posetgent}, as an abstract poset, $P_{\mathbf{seq}_d}(T)$ is the direct product 
$$
\prod_{(i,j)\in |T|, i<j}P_{\mathbf{seq}_d}(T)_{ij}
$$
Thus, the contractibility of the posets $P_{\mathbf{seq}_d}(T)$ reduces to the contractibility of such posets for the pruned $(d+1)$-level trees with two leaves. For the tree $T^{d+1}_0$ it is Proposition \ref{propropcontr}. For $T_a^{d+1}$ with $a>0$ it follows by induction on $d$, as an abstract poset, $P_{\mathbf{seq}_d}(T^{d+1}_a)=P_{\mathbf{seq}_{d-a}}T^{d-a+1}_0$, which is clear from definition \eqref{conditionsd=gen}.
\end{proof}

\subsection{\sc A proof of contractibility of $\mathbf{seq}_d$}\label{sectioncontr}
We prove the following Theorem:

\begin{theorem}\label{maintheorem}
For any $d\ge 1$, the $d$-terminal $\Theta_d$-colored $(d+1)$-operad $\mathbf{seq}_d$ is contractible in the topological and in the dg condensations. 
\end{theorem}
We prove the topological condensation part of Theorem \ref{maintheorem} in Section \ref{condtopfinal}, and the dg condensation part in Section \ref{sectionveryfinal}.

The general idea is to use Proposition \ref{propdoperad}(2), and Theorems \ref{propthetatop} and \ref{propthetadg} saying that a ``single block'' $\mathcal{L}^d_{(\bmu,\bsigma)}$ is contractible. 
However, one rather has to use the scheme of proofs of Theorems \ref{propthetatop} and \ref{propthetadg} than their statements, as we will see below.

\subsubsection{\sc The topological condensation}\label{condtopfinal}
One has, for a fixed cocellular argument $D\in \Theta_d$, 
\begin{equation}\label{colimoverposet}
\mathbf{seq}_d(T)[D]=\colim_{(\bmu,\bsigma)\in P_{\mathbf{seq}_d(T)}}\mathcal{L}^d_{(\bmu,\bsigma)}[D]
\end{equation}
\begin{prop}\label{prophirsch}
Consider the poset $P_{\mathbf{seq}_d}(T)$ as a directed Reedy category.
The following statements are true:
\begin{itemize}
\item[(1)] The functor $P_{\mathbf{seq}_d}(T)\to \Top$, $(\bmu,\bsigma)\mapsto |\mathcal{L}^d_{(\bmu,\bsigma)}[D]|_{\Theta_d,\top}$ is Reedy cofibrant.
\item[(2)] For any Reedy cofibrant functor $F\colon P_{\mathbf{seq}_d}(T)\to \Top$, the map $\hocolim_{P_{\mathbf{seq}_d}(T)}F\to\colim_{P_{\mathbf{seq}_d}(T)}F$ is a weak equivalence.
\end{itemize}
\end{prop}
\begin{proof}
(1): We show that the functor $Q: P_{\mathbf{seq}_d}(T)\to \hat{\Theta}_d^{\times |T|}$, $(\bmu,\bsigma)\mapsto \mathcal{L}^d_{(\bmu,\bsigma)}[T3]$ is Reedy cofibrant. Then the claim will follow because the realisation functor $\hat{\Theta}_d^{\times |T|}\to\Top$ is left Quillen, by Propositions \ref{propbemtheta}(1) and \ref{propberger3}. 

For the functor $Q$ the latching maps are clearly embeddings in $\hat{\Theta}_d^{\times |T|}$, thus cofibrations. (More precisely, one has to restrict to the diagonal $\Theta_d^\op\to (\Theta_d^\op)^{\times |T|}$, then one can refer to the Berger model structure on $\hat{\Theta}_d$ [Be2, Th.3.9], or further restrict to $\Delta^\op$ and refer to the Quillen model structure on simplicial sets and Proposition \ref{propberger3}).

(2): Any directed Reedy category $\mathcal{R}$ is a Reedy category with fibrant constants [Hi, Def. 15.10.1], just because $\mathcal{R}^-$ contains only the identity morphisms, so the matching categories are empty (see [Hi, Prop. 15.10.2]). Then [Hi, Theorem 19.9.1 (1)] and (1) proves the claim. 
\end{proof}

We have from \eqref{colimoverposet}:
\begin{equation}\label{posetcontrfin}
\begin{aligned}
\ &|\mathbf{seq}_d(T)[D]|_{\Theta_d,\top}=\colim_{(\bmu,\bsigma)\in P_{\mathbf{seq}_d(T)}}|\mathcal{L}^d_{(\bmu,\bsigma)}[D]|_{\Theta_d,\top}\overset{\text{  Prop. \ref{prophirsch} }}{\overset{\sim}{\leftarrow}} \\
&\hocolim_{(\bmu,\bsigma)\in P_{\mathbf{seq}_d(T)}}|\mathcal{L}^d_{(\bmu,\bsigma)}[D]|_{\Theta_d,\top}
\overset{\text{  Th. \ref{propthetatop}  }}{\sim}\hocolim_{(\bmu,\bsigma)\in P_{\mathbf{seq}_d(T)}}*=N(P_{\mathbf{seq}_d(T)})\overset{\text{  Prop. \ref{propdoperad}(2)}}{\sim }*
\end{aligned}
\end{equation}
\begin{remark}{\rm
Note that the computation \eqref{posetcontrfin} was inspired by the computation in [Be1, Lemme 1.8] and has a similar flavour.
}
\end{remark}

Now we have to compute the totalization $\Tot_{D\in \Theta_d}|\mathbf{seq}_d(T)[D]|$ and prove that it is contractible, for any pruned $(d+1)$-level tree $T$. We know from the discussion above that, for a fixed $D\in\Theta_d$, the realization $|\mathbf{seq}_d(T)[D]|$ is a contractible topological space. 

By Proposition \ref{propbemtheta}(1), it is enough to prove
\begin{prop}\label{proptottopfinal}
The $d$-cocellular topological space $D\mapsto |\mathbf{seq}_d(T)[D]|$ is Reedy fibrant. 
\end{prop}
\begin{proof}
The proof is parallel to the proof of Proposition \ref{reedytottheta}. We consider the map $D\mapsto \mathbf{seq}_d(T)[D]$ and consider its matchning map before the realization. We consider $D$ as a $d$-level tree and distinguish the two cases: (1) $D$ has 
at least two leaves, (2) $D$ has a single leaf and thus is a truncated at level $\ell\le d$ linear $d$-tree (including the case $\ell=0$).
We describe the matching map $\mathbf{seq}_d(T)[D]\to M_D\mathbf{seq}_d(T)[-]$ in both cases. In case (1) the matching map is an isomorphism, the argument is the same as in Proposition \ref{reedytottheta}. Case (2) is slightly different from the case of $\mathcal{L}_{(\bmu,\bsigma)}[-]$, considered in Proposition \ref{reedytottheta}, as now we have a colimit of several blocks $\mathcal{L}_{(\bmu,\bsigma)}[-]$. Still the same argument works: the matching map becomes a product of maps $p_i$ ($0\le i\le \ell+1$) where $p_i$ is the identity map for $i<\ell$, $p_{\ell+1}$ is the projection to a point. The map $p_\ell$, the most non-trivial one, gives rise to a Kan fibration, the proof is similar to the one given in Proposition \ref{reedytot}(2).
\end{proof}

By \eqref{posetcontrfin} and Proposition \ref{proptottopfinal},  for each $D$ the map $|\mathbf{seq}_d(T)[D]|\to *$ is a weak equivalence of fibrant $d$-cocellular topological spaces. Then Proposition \ref{propbemtheta}(1), this degree-wise projection gives rise to a weak equivalence of the totalizations, which proves the topological part of Theorem \ref{maintheorem}. 

\subsubsection{\sc The dg condensation}\label{sectionveryfinal}

It remains to prove the contractibility of the operad $\mathbf{seq}_d$ in the dg condensation. 
The argument here is very similar to one in the proof of Theorem \ref{propthetadg}. Namely, it follows from \eqref{posetcontrfin} and Lemma \ref{lemmarealthetadgtop} that, for fixed $T$ and $D$,
$$
|\mathbf{seq}_d(T)[D]|_{\Theta_d,\dg}\simeq C_\ldot^{\mathrm{CW}}(|\mathbf{seq}_d(T)[D]|_{\Theta_d,\top},\mathbb{Z})\sim \mathbb{Z}[0]
$$

Then one shows that $\Tot_{D\in \Theta_d,\dg}(|\mathbf{seq}_d(T)[D]|_{\Theta_d,\dg})\simeq \mathbb{Z}[0]$.
In virtue of Lemma \ref{lemmatotthetaconst}, it is enough to prove that $D\mapsto |\mathbf{seq}_d(T)[D]|_{\Theta_d,\dg}$ is Reedy fibrant object in the Reedy model structure on the category of diagrams $\Theta_d\to C^\udot(\mathbb{Z})$. The argument is similar to the one in Proposition \ref{propreedyfibrantthetadg}, which completes the proof of Theorem \ref{maintheorem} for the dg condensation.

Theorem \ref{maintheorem} is proved. 

\qed

\comment
\section{\sc Generalised Deligne conjecture in dimension $d=2$}
\subsection{\sc }\label{sectionclaim}
In this concluding Section, we consider an application of the theory developed in the previous Sections to generalised Deligne conjecture in dimension 2. In [PS], we associated with two small dg bicategories $C,D$ over $\k$, two strong dg functors $F,G\colon C\to D$, and two strong dg natural transformations $\eta,\theta\colon F\Rightarrow G$, a functor $$A(C,D)(F,G)(\eta,\theta)\colon \Theta_2\to C^\udot(\k)$$ whose dg totalisation is interpreted as the complex of {\it derived modifications} $\eta\Rrightarrow \theta$ (see e.g. [PS], Section 3.1 for general definitions related to bicategories, we adopt the standard terminology; we briefly recall the definition of the complexes $A(C,D)(F,G)(\eta,\theta)$ below). Varying $C,D,F,G,\eta,\theta$, we get a 3-globular set, whose 3-arrows are enriched over complexes, called $\mathrm{Bicat}_3(\k)$. It should be considered as a 3-dimensional analogue of the 2-globular set $\mathrm{Cat}_2(\k)$ formed by small dg categories, dg functors, and {\it derived natural transformations} [T3]. Denote by $||\mathbf{seq}_2||_\dg$ the dg condensation of 2-cellular 3-operad $\mathbf{seq}_2$, it is a dg 3-operad. We know from Section \ref{sectionveryfinal} that it is contractible. 

We propose the following statement which will be proven elsewhere (though the argument is very close to what follows):

\vspace{2mm}

\noindent {\sc Claim.} The dg condensation 3-operad $||\mathbf{seq}_2||_\dg$ acts on the dg-enriched 3-globular object $\mathrm{Bicat}_3(\k)$.

\vspace{2mm}

Recall that D.Tamarkin proved in [T3] that the contractible dg 2-operad $||\mathbf{seq}_1||_\dg$ acts on the 2-globular object $\mathrm{Cat}_2(\k)$. Therefore, the claim is a proper analogue for the next dimension of the result of Tamarkin.

To simplify the presentation, here we prove only the following particular case of the claim. 
Let $C$ be a small dg bicategory over $\k$, denote by $\Hoch_2^\udot(C)$ the dg totalisation of the 2-cocellular space $A(C,C)(\Id,\Id)(id,id)|_\dg\in C^\udot(\k)$.

\begin{theorem}\label{theoremdeligne2}
The contractible dg 3-operad $||\mathbf{seq}_2||_\dg$ acts naturally on $\Hoch^\udot_2(C)$, for any small dg bicategory $C$. 
Consequently, $\Hoch^\udot_2(C)$ is an algebra over $C_\udot(E_3,\k)$. 
\end{theorem}
The second statement follows from the first, due to contractibility of $||\mathbf{seq}_2||_\dg$ proven in Section \ref{sectionveryfinal}, and from the M.Batanin symmetrisation theorem, see Appendix \ref{appendixa}. Below we prove the first statement. 
The statement of Theorem \ref{theoremdeligne2} can be considered as a proper generalisation of the Deligne conjecture for Hochschild cochains in dimension 2.

\begin{remark}{\rm
The 3-globular object $\mathrm{Bicat}_3(\k)$ and $\Hoch^\udot_2(\k)$ are defined {\it by a construction} [PS]. It would be very helpful to find a proper homotopical definition from deformation theoretical point of view, leading to these constructions.}
\end{remark}

\subsection{\sc Reminder on $A(C,D)(F,G)(\eta,\theta)$}
Here we very briefly recall the definition of the 2-cocellular complexes $A(C,D)(F,G)(\eta,\theta)$ from [PS, Sect. 4], restricting ourselves to the case we only deal here with, namely, to the case $\Hoch^\udot_2(C)$, which corresponds to the case $C=D, F=G=\Id$, $\eta=\theta=\id$. 

Firstly assume that $C$ is a small strict dg 2-category over a field $\k$ (that is, the associativity and the unit isomorphisms are the identity maps). 

One associate with $C$ the following functor $A^\Theta(C)\colon \Theta_2\to C^\udot(\k)$.

Let $T=([n]; [\ell_1],\dots, [\ell_n])$ be an object of the category $\Theta_2$, see ???. 
We call a data of $C$ of shape $T$ the following set. An element $(\underline{x}=(x_0,\dots,x_n); \underline{f}=(\underline{f}_1,\dots,\underline{f}_n)$ is a data of $C$ of shape $T$ if $x_0,\dots,x_n$ are objects of $C$, and $\underline{f}_i=(f_{i0},\dots, f_{i\ell_i})$ 
be a tuple of 1-morphisms in $C$ from $x_{i-1}$ to $x_i$. The set of all $C$-data of type $T$ is denoted by $X(C,T)$. For an element $\omega\in X(C,T)$ as above denote by 
$$
f_{\min}(\omega)=f_{n0}\circ f_{n-1,0}\circ\dots\circ f_{10},\ \ f_{\max}(\omega)=f_{n,\ell_n}\circ f_{n-1,\ell_{n-1}}\circ\dots\circ f_{1,\ell_1}
$$
be the compositions of all ``minimal'' and ``maximal'' 1-morphisms of $\omega$. 

Denote by $X_i(\omega)$ the complex over $\k$ $$X_i(\omega)=C_2(f_{i,\ell_i-1},f_{i,\ell_i})\otimes\dots\otimes C_2(f_{i0}.f_{i1})$$ where $C_2(f,g)$ denote the complex of 2-morphisms from $f$ to $g$. 

Denote by $A^\Theta(C)_T$ the following complex:
\begin{equation}
A^\Theta(C)_T=\prod_{\omega\in X(C,T)}\underline{\Hom}_\k(X_1(\omega)\otimes\dots\otimes X_n(\omega), C_2(f_{\min}(\omega),f_{\max}(\omega))
\end{equation}
It is shown in [PS], Sect. 4 that the assignment $T\rightsquigarrow A^\Theta(C)_T$ gives rise to a functor $\Theta_2\to C^\udot(\k)$, which we denote by $A^\Theta(C)$. (In fact, in loc.cit. a much more general statement is proven). 

Then the complex $\Hoch_2^\udot(C)=A(C)$ is defined as the $\Theta_2$-dg totalisation of $A^\Theta(C)$, it basically results in the product-total complex of the bicomplex whose ``horizontal'' terms $A^s(C)$ are the complexes $$A^s(C)=\bigoplus_{\substack{{T=([n];[\ell_1],\dots,[\ell_n])\in\Theta_2}\\
{\ell_1+\dots+\ell_n+n=s}}}A^\Theta(C)_T$$
and the ``horizontal'' differential is formed by alternating sum of elementary face maps in $\Theta_2$, see [PS], Sect. 2.2. 

It completes the definition for the case of a strict dg 2-category $C$. 

When $C$ is a small dg bicategory over $\k$, one has to modify the construction above. The reason is that due to non-strict associativity and non-strict units the assignment (or any its parenthesised version) $T\mapsto A^\Theta(C)_T$ does not give rise to a functor $\Theta_2\to C^\udot(\k)$, see discussion in [PS, Sect. 4.1]. To fix the problem, one considers the subcomplex of $A^\Theta(C)$, formed by cohains which respect the associativity and the unit structural isomorphisms, see [PS, Sect. 4.5]. More conceptually, one leads to this description as follows. Consider a ``thickening'' $\hat{\Theta}_2$ of the category $\Theta_2$, having the same objects by the morphisms $\hat{\Theta}_2(T_1,T_2)$ are defined as the set of strict functors of bicategories 
$$
\hat{\Theta}_2(T_1,T_2)=\Bicat(\hat{\omega}_2(T_1),\hat{\omega}_2(T_2))
$$
where for a globular diagram $T\in \Theta_2$, $\hat{\omega}_2(T)$ is the free bicategory generated by $T$, 
see [PS, Sect. 3.4], compare with \eqref{eqthetamorph}. In [PS, Sect. 3.4] we provide an explicit description of the bicategory $\hat{\omega}_2(T)$. 

Then the same strategy as in the definition of $A^\Theta(C)_T$ for a strict dg 2-category gives rise, for the case of a dg bicategory $C$, to a functor 
$$
\hat{A}^\Theta(C)\colon \hat{\Theta}_2\to C^\udot(\k)
$$
see [PS, Sect. 4.2]. After that, we take the left Kan extension along the natural projection $p\colon \hat{\Theta}_2\to\Theta_2$, and it is shown in [PS, Lemma 3.9, Prop. 3.10] that it gives rise exactly to the subcomplex [PS, Sect. 4.5]. In this case, the left Kan extension is identified with the left homotopy Kan extension. Morally, due to the coherence theorem for bicategories, the passage from $\hat{\Theta}_2$ to $\Theta_2$ does not affect the homotopy type. 

\subsection{\sc Dimension 1 case: lattice paths of complexity 2 and ``natural operations'' on Hochschild cochains}
In dimension 1, there is an isomorphism between the simplicial symmetric lattice paths of complexity 2 operad and the symmetric operad of natural operations [KS], [DTT], [BB]. It has a 2-operadic counterpart, which we use here. 

We recall first what the operad of natural transformations is. It is a simplicial operad in complexes acting naturally on the cosimplicial Hochschild complex of any dg algebra $A$. 

The most general operation on a set of Hochschild cochain of $A$ can be obtained as follows. Let $\Psi_1,\dots,\Psi_n$ be the cochains, which are considered as elements of the cosimplicial Hochschild complex, before the totalisation. The most general operation is a superposition of the operations of the following type:
\begin{itemize}
\item[(i)] plugging of a cochain to an argument of another cochain,
\item[(ii)] taking the associative product of several inputs,
\item[(iii)] plugging the unit 1 of $A$ to any of arguments.
\end{itemize}
The resulting superposition operation can be represented by labelled rooted planar tree, with vertices of two types: the white vertices $s_k$, corresponded to the cochains $\Psi_k$, $(k=1\dots n)$, and the black vertices, corresponded to the product operation (in particular, a black vertex with no input corresponds to plugging of 1, we call it a stub). There are also free inputs (leaves) labelled $j_l$, which are ordered in counter-clock-wise direction counted from the root. To each $s_k$ (or $\Psi_k$) is associated an ordinal $I_k\in \Delta$, and the set of outputs $j_\ell$ is an ordinal $[J]$. We assume that the planar trees are minimal (or reduced) (see [DTT], Sect. 3.2.1), so there is 1-to-1 correspondence between the reduced trees and the operations on cochains. 

In Figure \ref{figtt1}, two planar labelled trees representing natural operations on two cochains $\Psi_1,\Psi_2$ are shown.

\sevafigc{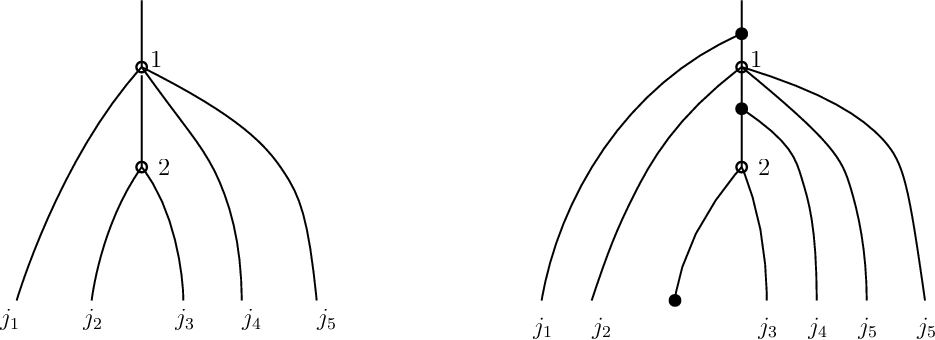}{110mm}{0}{Examples of trees representing natural operations\label{figtt1}}

The corresponding natural operations on the cochains $\Psi_1$ and $\Psi_2$ are 
$$
\Psi_1\left(a_1\otimes\Psi_2(a_2\otimes a_3)\otimes a_4\otimes a_5\right)
$$
for the left-hand side tree, and 
$$
a_1\cdot \Psi_1\left(a_2\otimes(\Psi_2(1\otimes a_3)\cdot a_4)\otimes  a_5\otimes a_6\right)
$$
for the right-hand side tree. 

In [DTT], Sect. 3.3.1, a very smart way to encode the datum of a tree as above to a combinatorial datum is suggested. In fact, loc.cit. construct out of a tree a map of 1-ordinals $[I]\to [J]$, where $[I]=\sqcup [I_k]$ equipped with maps $\phi_k\colon [I_k]\to [I]$ representing a shuffle of ordinals $[I_k]$ (that is, $\phi_k$ are injections with non-intersecting images, see Section \ref{section11appb}).

We use here an equivalent version of this correspondence the output of which is a lattice path of complexity 2, which is due to [BB, Prop. 2.14], see also a more visualised description in [BBM, Sect. 4.10]. Recall briefly the correspondence. 

We are tracing the tree from the root counter-clock-wise (thus from the left to the right), tracing each of its edge exactly twice, see Figure \ref{figtt3}.

\sevafigc{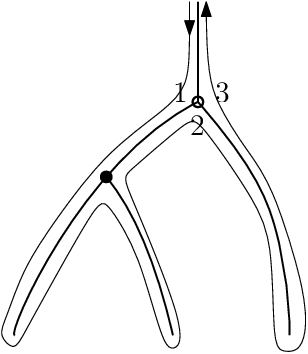}{40mm}{0}{Tracing a tree from the root counter-clock-wise \label{figtt3}}
The sectors at a white vertex with label $i$ (corresponded to the $i$-th cochain $\Psi_i$) are numbered starting from 1 to $n_{i}+1$ in the direction of tracing, if $\Psi_i$ is an $n_i$-cochain, the corresponding vertex has $n+1$ sectors. These $n_i+1$ sectors are thought of as the elementary intervals of the ordinal $[[n_i+1]]=\{0<1<\dots<n_i+1\}$ Joyal dual to $[n_i]$. (Recall that both $[n]$ and $[[n]]$ are ordinals with $n+1$ elements, but the double square brackets indicate the Joyal dual ordinal, compare \eqref{latticepath1}). Thus these sectors (or elementary intervals) are corresponded to the elementary moves of the lattice path in direction $i$. Recall that all nodes of a lattice path have labels (integral numbers $\ge 0$) indicating the number of pre-images of this node (see [BBM], Sect. 2.1). These nodes are determined by the inputs of the tree, corresponded to the arguments of the natural operation. When we are tracing an input so that the previous move is in direction $i$ (after tracing an angle of the white vertex with label $i$), we add to the label of the node +1. The black inner vertices do not change the direction, and the black inputs (the stubs) do not affect the labels of the nodes. 

The lattice paths corresponded to the corresponding left and right trees shown in Figure \ref{figtt1} are shown in Figure \ref{figtt2}.

\sevafigc{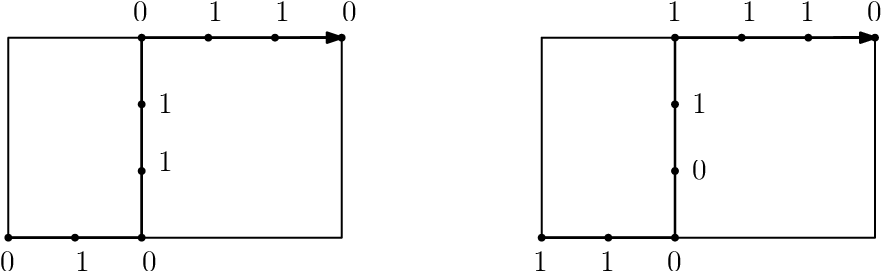}{100mm}{0}{The lattice paths corresponded to the trees of Figure \ref{figtt1} \label{figtt2}}

In this way we get an isomorphism of simplicial symmetric operad $\mathcal{L}_2^1$ of complexity 2 lattice paths in dimension $d=1$ with the simplicial symmetric operad of natural operations [DTT] (see [BB, Prop. 2.14]). Similarly, one gets an isomorphism of simplicial 2-operads $\mathcal{L}^1_{2,B}$ (see Section \ref{rembb}) and the simplicial 2-operad of natural transformations (the latter is defined as the simplicial 2-operad $\mathbf{seq}_1$ in definition of [T3], that is via shuffles). denote this isomorphism by $\Phi$. 

In what follows we use the inverse $\Phi^{-1}$ to the latter isomorphism, from $\mathcal{L}_{2,B}^1$ to the simplicial 2-operad of natural operations. 
We need to have an explicit description of $\Phi^{-1}$, which we recall below (elaborating [BB, Prop. 2.14]). 

Let $\omega\colon [[n+1]]\to [[n_1+1]]\boxx\dots \boxx[[n_k+1]]$ be a lattice path of complexity $\le 2$, and let $\Psi_i\in \Hoch^\udot(A)$ be $n_i$-cochains, $i=1\dots k$. We associate  with this data a Hochschild $n$-cochain $\Psi_\out:=\Phi(\omega, \{\Psi_i\})$. 

Consider the 1-shuffle dual to the lattice path, see Appendix B. It is a sequence of injections $q_i\colon [n_i]\to [n_1+\dots+n_k+k-1]$ 
the union of whose not-intersecting images ($i=1\dots k$) is the entire ordinal $[n_1+\dots+n_k+k-1]$. In dimension 1 we show in Appendix B that the data of $\omega$ is the same that the data of such set of injections (which we call a 1-shuffle) and a map of ordinals $w\colon [n_1+\dots+n_k+k-1]\to [n]$. In this way, the Joyal dual to the composition $w\circ q_i$ is equal to the composition $\alpha_i\colon [[n+1]]\to [[n_i+1]]$ of $\omega$ followed 
by the projection to $i$-th factor.

Consider the 1-shuffle $\{q_i\}$. 
We use it to represent the set $\{1,\dots,k\}$ as a disjoint union of linearly ordered subsets. 

Let $a$ be an elementary interval of $[n_1+\dots+n_k+k-1]$. We say that $a$ is of color $i$ if there are $s_1< s_2\in q_i([n_i])$ such that $a\subseteq [s_1,s_2]$ (by this definition, $a$ may have several colors). 
For $1\le i\le k$, let $I_i$ be the minimal {\it connected} subordinal of the ordinal of elementary intervals of $[n_1+\dots+n_k+k-1]$
which contains all elementary intervals of color $i$. One can see easily from the condition that the complexity of $\omega$ is $\le 2$ that for $i\ne j$ there are three possibilities: either $I_i\subset I_j$, or $I_j\subset I_i$, or $I_i\cap I_j=\varnothing$. We say $i<j$ if $I_i\subset I_j$. Thus, the set $\{1,2,\dots,k\}$ becomes a disjoint union of linearly ordered subsets. Note that there could be elementary intervals without any color from $\{1,\dots,k\}$, we assign to them color 0. 

Now we construct $\Psi_{\out}$ by induction, as follows. 
The simplest way is probable to construct firstly $\hat{\Psi}_\out$ which is a cochain with $n_1+\dots+n_k +k-1$ arguments, define via the shuffle $\{\phi_i\}$, and set $\Psi_\out=w_*\hat{\Psi}_\out$, where $w\colon [n_1+\dots+n_k+k-1]\to [n]$ is a map defined above. The action of $w_*$ results in taking the composition of some of groups of consecutive arguments, and plugging the identity morphisms for the other ones. It is a standard map as $Hoch^\Delta(C)$ is a functor $\Delta\to C^\udot(\k)$.

Now define $\hat{\Psi}_\out(f_1\otimes\dots\otimes f_{n_1+\dots+n_k+k-1})$. We do it by induction. 
 At the first step we consider the set $S_0\subset\{1,\dots,k\}$ which consists of elements which are minimal for the order described above. The elementary intervals from $I_i$ are in 1-to-1 correspondence with elementary intervals of $[n_i]$. For each $i\in S_0$, the corresponding substring of $f_s$ of length $n_i$ is connected. Let it be $f_{s(i)},\dots, f_{s(i)+n_i-1}$. We plug them to $\Psi_i$, taking $\Psi_i(f_s(i)\otimes \dots\otimes f_{s(i)+n_i-1})$, and taking the composition $f_{s(i)-1}\circ Psi(f_{s(i)}\otimes\dots\otimes f_{s(i)+n_i-1})\circ f_{s(i)+n_i}$, if $s(i)>1$ and $s(i)+n_i-1<n_1+\dots+n_k+k-1]$, otherwise if some of the inequalities does not hold, we do not take the composition of $\Psi_i(f_{s(i)}\otimes\dots\otimes f_{s(i)+n_i})$ with the corresponding factor $f_1$ and/or $f_{n_1+\dots+n_k+k-1}$; denote the resulting expression by $Q_i$. Let $i^\prime$ be the least element of $\{1,\dots,k\}$ greater than $i$. We replace the string  of all arguments of $Q_i$ by a single elementary interval of color $i^\prime$, and assign to it the argument $Q_i$. We perform this procedure for all minimal elements of $\{1,\dots,k\}$, then we have a smaller string of arguments and smaller number of colors. Consider the set $S_1$ of minimal elements in $\{1,\dots,k\}\setminus S_0$, apply the same procedure, and so on. At the end we get the cochain $\hat{\Psi}_\out(f_1\otimes\dots\otimes f_{n_1+\dots+n_k+k-1})$. 
It completes the description of $\Phi^{-1}$. 

To prove that it is a map of operads, we need to show that it is inverse to $\Phi$ (which is quite clear from the construction), and make use the statement that $\Phi$ is a map of operads [BB, Prop.2.14], [DTT].

Now we are ready to define an action of $\Theta_2$-colored 3-operad $\mathbf{seq}_2$ on $\Hoch^\Theta_2(C)$, for a small dg bicategory $C$.

\subsection{\sc The case of a strict dg 2-category}
We start with the case of a small strict dg 2-category $C$. 

Let $\Psi\in \Hoch^\Theta_2(C)$ be a 2-cochain of shape $T=([n]; [\ell_1],\dots, [\ell_n])\in \Theta_2$. 
We employ the following trick: we assume that $\Psi$ is equal to the horizontal product
$$
\Psi=\Psi^n\circ_h\dots\circ_h\Psi^2\circ_h\Psi^1
$$
where $\Psi^i$ is of the shape $([1]; [\ell_i])$, so it is a ``vertical'' 2-cochain 
$$
\Psi^i\in\prod_{\substack{{x_{i-1},x_i\in C}\\{f_{is}\in C_1(x_{i-1},x_i)}\\{s=0\dots \ell_i}}}\underline{\Hom}_k(C(f_{i,\ell_i-1},f_{i,\ell_i})\otimes \dots\otimes    C(f_{i0},f_{i1}), C(f_{i0},f_{i,\ell_i}))
$$
We call such 2-cochains $\Psi$ horizontally reducible. 

Below we write down a formula for the operadic action of $\mathbf{seq}_2$ on $\Hoch^\Theta_2(C)$ in assumption that all arguments $\Psi\in \Hoch^\Theta_2(C)$ are horizontally reducible, and then notice that the final formula makes sense for general $\Psi$, and can be used to define the action of $\mathbf{seq}_2$ on $\Hoch^\Theta_2(C)$.

Recall the definition of 2-lattice path operad, see Section \ref{sectiongenlp1}:
\begin{equation}\label{latticepath2bis6}
\mathcal{L}^{d}(T_1,\dots, T_k; T)=\GDisk_d(\overline{T}, \overline{T}_1\boxx \overline{T}_2\boxx\dots\boxx \overline{T}_k)
\end{equation}
Let $\Psi_1,\dots,\Psi_k\in \Hoch^\Theta_2(C)$ be 2-cochains of shapes $T_1,\dots,T_k\in\Theta_2$, correspondingly. 
We assume they are horizontally reducible, with $\{\Psi_i^j\}$ the corresponding vertical 2-cochains. 
Let $T=([n]; [\ell_1],\dots,[\ell_n])$, $T_i=([n_i]; [\ell^i_1],\dots, [\ell^i_{n_i}])$, $\omega\in \mathcal{L}^{d}(T_1,\dots, T_k; T)$. 
The elements of the ordinal $[[n+1]]$ are the elementary intervals of the ordinal $[n]$ plus two extreme elementary intervals. 
The 2-lattice path $\omega$ is given by a lattice path 
$$
\omega_1\colon [[n+1]]\to [[n_1+1]]\boxx\dots \boxx [[n_k+1]]
$$
at level 1, and the lattice paths $\omega_2^j$:
$$
\omega_2^j: \omega_2(j)\colon [[\ell_j+1]]\to [[\ell^1_{\phi_1(j)}+1]]\boxx\dots\boxx [[\ell^k_{\phi_k(j)}+1]],\ \ j\in [[n+1]], j\ne 0, n+1
$$
where $\phi_i\colon [[n+1]]\to [[n_i+1]]$ are
the compositions of $\omega_1$ with the projections

We are going to define $\Psi_\out:=\omega(\Psi_1\otimes\dots \Psi_n)$ which is a 2-cochain of shape $T$. 

We define its component $\Psi_\out(\{x_0,\dots,x_n;\ f_{ij}\in C_1(x_{i-1},x_i);\ \sigma_{ij}\in C_2(f_{ij},f_{i,j+1})\}_{i=1\dots n, j=0\dots n_i})$, which we shortly denote by $\Psi_{\out}(\underline{x},\underline{f},\underline{\sigma})$ (we ignore $\sigma_{i,n_i}$ by convention). 

We define $\Psi_\out(\underline{x},\underline{f},\underline{\sigma})$ in two steps. 

At the first step, we use the data of $\omega_1$ to assign to each element of each element of $[n_i]$ an expression in 1-morphisms $\{f_{ij}\}$. It is basically standard, the lattice path $\omega_1$ and the projections $\phi_i\colon [[n+1]]\to [[n_i+1]]$ 
define for any $a\in [[n_i+1]]$, $a\ne 0,n_i+1$
All lattice paths $\omega_2(j)$ have, by definition of $\mathbf{seq}_2$, complexities $\le 2$. 
Therefore, the map $\Phi$ from the previous Subsection is applied, and 
$$
\Phi(\omega_2(j))(\Psi_1^{\phi_1(j)}\otimes\dots \Psi_k^{\phi_k(j)})=:\Psi_{\out}^{(j)}
$$
is a Hochschild cochain on $C_1$ (here we consider the ``vertical'' Hochschild 2-cochains $\Psi_a^b$ as Hochschild 1-cochains on $C_1$. 

(Note that  $\Psi_i^{\phi_i(j)}=\Psi_i^{\phi_i(j+1)}$ if $\phi_i(j)=\phi_i(j+1)$, so the same 1-cochain $\Psi_i^s$ can be an argument of $\Phi(\omega_2(j))$ for several different $j$. When $s\in [[n_j+1]]$ is not in the image of $\phi_j$, the corresponding 1-cochain $\Psi_j^s$ is not involved in any of $\Psi_\out^{(j)}$). 

In this way we assign the Hochschild 1-cochain $\Psi_\out^{(j)}$ to each $j\in [[n+1]], j\ne 0, n+1$. 

We can now sew $\{\Psi_\out^{(j)}\}_{j=1\dots n}$ together, by means of the level 1 component $\omega_1$ of the 2-lattice path $\omega$, as follows. 

\subsection{\sc The case of a dg bicategory}

\subsection{\sc Concluding remarks}
In dimension $d=1$, it is known [BB, Prop. 2.14]  that an algebra over the complexity 2 lattice path operad $\mathcal{L}_2^1$ in cosimplicial sets is the same as a multiplicative non-symmetric operad in sets (recall that a non-symmetric operad $\mathcal{O}$ is multiplicative with an operad map from the non-symmetric terminal operad $\Assoc$ to $\mathcal{O}$). When $A$ is a monoid the operad $\underline{End}(A)$ is multiplicative; similar arguments are applied in the enriched case. In particular, the cosimplicial Hochschild cohomological object of a dg algebra $A$ (or a small dg caegory $A$) is obtained from $\underline{End}(A)$ in this way. That is, we automatically get not only a cosimplicial object but simultaneously an action of $\mathcal{L}_2^1$ on it.

In dimension $d=2$, we consider a multiplicative 2-operad $\underline{End}_2(A)$, out of a strict 2-category $A$ (here by multiplicativity we a 2-operad equipped with a map from the 2-operad $\underline{\k}$ to it). One can show that for any multiplicative 2-operad $\mathcal{O}_2$ one gets a 2-cocellular object with an action of the 2-cellular 2-operad $\mathbf{seq}_2$ on it. 

More generally, one can define a weak multiplicative 2-operad as a 2-operad $\mathcal{O}_2$ equipped with a map of a contractible 2-operad $B$ to it. On the other hand, one can define a category $\Theta_2(B)$ out of any 2-operad $B$. For example the category $\hat{\Theta}_2$ considered in [PS] is obtained from a contractible 2-operad $B$ in groupoids. Then the 2-operad $\mathcal{O}$ as above gives rise to functor from $X(\mathcal{O}_2,B)\colon \Theta_2(B)\to\mathcal{C}$ to the underlying category. We expect that one can also define a 3-operad $\mathbf{seq}_2(B)$ acting on $X(\mathcal{O}_2,B)$ in a natural way, and that $\mathbf{seq}_2(B)$ is contractible in a suitable condensation. Also one expects that such statements hold in any dimension $d$. 

From this point of view, the $\Theta_d$-colored $(d+1)$-operads $\mathbf{seq}_d$ open a way for higher-dimensional generalisations of the Deligne conjecture, in a weak $d$-categorical setting.  

\endcomment

\appendix

\section{\sc Reminder on Batanin higher operads}\label{appendixa}

\subsection{\sc Level trees and $n$-ordinals}
Recall the definition of the category $\Tree_n$ of $n$-level trees, see Definition \ref{defntree}.
An $n$-level tree is called {\it pruned} if all its leaves are at the highest level. An $n$-tree is called {\it degenerate} if the level $n$ ordinal is empty. By $|T|$ is denoted the set of leaves of an $n$-level tree $T$. 
\comment
Denote by $\{k\}$ the underlying finite set of the ordinal $[k]$. 
Denote by $\Delta_+$ the category of possibly empty ordinals $[n]$, $n\ge -1$, $[n]=\{0<1<\dots<n\}$.
The category $\Tree_n$ is defined as follows. Its object $T$ is an $n$-string of maps in $\Delta_+$:
$$
T=[k_n-1]\xrightarrow{\rho_{n-1}} [k_{n-1}-1]\xrightarrow{\rho_{n-2}}\dots\xrightarrow{\rho_0}[0]
$$
Such $T$ is visualised as a $n$-level tree. In general, it may be {\it degenerate} if $k_i=0$ for $i\ge a$, $a\le n$. An $n$-tree is called {\it pruned} if all ordinals are non-empty (all $k_i\ne 0$) and if all maps $\rho_i$ are surjective. For a pruned $n$-tree, all its leaves are at the highest level $n$. The finite set of leaves of an $n$-tree $T$ is denoted by $|T|$.
The maps $\rho_i$ are referred to as the {\it structure maps} of an $n$-tree.

A morphism $F\colon T\to S$, where 
$$
S=[\ell_n-1]\xrightarrow{\xi_{n-1}} [\ell_{n-1}-1]\xrightarrow{\xi_{n-2}}\dots\xrightarrow{\xi_0}[0]
$$
is defined as a sequence of maps $f_i\colon \{k_i-1\}\to \{\ell_i-1\}$, $i=0,1,\dots,n$ (not monotonous, in general), which commute with the structure maps, and such that  
for each $0\le i\le n$ and each $j\in [k_{i-1}-1]$ the restriction of $f_{i}$ on $\rho_{i-1}^{-1}(j)$ is monotonous. That is, $f_i$ has to be monotonous when restricted to the fibers of the structure map $\rho_{i-1}$. 
It is clear that a map of $n$-level trees is uniquely defined by the map $f_n$. Conversely, any map $f_n$ which is a map of {\it $n$-ordered sets}, associated with $n$-trees $S$ and $T$, defines a map of $n$-trees (see [Ba3, Lemma 2.3]). 
\endcomment

Let $F\colon T\to S$ be a map of $n$-level trees, with components $f_k$ (see Definition \ref{defntree}), $a\in |S|$.
The fiber $F^{-1}(a)$ for a morphism $F\colon T\to S$, $a\in |S|$, is defined as the set-theoretical preimage of the linear subtree $\Out(a)$ of $S$ spanned by $a$. This linear subtree $\Out(a)$ 
is formed by all (uniqele defined) descendants of $a$ at the lower levels than the level of $a$. It is an $n$-level tree, a subtree of $T$, possibly degenerate. Note that the fiber of a leaf of a pruned $n$-tree, for a map of pruned $n$-trees, is not necessarily pruned, even if all components $\{f_i\}$ of the map $F$ are surjective, see Remark \ref{remprune}.

\begin{example}{\rm Consider the case $n=2$, $T=[3]\xrightarrow{\rho_1} [1] \to [0]$, $S=[1]\to [0]\to [0]$. Denote by $0<1$ the leaves of $S$. Define maps $F_1, F_2\colon T\to S$ as follows: $F_1(0)=F_1(1)=0, F_1(2)=F_1(3)=1$, and $F_2(0)=F_2(2)=0, F_2(1)=F_2(3)=1$. Both $F_1,F_2$ are maps of level trees. Note that the map $F_2$ is not defined via an ordinal map $f\colon [3]\to [1]$, as $f(1)>f(2)$. At the same time, the restriction of $f$ on each fiber $f^{-1}(i)$, $i=0,1$, is a map of ordinals. }
\end{example}

\begin{remark}\label{remprune}{\rm
(1) Let $T,S$ be pruned $n$-trees, $\sigma\colon T\to S$ a map of $n$-trees. Note that the fibers $F^{-1}(a)$, $a\in |S|$, needn't be pruned $n$-trees, even if the components $f_i$ are surjective. 
For a possibly non-pruned $n$-tree $T$, denote by $P(T)$ the maximal pruned $n$-subtree of $T$. By definition, it is the pruned $n$-tree generated by all level $n$ leaves of $T$, by ignoring the leaves at levels $<n$ as well as their descendants. We call $P(T)$ the {\it prunisation} of $T$.


(2) Recall that an $n$-ordinal structure on a set $X$ is given by $n$ complementary orders on $X$, denoted by  $<_0,\dots,<_{n-1}$ (the complementarity means that for any two elements $a,b\in X$ there exists {\it a unique} $i$ from 0 to $n-1$ such that either $a<_i b$ or $b<_i a$), such that for any three elements $a,b,c\in X$ one has
\begin{equation}\label{nordinals}
a<_i b \text{  and  }  b<_j c\Rightarrow a<_{\min(i,j)}c
\end{equation}
A map of $n$-ordinals $\phi\colon X\to Y$ is a map of the underlying sets such that 
\begin{equation}\label{nordinals2}
a<_i b\Rightarrow\ 
\phi(a)<_j \phi(b)\text{  for  }j\ge i\text{  or  }
\phi(b)>_j\phi(a)\text{  for  } j>i
\end{equation}

(3) The set of leaves $|T|$ of a pruned $n$-tree is an $n$-ordinal, in the sense of [Ba2] Def. 2.2. Indeed, for two leaves $a,b\in |T|$, $a\ne b$,  we say $a<_i b$, $0\le i\le n-1$, if $i$ is the maximal level at which $\Out(a)$ and $\Out(b)$ meet (recall the the leaves of $T$ are at level $n$ and the root is at level 0). One checks that it makes the set $|T|$ an $n$-ordinal. Vice versa, an $n$-ordinal structure on a finite set $X$ gives rise to a pruned $n$-tree $T_X$ with $|T_X|=X$, such that the $n$-ordinal structure on $X$ coming from the pruned $n$-tree $T_X$ coincides with the original one. Moreover, a map of pruned trees is the same that the map of corresponding $n$-ordinals, in the sense of \eqref{nordinals}. The reader is referred to [Ba3, Th. 2.1] for proofs. 

(4) As categories, the pruned $n$-trees and $n$-ordinals $\mathbf{Ord}_n$  are isomorphic. On the other hand, as {\it operadic categories}
$\Tree_n$ and $\Ord_n$ are different: for a morphism $\sigma \colon T\to S$ of pruned $n$-trees, a fiber $\sigma_{\Ord_n}^{-1}(i)$ in $\Ord_n$ is defined as the prunisation $P(\sigma_{\Tree_n}^{-1}(i))$. 
}
\end{remark}

The linear pruned $n$-level tree $U_n$ (having a single element at each level) is the final object in both categories $\Tree_n$ and $\Ord_n$. 

\subsection{}
We recall here the definition of a {\it pruned} {\it reduced} $n$-operad. In terminology of [Ba3], the operads we consider here are all {\it $(n-1)$-terminal} $n$-operads, for some $n$. The $(n-1)$-terminality makes us possible to restrict ourselves with $n$-operads taking values in a symmetric monoidal globular category $\Sigma^nV$, where $V$ is a closed symmetric monoidal category, see [Ba2], Sect. 5. By a slight abuse of terminology, we say that an operad takes values in the closed symmetric monoidal category $V$ (not indicating $\Sigma^nV$).

\begin{defn}\label{prunedop}{\rm

A pruned reduced $(n-1)$-terminal $n$-operad $\mathcal{O}$ in a symmetric monoidal category $V$ is given by an assignment $T\rightsquigarrow \mathcal{O}(T)\in V$, for a pruned $n$-tree $T$, so that for any {\it surjective} map $\sigma\colon T\to S$ of pruned $n$-trees, one is given the composition
\begin{equation}\label{opcompgen}
m_\sigma\colon \mathcal{O}(S)\otimes \mathcal{O}(P(\sigma^{-1}(1)))\otimes\dots\otimes \mathcal{O}(P(\sigma^{-1}(k)))\to \mathcal{O}(T)
\end{equation}
where $k=|S|$ is the number of leaves of $S$, and $P(-)$ is the prunisation (which cuts all non-pruned branches, see Remark \ref{remprune}). It is subject to the following conditions (in which we assume that $V=C^\udot(\k)$ is the category of complexes of $\k$-vector spaces):
\begin{itemize}
\item[(i)] $\mathcal{O}(U_n)=\k$, and $1\in \k$ is the operadic unit,
\item[(ii)] the associativity for the composition of two surjective morphisms
$T\xrightarrow{\sigma}S\xrightarrow{\rho} Q$ of pruned $n$-trees, see [Ba2] Def. 5.1,
\item[(iii)] the two unit axioms, see [Ba2], Def. 5.1.
\end{itemize}

The category of pruned reduced $(n-1)$-terminal $n$-operads in a symmetric monoidal category $V$ is denoted by $\Op_n(V)$, or simply by $\Op_n$. 
}
\end{defn}

\begin{remark}{\rm
The idea behind the definition of pruned reduced operad is that algebras over such operads should be {\it strictly unital}. The fact that we can cut off all not pruned branches means that these redundant pieces act by (whiskering with) the identity morphism. When we deal with algebras with weak units, we have to consider more general $n$-operads. 
}
\end{remark}

\subsection{\sc Batanin Theorem}
Denote the category of symmetric operads (in a given symmetric monoidal category) by $\Op_\Sigma$.

Batanin [Ba2], Sect. 6 and 8, constructs a pair of functors relating symmetric operads and $n$-operads:

$$
\Symm\colon \Op_n^{n-1}\rightleftarrows   \Op_\Sigma \colon \Des
$$
The right adjoint functor of desymmetrisation $\Des$ associates to each pruned $n$-tree $T$ its set of leaves $|T|$ (which are all at the level $n$):
$$
\Des(\mathcal{O})(T)=\mathcal{O}(|T|)
$$
and for a map $\sigma\colon T\to S$ of $n$-trees, the $n$-operadic composition associated with $\sigma$ is defined as the corresponding composition for $|\sigma|=|\sigma_n|\colon |T|\to |S|$, twisted by the shuffle permutation $\pi(\sigma_n)$ of the map $|\sigma_n|\colon |T|\to |S|$ defined by the condition that the composition of $\pi(\sigma)$ followed by an order preserving map of finite sets is $\sigma_n$ (see [Ba2], Sect. 6). 

The symmetrisation functor $\Symm$  is defined as the left adjoint to $\Des$, its existence is established in [Ba2], Sect. 8.

The main theorem on $(n-1)$-terminal reduced $n$-operads was proven in [Ba3] Th.8.6 for topological spaces and in [Ba3] Th.8.7 for complexes of vector spaces. We provide below the statement for $C^\udot(\k)$, where $\k$ is any commutative ring. 
Denote by $\underline{\k}$ the constant $n$-operad, $\underline{\k}(T)=\k$, with evident operadic compositions. We say that an $n$-operad in $C^\udot(\k)$ is {\it augmented} by $\underline{\k}$ if there is a map of $n$-operads $p\colon \mathcal{O}\to\underline{\k}$, called the augmentation map. 

\begin{theorem}\label{theoremm}{\rm [Batanin]}
Let $\mathcal{O}$ be reduced pruned $(n-1)$-terminal $n$ operad in the symmetric monoidal category $C^\udot(\k)$. 
Assume $\mathcal{O}$ is augmented to the constant $n$-operad $\underline{\k}$, and that for any arity $T$ the augmentation map $p(T)\colon \mathcal{O}(T)\to \k$ is a quasi-isomorphis of complexes. 
Then there is a morphism of $\Sigma$-operads $C_\udot(E_n;\k)\to\Sym(\mathcal{O})$, thus making any $\mathcal{O}$-algebra a $C_\udot(E_n;\k)$-algebra.
\end{theorem}

\begin{remark}{\rm
There are closed model structures on the categories of $\Sigma$-operads and $n$-operads, constructed in [BB2]. Within these model structures, $(\Symm, \Des)$ is a Quillen pair, with $\Symm$ the left adjoint. The stronger version of this theorem [Ba3] actually says that the symmetrisation of a {\it cofibrant} contractible pruned, reduced, $(n-1)$-terminal is weakly equivalent to the symmetric operad $C^\udot(E_n;\k)$. 
}
\end{remark}

An advantage of the approach of Theorem \ref{theoremm} to $n$-algebras via contractible $n$-operads is that the latter is much simpler and more ``linear'' object than the symmetric operads $E_n$ and $e_n$. At the same time, it links higher category theory and $E_n$-algebras in a very explicit way.

\bigskip

\noindent{\small
 {\sc Euler International Mathematical Institute\\
10 Pesochnaya Embankment, St. Petersburg, 197376 Russia }}

\bigskip

\noindent{{\it e-mail}: {\tt shoikhet@pdmi.ras.ru}}

\end{document}